\documentclass[11pt]{amsart}
\usepackage{times}

\usepackage{bbm}

\usepackage[hang]{footmisc} 

\usepackage{varioref}
\usepackage[bookmarks=false]{hyperref}
\usepackage{cleveref}

\usepackage{amsthm, amssymb, amsmath, a4wide}
\usepackage[inline]{enumitem}
\usepackage{graphicx} 

\usepackage{tikz}
\usetikzlibrary{calc}
\usepackage{subfig}
\usepackage{pgfplots}

\usepackage{footnote} 

\usepackage{bm} 
\usepackage{cancel} 
\usepackage{dsfont} 
\usepackage{mathtools} 
\usepackage{mathrsfs}

\usepackage{amsfonts}

\usepackage{xcolor}
\hypersetup{
    colorlinks,
    linkcolor={red!50!black},
    citecolor={blue!50!black},
    urlcolor={blue!80!black}
}

\author{Pinaki Mondal}
\email{pinakio@gmail.com}

\title[Breaking the symmetry in excess intersection]{
	Breaking the symmetry in excess intersection and counting solutions of systems of polynomials
}

\makeatletter

\newcommand{\Rmnum}[1]{\expandafter\@slowromancap\romannumeral #1@}
\makeatother

\DeclareMathOperator\ord{ord}

\DeclareMathOperator\sing{Sing}

\DeclareMathOperator\supp{Supp}
\DeclareMathOperator\vol{Vol}

\newcommand{\scrA}{\ensuremath{\mathcal{A}}}

\newcommand{\scrE}{\ensuremath{\mathcal{E}}}

\newcommand{\scrH}{\ensuremath{\mathcal{H}}}
\newcommand{\scrI}{\ensuremath{\mathcal{I}}}

\newcommand{\scrL}{\ensuremath{\mathcal{L}}}
\newcommand{\scrM}{\ensuremath{\mathcal{M}}}

\newcommand{\scrO}{\ensuremath{\mathcal{O}}}
\newcommand{\scrP}{\ensuremath{\mathcal{P}}}
\newcommand{\scrQ}{\ensuremath{\mathcal{Q}}}

\newcommand{\scrS}{\ensuremath{\mathcal{S}}}

\newcommand{\scrV}{\ensuremath{\mathcal{V}}}

\newcommand{\scrZ}{\ensuremath{\mathcal{Z}}}

\newcommand{\kk}{\ensuremath{\mathbb{K}}}

\newcommand{\pp}{\ensuremath{\mathbb{P}}}
\newcommand{\qq}{\ensuremath{\mathbb{Q}}}
\newcommand{\rr}{\ensuremath{\mathbb{R}}}

\newcommand{\zz}{\ensuremath{\mathbb{Z}}}

\newcommand{\ppp}{\ensuremath{\mathfrak{p}}}
\newcommand{\qqq}{\ensuremath{\mathfrak{q}}}

\newcommand{\woutlog}{without loss of generality}
\newcommand{\Woutlog}{Without loss of generality}

\newtheorem{thm}[subsection]{Theorem}
\newtheorem*{thm*}{Theorem}

\newtheorem{prothm}[subsubsection]{Theorem}
\newtheorem{claim}[subsection]{Claim}
\newtheorem*{claim*}{Claim}

\newtheorem*{conjecture*}{Conjecture}
\newtheorem{cor}[subsection]{Corollary}
\newtheorem{procor}[subsubsection]{Corollary}
\newtheorem{lemma}[subsection]{Lemma}
\newtheorem*{lemma*}{Lemma}

\newtheorem{proclaim}[subsubsection]{Claim}
\newtheorem{subproclaim}[paragraph]{Subclaim}

\newtheorem{prop}[subsection]{Proposition}
\newtheorem*{prop*}{Proposition}
\newtheorem{proprop}[subsubsection]{Proposition}
\newtheorem{prolemma}[subsubsection]{Lemma}

\theoremstyle{definition}

\newtheorem{bold-question}[subsection]{Question}
\newtheorem*{bold-question*}{Question}

\newtheorem*{constrinition*}{Construction-Definition}

\newtheorem*{convention*}{Convention}
\newtheorem{defn}[subsection]{Definition}
\newtheorem*{defn*}{Definition}

\newtheorem*{definotation*}{Definition-Notation}
\newtheorem{example}[subsection]{Example}
\newtheorem*{example*}{Example}

\newtheorem*{fact*}{Fact}

\newtheorem*{facts*}{Facts}

\newtheorem*{bold-note*}{Note}
\newtheorem{problem}[subsection]{Problem}
\newtheorem*{problem*}{Problem}
\newtheorem{rem}[subsection]{Remark}

\newtheorem{remexample*}{Remark-Example}

\newtheorem*{reminition*}{Remark-Definition}

\newtheorem*{remtation*}{Remark-Notation}

\newtheorem*{remuestion*}{Remark-Question}

\newtheorem*{remvention*}{Remark-Convention}
\newtheorem{prorem}[subsubsection]{Remark}

\newtheorem{proexample}[subsubsection]{Example}

\theoremstyle{remark}
\newtheorem*{note*}{Note}
\newtheorem*{notation*}{Notation}
\newtheorem*{question*}{Question}

\newtheorem*{questions*}{Questions}
\newtheorem*{rem*}{Remark} 
\DeclareMathOperator\Center{center}
\DeclareMathOperator\character{char}
\DeclareMathOperator{\codim}{codim}
\newcommand{\dist}[2]{\Phi^{(#2)}_{#1}}
\newcommand{\distprime}[2]{\Phi'^{(#2)}_{#1}}
\newcommand{\hatlocal}[2]{\hat{\mathcal{O}}_{#2,#1}}  
\DeclareMathOperator\In{In}
\renewcommand{\kk}{\mathbbm{k}}
\DeclareMathOperator{\ld}{Ld}
\DeclareMathOperator\len{len}
\newcommand{\local}[2]{\mathcal{O}_{#2, #1}}  

\newcommand{\mult}[2]{\multonly{#1, \ldots, #2}}
\newcommand{\multsub}[3]{\mult{#1}{#2}_{#3}}
\newcommand{\multsubsup}[4]{\mult{#1}{#2}_{#3}^{#4}}

\newcommand{\multdeform}[4]{\multsubsup{#1}{#2}{#3}{#4}}
\newcommand{\multfzero}{\multzero{f_1}{f_n}}
\newcommand{\multGammazero}{\multzero{\Gamma_1}{\Gamma_n}}

\newcommand{\multiso}[3]{\multsubsup{#1}{#2}{#3}{iso}}
\newcommand{\multisoonly}[2]{\multonlysubsup{#1}{#2}{iso}}

\newcommand{\multonly}[1]{(#1)}
\newcommand{\multonlysub}[2]{\multonly{#1}_{#2}}
\newcommand{\multonlysubsup}[3]{\multonly{#1}_{#2}^{#3}}

\newcommand{\multord}[3]{\multsubsup{#1}{#2}{#3}{\ord}}
\newcommand{\multordonly}[2]{\multonlysubsup{#1}{#2}{\ord}}

\newcommand{\multzero}[2]{\multsub{#1}{#2}{\origin}}
\newcommand{\multzeroonly}[1]{\multonlysub{#1}{\origin}}
\newcommand{\multzerostar}[2]{\multsubsup{#1}{#2}{\origin}{*}}
\newcommand{\multzerostaronly}[1]{\multonlysubsup{#1}{\origin}{*}}

\DeclareMathOperator\mv{MV}
\DeclareMathOperator\nd{ND}

\newcommand{\origin}{0}
\def\picfontsize{\scriptsize}

\newcommand{\rnonneg}{\rr_{\geq 0}}
\newcommand{\rnonnegg}[1]{(\rnonneg)^{#1}}
\newcommand{\rnnuperp}{\rr^n_{\nu^\perp}}
\newcommand{\rnstar}{(\rr^n)^*}

\newcommand{\tGammaOne}{\mscrT_{\Gamma,1}}
\newcommand{\ti}{T^I}
\newcommand{\tiGamma}{\ti_{\Gamma}}

\newcommand{\znonneg}{\zz_{\geq 0}}
\newcommand{\znonnegg}[1]{(\znonneg)^{#1}}

\newlist{defnlist}{enumerate}{3}
\setlist[defnlist,1]{label=(\alph*)}
\setlist[defnlist,2]{label=(\arabic*), ref=(\alph{defnlisti}.\arabic*)}
\setlist[defnlist,3]{label=(\roman*), ref=(\alph{defnlisti}.\arabic{defnlistii}.\roman*)}

\newlist{prooflist}{enumerate}{3}
\setlist[prooflist,1]{label=(\roman*)}
\setlist[prooflist,2]{label=(\alph*), ref=(\roman{prooflisti}.\alph*)}
\setlist[prooflist,3]{label=(\arabic*), ref=(\roman{prooflisti}.\alph{prooflistii}.\arabic*)}

\newlist{prooflist'}{enumerate}{3}
\setlist[prooflist',1]{label=(\roman*$'$), ref=(\roman*$'$)}
\setlist[prooflist',2]{label=(\arabic*$'$), ref=(\roman{prooflist'i}.\arabic*$'$)}
\setlist[prooflist',3]{label=(\alph*$'$), ref=(\roman{prooflist'i}.\arabic{prooflist'ii}.\alph*$'$)}

\newlist{prooflist''}{enumerate}{3}
\setlist[prooflist'',1]{label=(\roman*$''$), ref=(\roman*$''$)}
\setlist[prooflist'',2]{label=(\arabic*$''$), ref=(\roman{prooflist''i}.\arabic*$''$)}
\setlist[prooflist'',3]{label=(\alph*$''$), ref=(\roman{prooflist''i}.\arabic{prooflist'ii}.\alph*$''$)}

\newcommand{\mscrI}{\mathscr{I}}

\newcommand{\mscrT}{\mathscr{T}}

\crefname{claim}{claim}{claims}
\crefname{cor}{corollary}{corollaries}
\crefname{defn}{definition}{definitions}
\crefname{paragraph}{paragraph}{paragraphs}
\crefname{problem}{problem}{problems}
\crefname{proclaim}{claim}{claims}
\crefname{proexample}{Example}{examples}
\crefname{prop}{proposition}{propositions}
\crefname{subclaim}{subclaim}{subclaims}
\crefname{subsection}{section}{sections}
\crefname{subsubsection}{paragraph}{paragraphs}
\crefname{thm}{theorem}{theorems}
\crefname{prothm}{theorem}{theorems}

\let\theoldsubsection\thesubsection
\renewcommand{\thesubsection}{\bf \theoldsubsection}

\begin{document}

\begin{abstract}
We revisit the fundamental problem of assigning intersection multiplicities to subsets of solutions of (square) systems of polynomials. Severi [Ann.\ Mat.\ Pura Appl.\ 26 (4), 1947] suggested an intuitive {\em dynamic} solution to this problem which was later corrected and made rigorous by Lazarsfeld [Compos.\ Math.\ 43, 1981]. We consider an {\em asymmetric} variant of this approach and find an explicit description of the resulting ``ordered intersection multiplicity'' which opens pathways to step by step solutions to the {\em affine B\'ezout problem} of counting isolated solutions to (square) systems of polynomials via ``Bernstein-Kushnirenko type'' estimates in terms of Newton diagrams. To illustrate our methods we compute the number of common tangent lines to $4$ general spheres in the affine $3$-space (which is known to be 12 due to Macdonald, Pach, and Theobald [Discrete Comput.\ Geom.\ 26 (1), 2001]) via certain ordered intersection multiplicities on the corresponding Grassmannian.
\end{abstract}

\maketitle

\tableofcontents

\section{Introduction} \label{sec:intro}
This work is motivated by the {\em affine B\'ezout problem} of counting the precise number of isolated solutions of a system $f_1, \ldots, f_n$ of $n$ polynomials in $n$ variables over an algebraically closed field $\kk$. A general approach to this problem is to consider appropriate vector spaces of polynomials containing $f_i$, which corresponds to embedding the $f_i$ in families of base-point free linear systems $\scrL_i$ (on a ``natural'' ambient space). In the generic case the given system is ``non-degenerate'', and the number of its solutions equals the number of common solutions of generic divisors $G_i \in \scrL_i$, which is calculated by a version of B\'ezout's theorem or Bernstein-Kushnirenko theorem. However, if the system is ``degenerate'', then this calculated bound is not sharp, and the intersection of the (effective) divisors $F_1, \ldots, F_n$ corresponding to the given polynomials contains ``problematic'' points (which could be ``points at infinity'' or non-isolated components). This leads one to the following general problem:

\begin{problem} \label{problem0}
Given a subset $Z$ of $\bigcap_i \supp(F_i)$, count (with appropriate multiplicity) how many of the points in $\bigcap_i \supp(G_i)$ ``move to'' $Z$ as the $G_i$ ``move to'' $F_i$ (where $G_i$ are generic elements of $\scrL_i$).
\end{problem}

\subsection{}
In this paper we approach \Cref{problem0} by considering deformations of $\bigcap_i F_i$ defined locally by $f_1 - t^{r_1}g_1 = \cdots = f_n -t^{r_n}g_n = 0$, where the $f_i, g_i$ are local representatives respectively of $F_i, G_i$, and computing the {\em ordered intersection multiplicity} of $Z$, which is the number of points  which ``move to'' generic points of $Z$ as $t$ moves to $0$ in the case that each $r_i/r_{i+1} \gg 1$.
%
%
%
%
The point is that the ordered intersection multiplicity admits a fairly explicit and elementary description in terms of sums of products of various ``degrees'' of $Z$ (the ``global'' factors) and lengths of certain ideals in the local ring of $Z$ (the ``local'' factors). In turn this opens a possibility for a step by step solution to the affine B\'ezout problem via estimates in terms of associated {\em Newton diagrams}.

\subsection{} In \Cref{sec:overview} below we provide a detailed introduction to ordered intersection multiplicity and the main results of this article. The subsequent sections are divided into two largely independent parts: in the first part (comprising of \Crefrange{sec:prelim}{sec:formula}) we show that the ordered intersection multiplicity is well defined and find an explicit formula for it; the second part (spanning \Crefrange{sec:local}{sec:ind}) provides various expressions of the ordered intersection multiplicity in terms of usual intersection multiplicities of divisors, and resulting estimates in terms of Newton diagrams. Note that for applications to the affine B\'ezout problem one does not need the dynamic interpretation of ordered intersection multiplicity - it suffices to read only the proof of \Cref{thm:comb:intersection} in \Cref{sec:combinatorics}, and then \Crefrange{sec:local}{sec:ind} in the second part.


\subsection{The big picture in the context of the affine B\'ezout problem} \label{sec:big}
A problem associated to, and deeply intertwined with, the affine B\'ezout problem is that of computing the multiplicity of a (proper) isolated intersection at a nonsingular point. In the ``non-degenerate'' case this multiplicity can be estimated in terms of associated Newton diagrams, and in the ``degenerate'' case this multiplicity is equal to the non-degenerate multiplicity plus the multiplicity of certain excess intersections on a blow up at the point (see e.g. \cite[Chapter IX]{howmanyzeroes}).


\subsubsection{} \label{sec:big:asymp}
In turn, the ``ordered intersection multiplicities'' can be expressed in terms of the (usual) intersection multiplicities of divisors at isolated points of intersections on a birational modification of the given variety (this modification can be described in terms of the given $F_i$ and ``generic'' $G_i$ from \Cref{problem0} - see \Cref{sec:ind:asymp}). In a sense this provides an inductive ``solution'' to the affine B\'ezout problem
: start with the non-degenerate bound; if there are excess intersections, then express the corresponding ordered multiplicities in terms of the usual intersection multiplicities of divisors at isolated points of intersection on an appropriate modification; then compute the ``non-degenerate portion'' of the latter multiplicities in terms of Newton diagrams and express the ``degenerate portion'' in terms of excess intersections, and then repeat. We are however not sure about the practicality or complexity of this ``solution.''

\subsubsection{} \label{sec:big:res-pos}
There are similar inductive routes to the affine B\'ezout problem when $\bigcup_i \supp(F_i)$ is the support of a strict normal crossing divisor, i.e.\ when resolution of singularities is available (e.g.\ in characteristic zero), or when certain curves are locally set-theoretic complete intersections (which is true in positive characteristics by a result of Cowsik and Nori \cite{cowsik-nori}) - see \Cref{sec:ind:pos,sec:ind:res}. 

\subsubsection{}
In general it is unknown if every curve is a local set-theoretic complete intersection. Short of the birational ``modifications'' mentioned in \Cref{sec:big:asymp,sec:big:res-pos}, we do not know of any general method to obtain combinatorial or convex geometric estimates of the ordered intersection multiplicity if the associated curves are not set-theoretic complete intersection near the corresponding ``centers''. Sometimes it is possible to give such an estimate when a ``punctured'' germ of the curve is a complete intersection in the complement of a union of ``coordinate subspaces'' (see \Cref{sec:newton0}); in order to extend this approach to a general method one would need to get a handle on the {\em tropicalization} of a curve $C$ near a point $P$, which seems to be a hard problem if $C$ is not a (set-theoretic) complete intersection near $P$.

\section{An overview of ordered intersection multiplicity and the main results} \label{sec:overview}
\subsection{Ordered intersection multiplicity} \label{intro:mult0}
Consider the set up of \Cref{problem0}. More precisely, let $X$ be a 
variety of dimension $n$, and $F_i$ be an effective Cartier divisor on $X$ which is an element of a base-point free linear system $\scrL_i$, $i = 1, \ldots, n$.
For each choice of $G_i \in \scrL_i$, $i = 1, \ldots, n$, we denote by $C^* = C^*(G_1, \ldots, G_n; r_1, \ldots, r_n)$ the subscheme of $X \times (\kk \setminus \{0\})$ which is defined locally as $V(f_1 - t^{r_1}g_1, \ldots, f_n - t^{r_n}g_n)$, where $f_i,g_i$ are local representatives of respectively $F_i,G_i$. If the $G_i$ are generic, then for almost all $a \in \kk$, the number of points on $C^*$ at $t = a$, counted with appropriate multiplicities, is precisely the intersection number $\mult{\scrL_1}{\scrL_n}$. 
Consequently, $C^*$ can be viewed as a moving collection of $\mult{\scrL_1}{\scrL_n}$ points, and it seems reasonable, as Severi \cite{severi} suggested, to count how many of these points approach ``generic'' points of a given subvariety $Z \subseteq X$, and take it as a measure of the intersection multiplicity of $F_1, \ldots, F_n$ along $Z$; we informally denote this number as $\multdeform{F_1}{F_n}{Z}{\vec r}$ (where $\vec r := (r_1, \ldots, r_n)$). In the case that $X$ is nonsingular, Lazarsfeld \cite{lazarsfeld-excess} rigorously developed this notion for $\vec r = (1, \ldots, 1)$ following (and correcting) the idea of Severi \cite{severi}\footnote{Severi \cite{severi} suggested taking the {\em minimum} possible number of points approaching generic points of $Z$. Lazarsfeld \cite{lazarsfeld-excess} showed that one should instead take the number of points in a {\em generic} such deformation.}. The point of departure for this article is the observation that $\multdeform{F_1}{F_n}{Z}{\vec r}$ admits an explicit description in the extreme limit that $r_1 \gg r_2 \gg \cdots \gg r_n > 0$; motivated by this somewhat of a miracle we consider the {\em ordered intersection multiplicity}
\begin{align*}
\multord{F_1}{F_n}{Z} &:= \lim_{\min\{\frac{r_1}{r_2}, \ldots, \frac{r_{n-1}}{r_n}\} \to \infty} \multdeform{F_1}{F_n}{Z}{\vec r}
\end{align*}
We show that $\multord{F_1}{F_n}{Z}$ is well defined and provide explicit descriptions and estimates of this number. Unlike the symmetric version $\multdeform{F_1}{F_n}{Z}{(1, \ldots, 1)}$ considered by Severi and Lazarsfeld, the ordered intersection multiplicity in general depends on the ordering of $F_1, \ldots, F_n$, and corresponds to a ``step-by-step expansion of the intersection product $\mult{F_1}{F_n}$ from left to right'' (\Cref{thm:comb:intersection}). However, in the case that $Z$ is an {\em isolated} point of $V := \bigcap_i \supp(F_i)$, the ordered intersection multiplicity is independent of the ordering, and equals the usual intersection multiplicity of the $F_i$ at $Z$. It follows that
\begin{align}
\multiso{F_1}{F_n}{X} &= \mult{\scrL_1}{\scrL_n} - \sum_{Z \subseteq V_+} \multord{F_1}{F_n}{Z}
\label{eq:=multiso}
\end{align}
where $\multiso{F_1}{F_n}{X}$ is the sum of intersection multiplicities of $F_1, \ldots, F_n$ over all isolated points of $V$, and $V_+$ is the union of the positive dimensional irreducible components of $V$. Identity \eqref{eq:=multiso} opens the passage to applications of this ordered intersection multiplicity to the affine B\'ezout problem.

\begin{example} \label{ex:L2:11}
We start with a classical example, considered e.g.\ in \cite{lazarsfeld-excess}. Let $F_1 = V(x_0^2x_1)$, $F_2 = V(x_0x_1^2) \subseteq \pp^2$, and take $\scrL_1 = \scrL_2 = \scrO(3)$, the linear system on $\pp^2$ cut out by the cubics. Take $U := \pp^2 \setminus V(x_2)$, and affine coordinates $(u_0, u_1) := (x_0/x_2, x_1/x_2)$ on $U$. Then $F_i = V(f_i)$ on $U$, where $f_1 := u_0^2u_1$ and $f_2 := u_0u_1^2$. Let $\phi_i := f_i - t^{r_i}g_i$, where $g_1, g_2$ are cubic polynomials in $(u_0, u_1)$. First we consider the case that $r_1 = r_2 = 1$. Then
\begin{align*}
u_1\phi_1 - u_0\phi_2 = t(u_0g_2 - u_1g_1)
\end{align*}
Write $C$ for the (scheme-theoretic) closure of $C^*$. If $g_1, g_2$ are generic, then $C = V(\phi_1, \phi_2)$ is defined near $U \times \{0\}$ by
\begin{align*}
f_1 - tg_1 &= u_0^2u_1 - ta_1 + \cdots,\ \text{and} \\
u_0g_2 - u_1g_1 &= u_0a_2 - u_1a_1 + \cdots
\end{align*}
where the $a_i$ are the constant terms of $g_i$ (and ``$\cdots$'' denotes terms with higher order in $t$). In the generic case both $a_i$ are nonzero, which implies that $C$ is nonsingular at the origin $O := (0,0,0)$, and both $u_0$ and $u_1$ are parameters of $C$ at $O$. It follows that
\begin{align*}
\multonlysubsup{F_1,F_2}{V(x_0,x_1)}{(1,1)} = \ord_{(0,0,0)}(t|_C) = \ord_{(0,0,0)}(u_0^2u_1|_C) = 3
\end{align*}
Now we consider $Z = V(x_0)$. Then $Z \cap U = V(u_0)$, and the intersection of $C \cap \{t = 0\}$ and $(Z \cap U) \times \{0\}$ is $V(f_1 - tg_1, u_0g_2 - u_1g_1, t, u_0) = V(f_1, u_1g_1, t, u_0)$, which set theoretically consists of the origin $O$ and $V(g_1, u_0) \times \{0\}$. In the generic case the latter set consists of three generic points on $Z \times \{0\}$ distinct from $O$, and it is not hard to see that at each of these points $C$ is nonsingular and $t$ is a parameter. Consequently,
\begin{align*}
\multonlysubsup{F_1,F_2}{V(x_0)}{(1,1)}
    &= 3,\ \text{and similarly,} \\
\multonlysubsup{F_1,F_2}{V(x_1)}{(1,1)}
    &= 3
\end{align*}
Since $\multonly{F_1, F_2} = 9$, it follows that
\begin{align*}
\multonly{F_1, F_2} &= \multonlysubsup{F_1,F_2}{V(x_0,x_1)}{(1,1)} + \multonlysubsup{F_1,F_2}{V(x_0)}{(1,1)}  + \multonlysubsup{F_1,F_2}{V(x_1)}{(1,1)}
\end{align*}
is a complete decomposition of $(F_1, F_2)$, and $\multonlysubsup{F_1,F_2}{Z}{(1,1)} = 0$ for all $Z$ different from $V(x_0), V(x_1)$ or $V(x_0,x_1)$.
\end{example}

\begin{example} \label{ex:L2:1>2}
We continue with the same set up as \Cref{ex:L2:11}. If $r_1 \gg r_2$, then we claim that $g_2(z) = 0$ for each $z \in (C \cap \{t=0\}) \cap (U \times \{0\})$. Indeed, $f_2 = u_0u_1^2 \in \sqrt{u_0^2u_1\local{C}{z}} = \sqrt{f_1\local{C}{z}}$. Consequently, there is $k \geq 1$ such that $f_2^k = f_1h$ for some $h \in \local{C}{z}$ (it is clear that $k = 2$ works for this example), and
\begin{align*}
t^{kr_2}(g_2)^k = (f_2)^k = f_1h = ht^{r_1}g_1 \in \local{C}{z}
\end{align*}
Consequently, if $r_1 > kr_2$, then $(g_2)^k$ vanishes at $z$, as claimed. Since a generic $g_2$ does {\em not} vanish at the origin, it follows that
\begin{align*}
\multonlysubsup{F_1,F_2}{V(x_0,x_1)}{\vec r} = 0
\end{align*}
Since a generic $g_2$ vanishes at $3$ generic points on $V(x_0) \cap U = V(u_0)$, and since $\ord_{u_0}(f_1) = 2$, it can be shown that (see e.g.\ assertion \eqref{local=mult:1:1} of \Cref{prop:local=mult:1})
\begin{align*}
\multonlysubsup{F_1,F_2}{V(x_0)}{\vec r} &= 2 \times 3 = 6.
\end{align*}
Similarly, since $\ord_{u_1}(f_1) = 1$, we have
\begin{align*}
\multonlysubsup{F_1,F_2}{V(x_1)}{\vec r} &= 1 \times 3 = 3.
\end{align*}
Since $\multonly{F_1, F_2} = 9 = 6 + 3$, it follows that if $r_1 \gg r_2$, then
\begin{align*}
\multonlysubsup{F_1,F_2}{Z}{\vec r} &=
\begin{cases}
  6 & \text{if}\ Z = V(x_0), \\
  3 & \text{if}\ Z = V(x_1), \\
  0 & \text{otherwise.}
\end{cases}
\end{align*}
If $r_2 \gg r_1$, then it similarly follows that
\begin{align*}
\multonlysubsup{F_1,F_2}{Z}{\vec r} &=
\begin{cases}
  3 & \text{if}\ Z = V(x_0), \\
  6 & \text{if}\ Z = V(x_1), \\
  0 & \text{otherwise.}
\end{cases}
\end{align*}
\end{example}

\begin{example}[Cf.\ \Cref{ex:2-1}] \label{ex:L2:1'1}
Let $F_2 = V(x_0x_1^2)$ be as above and $F'_1 = V(x_0^2)$, i.e.\ $F'_1$ is a ``component'' of $F_1$ from above. The same arguments as the previous example show 
that if $r_1 \gg r_2$, then
\begin{align*}
\multonlysubsup{F'_1,F_2}{Z}{\vec r} &=
\begin{cases}
  6 & \text{if}\ Z = V(x_0), \\
  0 & \text{otherwise.}
\end{cases}
\end{align*}
whereas, if $r_2 \gg r_1$, then
\begin{align*}
\multonlysubsup{F'_1,F_2}{Z}{\vec r} &=
\begin{cases}
  2 & \text{if}\ Z = V(x_0), \\
  4 & \text{if}\ Z = V(x_0, x_1), \\
  0 & \text{otherwise.}
\end{cases}
\end{align*}
\end{example}

\subsection{}
We now describe the main results. Recall the basic set up: we are given fixed elements $F_i$ of base-point free linear systems $\scrL_i$, $i = 1, \ldots, n$, on a variety $X$ of dimension $n$. Note that we allow $F_i$ to be identically zero on $X$. For generic $G_i \in \scrL_i$, the subscheme of $X \times \kk$ locally defined by $V(f_1 - t^{r_1}g_1, \ldots, f_n - t^{r_n}g_n)$, where $f_i, g_i$ are local representatives respectively of $F_i$ and $G_i$, restricts to a (possibly non-reduced) curve $C^*(G_1, \ldots, G_n; r_1, \ldots, r_n)$ on $X \times (\kk \setminus S)$, where $S$ is a (possibly empty) finite subset of $\kk$; we write $C(G_1, \ldots, G_n; r_1, \ldots, r_n)$, or in short, $C(\vec G; \vec r)$, or simply $C$, for the scheme-theoretic closure of this curve in $X \times \kk$.

\begin{thm}[Ordered intersection multiplicity - definition] \label{thm:defn}
Let $Z$ be an irreducible subvariety of $X$. Then there is $\nu \in \rr$ and a proper closed subset $\check{Z}$ of $Z$ such that one (and only one) of the following holds:
\begin{enumerate}
\item for each $\vec{r}$ such that $r_i/r_{i+1} > \nu$ for each $i = 1, \ldots, n-1$, there is a nonempty Zariski open subset $\scrL^*$ of global sections of $\prod_i \scrL_i$ such that $C$ does not intersect $(Z \setminus \check{Z}) \times \{0\}$ for any $(G_1, \ldots, G_n) \in \scrL^*$,
\item or there is a positive integer $N_Z$ depending only on $Z$ such that for each $\vec{r}$ such that $r_i/r_{i+1} > \nu$ for each $i = 1, \ldots, n-1$, and each nonempty open subset $Z^*$ of $Z$, there is a nonempty open subset $\scrL^*$ of $\prod_i \scrL_i$ such that
    \begin{enumerate}
    \item $C \cap ((Z^* \setminus \check{Z}) \times \{0\}) \neq \emptyset$ for each $(G_1, \ldots, G_n) \in \scrL^*$, and
    \item the sum of $\ord_B(t)$ over the branches\footnote{\label{sec:branch-defn} Let $C$ be a curve and $\pi: C' \to C$ be a desingularization of $C$. A {\em branch} $B$ of $C$ is the germ of a point $z$ in $C'$. The {\em center} of $B$ is $y := \pi(z) \in C$; we also say that $B$ is a {\em branch of $C$ at $y$}. The {\em order} at $B$ of a regular function $f$ on $C$, denoted $\ord_B(f)$, is the order of vanishing of $f$ at $z$.} $B$ of $C$ centered at $(Z^* \setminus \check{Z}) \times \{0\}$ is precisely $N_Z$.
    \end{enumerate}
\end{enumerate}
We define $\multord{F_1}{F_n}{Z}$ to be zero in the first case, and equal to $N_Z$ in the second case. There are {\em finitely} many irreducible subvarieties $Z$ of $X$ such that $\multord{F_1}{F_n}{Z}$ is nonzero. If $X$ is complete, then
\begin{align*}
\mult{\scrL_1}{\scrL_n} = \sum_Z \multord{F_1}{F_n}{Z}
\end{align*}
The subset $\check{Z}$ can be described as follows: $\check{Z}$ is the intersection of $Z$ with the union of all irreducible subvarieties $Z' \neq Z$ of $X$ such that $\codim(Z') \geq \codim(Z)$ and $\multord{F_1}{F_n}{Z'} \neq 0$.
\end{thm}

The subvarieties $Z$ for which $\multord{F_1}{F_n}{Z}$ is nonzero can be explicitly described in terms of intersections of irreducible components of $\supp(F_i)$. We found this surprising since descriptions of ``distinguished varieties'' tend to require passing to a ``modification'' of the original space (normal cone in Fulton-MacPherson theory \cite{fultersection}, and a transcendental extension in St\"{u}ckrad-Vogel theory \cite{vogel}).

\begin{thm} \label{thm:Z}
Let $i^*$ be the smallest index such that $F_{i^*}$ is not identically zero. For each collection of indices $i^* = i_1 < \cdots < i_\rho \leq n$, where $1\leq \rho \leq n$, construct subvarieties $Z^{(j)}_{i_1, \ldots, i_\rho}$, $0 \leq j \leq \rho$, as follows:
\begin{defnlist}
\item $Z^{(0)}_{i_1, \ldots, i_\rho} = X$,
\item for $1 \leq j \leq \rho-1$, $Z^{(j)}_{i_1, \ldots, i_\rho}$ is the union of all irreducible components of $Z^{(j-1)}_{i_1, \ldots, i_\rho} \cap \supp(F_{i_j})$ which are not contained in $\bigcup_{j' >j} \supp(F_{i_{j'}})$, but are contained in $\supp(F_i)$ for each $i$, $i_j < i < i_{j+1}$,
\item $Z^{(\rho)}_{i_1, \ldots, i_\rho}$ is the union of all irreducible components of $Z^{(\rho - 1)}_{i_1, \ldots, i_\rho} \cap \supp(F_{i_\rho})$ which are contained in $\bigcap_{i=1}^n \supp(F_i)$.
\end{defnlist}
Then $\multord{F_1}{F_n}{Z} > 0$ only if there are $i^* = i_1 < \cdots < i_\rho \leq n$, where $\rho := \codim(Z)$, such that $Z$ is an irreducible component of some $Z^{(\rho)}_{i_1, \ldots, i_\rho}$.
\end{thm}

\begin{defn} \label{defn:distinguished}
By an {\em $(i_1, \ldots, i_\rho)$-type distinguished component} for (the ordered intersection of) $F_1, \allowbreak \ldots, \allowbreak F_n$, we mean an irreducible component of some $Z^{(\rho)}_{i_1, \ldots, i_\rho}$ from \Cref{thm:Z}. Note that it is possible for a distinguished component to have zero ordered intersection multiplicity - see \Cref{cor:distinguished} below.
\end{defn}

\begin{rem} \label{rem:all-zero}
If all $F_i$ are identically zero, then the only distinguished component for $F_1, \ldots, F_n$ is $X$, and the number prescribed by \Cref{thm:length} below for $\multord{F_1}{F_n}{X}$ is precisely the intersection number $\mult{\scrL_1}{\scrL_n}$.
\end{rem}

\begin{rem}
Every irreducible component $Z$ of $\bigcap_i \supp(F_i)$ is a distinguished component for $F_1, \ldots, F_n$. Indeed, if $\rho := \codim(Z) = 1$, then $i_1 = i^*$ works (where $i^*$ is as in \Cref{thm:Z}). Otherwise if $\rho > 1$, then set $i_1 = i^*$ and choose an irreducible component $H_1$ of $V(f_{i^*})$ containing $Z$, and let $i_2$ be the smallest index such that $H_1$ is not contained in any irreducible component of $V(f_{i_2})$. If $\rho = 2$, then stop. Otherwise pick an irreducible component $H_2$ of $H_1 \cap V(f_{i_2})$ containing $Z$, and pick the smallest $i_3$ such that $H_2$ is not contained in $V(f_{i_3})$. It is straightforward to check that this process can be continued up to $\rho$ steps, and at the end we have $i_1, \ldots, i_\rho$ as claimed.
\end{rem}

%

For the distinguished components for (the ordered intersection of) $F_1, \ldots, F_n$, we can give an explicit expression of $\multord{F_1}{F_n}{Z}$ in terms of the length of certain local rings.

\begin{thm} \label{thm:length}
Let $Z$ be an $(i_1, \ldots, i_\rho)$-type distinguished component for $F_1, \ldots, F_n$, where $\rho := \codim(Z)$. Fix an open affine neighborhood $U$ of $X$ such that $U \cap Z \neq \emptyset$ and each $F_i$ is defined by a regular function $f_i$ on $U$ which is either identically zero (in the case that $F_i$ is identically zero on $X$), or a non zero-divisor in $\kk[U]$. Consider ideals $I^{(j)}_{i_1, \ldots, i_\rho}$, $0 \leq j \leq \rho$, of $\kk[U]$ defined as follows:
\begin{defnlist}
\item $I^{(0)}_{i_1, \ldots, i_\rho}$ is the zero ideal,
\item $I^{(1)}_{i_1, \ldots, i_\rho}$ is the principal ideal generated by $f_{i_1}$,
\item for $2 \leq j \leq \rho$, $I^{(j)}_{i_1, \ldots, i_\rho}$ is the ideal generated by $f_{i_j}$ and the intersection of all primary components $\qqq$ of $I^{(j-1)}_{i_1, \ldots, i_\rho}$ such that the corresponding prime ideal $\sqrt{\qqq}$ contains $f_i$ for each $i < i_j$, but does {\em not} contain $f_{i_j}$.
\end{defnlist}
Then
\begin{align*}
\multord{F_1}{F_n}{Z}
    &= \sum_{i_1, \ldots, i_\rho}
        \len(\local{X}{Z}/I^{(\rho)}_{i_1, \ldots, i_\rho} \local{X}{Z})
        \deg_{i'_1, \ldots, i'_{n-\rho}}(Z)
\end{align*}
where the sum is over all $i_1, \ldots, i_\rho$ such that $Z$ is an $(i_1, \ldots, i_\rho)$-type distinguished component for $F_1, \ldots, F_n$, and $i'_1, \ldots, i'_{n-\rho}$ are the elements of $\{1, \ldots, n\} \setminus \{i_1, \ldots, i_\rho\}$, and
\begin{align*}
\deg_{i'_1, \ldots, i'_{n-\rho}}(Z)
    &:= \mult{\scrL_{i'_1}|_Z}{\scrL_{i'_{n-\rho}}|_Z}
\end{align*}
\end{thm}

\begin{cor} \label{cor:distinguished}
Let $Z$ be an irreducible subvariety of $X$. Then $\multord{F_1}{F_n}{Z} > 0$ if and only if there are $i^* = i_1 < \cdots < i_{\rho} \leq n$, where $\rho := \codim(Z)$ and $i^*$ is the smallest index such that $F_{i^*}$ is not identically zero, such that
\begin{enumerate}
\item $Z$ is an $(i_1, \ldots, i_\rho)$-type distinguished component for $F_1, \ldots, F_n$, and
\item $\deg_{i'_1, \ldots, i'_{n-\rho}}(Z) > 0$, 
\end{enumerate}
where $i'_1, \ldots, i'_{n-\rho}$ are the elements of $\{1, \ldots, n\} \setminus \{i_1, \ldots, i_\rho\}$. \qed
\end{cor}

\begin{rem}
$\deg_{i'_1, \ldots, i'_{n-\rho}}(Z)$ is the number (counted with appropriate multiplicities) of points in $Z \cap \bigcap_j \supp(G_{i'_j})$ for generic $G_{i'_j} \in \scrL_{i'_j}$. It is analogous to the usual degree of $Z$, and can be considered a ``global factor'' since it depends on the intersection theory of $X$ and the line bundles $\scrL_i$. The term $\len(\local{X}{Z}/I^{(\rho)}_{i_1, \ldots, i_\rho} \local{X}{Z})$ is the ``local factor'' since it depends on how the $F_{i_j}$ vanish along $Z$.
\end{rem}

\begin{rem}
Let $Z$ be a distinguished component for $F_1, \ldots, F_n$. If $Z$ is a point, then the ``global factor'' is $1$ (since $\rho = n$, and therefore $\{i'_1, \ldots, i'_{n-\rho}\}$ is the empty set), and if in addition $Z$ is an isolated point of $\bigcap_i \supp(F_i)$, then the ``local factor'' $\len(\local{X}{Z}/I^{(\rho)}_{i_1, \ldots, i_\rho} \local{X}{Z})$ is precisely the intersection multiplicity of $F_1, \ldots, F_n$ at $Z$ (this follows e.g.\ 
since intersections of divisors are well defined as cycle classes on the support of the intersection \cite[Theorem 2.4]{fultersection}). On the other extreme, if $Z = X$ (which is only possible if all the $F_i$ are identically zero), then as stated in \Cref{rem:all-zero}, the local factor is $1$, and the global factor is $\mult{\scrL_1}{\scrL_n}$.
\end{rem}

%

\begin{rem}
The ordered multiplicity is additive in the ``first'', i.e.\ the ``left most'' component. However, it is in general {\em not} additive in any other component. Indeed, let $F_1 = V(x_0^2x_1)$, $F_2 = V(x_0x_1^2)$, $F'_1 = V(x_0^2)$ be the divisors on $X := \pp^2$ from \Crefrange{ex:L2:11}{ex:L2:1'1}. It follows from the computations in \Cref{ex:L2:1>2,ex:L2:1'1} that
\begin{align*}
\multordonly{F_2,F_1}{Z}
    &=
    \begin{cases}
      3 & \text{if}\ Z = V(x_0), \\
      6 & \text{if}\ Z = V(x_1), \\
      0 & \text{otherwise.}
    \end{cases} \\
\multordonly{F_2,F'_1}{Z}
    &=
    \begin{cases}
      2 & \text{if}\ Z = V(x_0), \\
      4 & \text{if}\ Z = V(x_0, x_1), \\
      0 & \text{otherwise.}
    \end{cases}
\end{align*}
Write $F''_1 := V(x_1)$, so that $F_1 = F'_1 + F''_1$. It can be shown similarly (cf.\ \Cref{ex:2-1}) that
\begin{align*}
\multordonly{F_2,F''_1}{Z}
    &=
    \begin{cases}
      2 & \text{if}\ Z = V(x_1), \\
      1 & \text{if}\ Z = V(x_0, x_1), \\
      0 & \text{otherwise.}
    \end{cases}
\end{align*}
In particular, it follows that in general $\multordonly{F_2,F_1}{Z} \neq \multordonly{F_2,F'_1}{Z} + \multordonly{F_2,F''_1}{Z}$, even though $\multonly{F_2, F_1} = \multonly{F_2, F'_1} + \multonly{F_2, F''_1}$.
\end{rem}

\subsection{}
\Cref{thm:defn,thm:Z,thm:length} are proved in \Cref{part:ordered} of this article (i.e.\ \Crefrange{sec:prelim}{sec:formula}). In \Cref{part:bezout} we explicitly compute the ordered intersection multiplicity in some special cases and present examples to illustrate applications to the affine B\'ezout problem of counting isolated solutions of systems of polynomials. In particular, we consider some criteria under which the following holds:
\begin{align}
\parbox{0.67\textwidth}{
the ``local factor'' $\len(\local{X}{Z}/I^{(\rho)}_{i_1, \ldots, i_\rho} \local{X}{Z})$ can be represented as an intersection multiplicity of hypersurfaces on a nonsingular variety at a proper isolated point of intersection.
} \tag{property I} \label{property:I}
\end{align}
If $X$ is nonsingular, then \eqref{property:I} holds e.g.\ when $\codim(Z) \leq 2$ or $\codim(Z) = n$, or when $Z$ is an isolated component of $\bigcap_{j = 1}^\rho \supp(F_{i_j})$ (\Cref{prop:local=mult:1}). In case \eqref{property:I} holds, one can estimate the intersection multiplicity in terms of corresponding Newton diagrams and can also explicitly identify when this estimate is not exact using the results from \cite[Chapter IX]{howmanyzeroes} - this is described in \Cref{sec:newton:0,sec:newton:complete}. In \Cref{sec:newton0} we describe a scenario under which a similar convex geometric estimate for the local factor is available even when \eqref{property:I} does not hold. To illustrate an application of all these estimates to the affine B\'ezout problem, in \Cref{sec:tangent} we compute the number of common tangent lines to $4$ general spheres in $\kk^3$ by setting it up as a counting problem on the Grassmannian of lines in $\pp^3$ (or equivalently, planes in $\kk^4$), and estimating the ordered intersection multiplicity of corresponding excess intersections\footnote{This problem was solved by Macdonald, Pach, and Theobald [Discrete Comput.\ Geom.\ 26 (1), 2001] via modelling it as a different counting problem without any excess intersections. Our exposition was inspired by a paragraph by Sottile and Theobald in \cite{sottile-theobald} which mentions a similar explanation in terms of usual excess intersection multiplicities in a manuscript by Aluffi and Fulton.}. Finally, \Cref{sec:ind} presents a few inductive ``solutions'' to the affine B\'ezout problem mentioned in \Cref{sec:big}.

\part{Ordered intersection multiplicity} \label{part:ordered}

\section{Initial observations} \label{sec:prelim}

\subsection{} \label{sec:prelim:C}
In this section we make a few elementary observations which are the basis of all our results on ordered intersection multiplicity. We work locally on an irreducible affine variety $X$ with dimension $n$ and fixed regular functions $f_1, \ldots, f_\eta$ (in this section $\eta$ can be an arbitrary positive integer, i.e.\ does not have to equal $n$). Given $g_i \in \kk[X]$ and positive integers $r_i$, $1 \leq i \leq \eta$, let
\begin{align*}
C^*(g_1, \ldots, g_\eta; r_1, \ldots, r_\eta)
    &:= V(f_1 - t^{r_1}g_1, \ldots, f_\eta - t^{r_\eta}g_\eta) \cap (X \times (\kk \setminus \{0\}))
\end{align*}
and let $C(g_1, \ldots, g_\eta; r_1, \ldots, r_\eta)$ be the closure of $C^*(g_1, \ldots, g_\eta; r_1, \ldots, r_\eta)$ in $X \times \kk$. Note that in this section and the next we treat these objects only as algebraic {\em varieties}, not schemes (in \Cref{sec:formula} we would treat these as schemes), and often denote them simply by $C^*$ and $C$ if the $g_i$ and $r_j$ are clear from the context. Given a subvariety $Z$ of $X$, we will study under what conditions $C$ intersect generic points of $Z$ at $t = 0$ for generic $g_i$ in the case that $r_i/r_{i+1} \gg 1$ for each $i$. We start with a simple observation.

\begin{lemma}\label{observation:zero}
Let $B$ be a branch of a curve on $C^*(g_1, \allowbreak \ldots, \allowbreak g_\eta; \allowbreak r_1, \allowbreak \ldots, \allowbreak r_\eta)$ centered at $(z,0)$ for some $z \in X$. Then
\begin{align*}
\ord_B(f_i) \geq r_i\ord_B(t)
\end{align*}
for each $i$; moreover $\ord_B(f_i) = r_i\ord_B(t)$ if and only if $g_i(z) \neq 0$. (See \Cref{sec:branch-defn} on page \pageref{sec:branch-defn} for the definition of (order of a function at) a branch of a curve.)\qed
\end{lemma}

\subsection{}
For $\nu \in \rr$, we write \eqref{condition:>nu} as a shorthand for the following condition:
\begin{align}
r_i > 0\ \text{and}\ r_i/r_{i+1} > \nu,\ 1 \leq i \leq \eta-1, \text{and}\ r_\eta > \nu
\tag{condition $r_{.}/r_{.+} > \nu$} \label{condition:>nu}
\end{align}
We say that ``a certain property $P$ is true under \eqref{condition:>>}'' to mean the following:
\begin{align}
\parbox{0.5\textwidth}{
there is a real number $\nu$ such that property $P$ holds whenever \eqref{condition:>nu} holds.
}
\tag{condition $r_{.}/r_{.+} >> 1$} \label{condition:>>}
\end{align}

\begin{prop} \label{prop:zero}
Let $Z$ be an irreducible subvariety of $X$ such that $\prod_{i=1}^s f_{k'_i} \in \sqrt{\langle f_1, \ldots, f_k \rangle \local{X}{Z}}$ for $k'_1, \ldots, k'_s > k$. Then there is a nonempty open subset $Z^*$ of $Z$ such that under \eqref{condition:>>}, for all $z \in Z^*$, if $(z,0) \in C$, then $g_{k'_i}(z) = 0$ for some $i = 1, \ldots, s$.
\end{prop}

\begin{proof}
Assume $(\prod_{i=1}^s f_{k'_i})^m = \sum_{j=1}^k h_jf_j$ with $h_j \in \local{X}{Z}$. Let $U$ be an open subset of $X$ such that $U \cap Z \neq \emptyset$ and each $h_j$ is a regular function on $U$. Choose $r_1, \ldots, r_\eta$ such that
\begin{align*}
r_j > sm\max\{r_{k'_1}, \ldots, r_{k'_s}\}
\end{align*}
for each $j = 1, \ldots, k$. Let $B$ be a branch (of a curve) on $C^* = C^*(g_1, \ldots, g_\eta; r_1, \ldots, r_\eta)$ centered at $(z,0)$ where $g_i \in L_i$, $i = 1, \ldots, \eta$. Pick $i'$ such that $\ord_B(f_{k'_{i'}}) = \max_{i=1}^s \ord_B(f_{k'_i})$. Then
\begin{align*}
\ord_B(f_{k'_{i'}})
    &\geq   \frac{1}{s} \sum_{i=1}^s \ord_B(f_{k'_i})
    =       \frac{1}{sm} \ord_B((\prod_{i=1}^sf_{k'_i})^m) \\
    &\geq   \frac{1}{sm} \min_{j=1}^k \ord_B(f_j)
    \geq    \frac{1}{sm} \min_{j=1}^k r_j\ord_B(t)
    > r_{k'_{i'}}\ord_B(t)
\end{align*}
It follows that $g_{k'_{i'}}(z) = 0$ (\Cref{observation:zero}), as required.
\end{proof}

\subsection{} \label{sec:condition:nonzero}
Fix a subvariety $Z$ of $X$ and $1 \leq i_1 < \cdots < i_\rho \leq \eta$. \Cref{prop:zero} suggests that it may be fruitful to understand when there is a nonempty open subset $Z^*$ of $Z$ such that under \eqref{condition:>>} there is $z \in Z^*$ which satisfies the following property:
\begin{align}
\parbox{0.6\textwidth}{
$(z,0) \in C$, and $g_{i_{j}}(z) \neq 0$ for all $j = 1, \ldots, \rho$.
}
\tag{$\cancel{0}$}
\label{condition:nonzero}
\end{align}
We say that a branch $B$ on $C$ satisfies \eqref{condition:nonzero} if the center of $B$ is of the form $(z,0)$ such that $z$ satisfies \eqref{condition:nonzero}. We write \eqref{condition:nonzero>nu} to denote the combination of \eqref{condition:>nu} and \eqref{condition:nonzero}, i.e.\ $z$ (respectively, $B$) and $r_1, \ldots, r_\eta$ satisfy \eqref{condition:nonzero>nu} if and only if
\begin{align}
\parbox{0.8\textwidth}{
$r_1, \ldots, r_\eta$ satisfy \eqref{condition:>nu}, and $z$ (respectively, $B$) satisfies \eqref{condition:nonzero}.
}
\tag{$\cancel{0}(\nu)$}
\label{condition:nonzero>nu}
\end{align}

\begin{rem} \label{observation:ord}
Assume $B$ is a branch on $C$ which satisfies \eqref{condition:nonzero}. Then \Cref{observation:zero} implies that
\begin{align*}
\ord_B(f_{i_j})
    &= \ord_B(g_{i_j}t^{r_{i_j}})
    = r_{i_j}\ord_B(t),\quad j = 1, \ldots, \rho.
\end{align*}
In particular, if $i^*$ is the smallest index such that $f_{i^*}$ is {\em not} identically zero, then $i_j \geq i^*$ for each $j$.
\end{rem}

\begin{prop} \label{prop:nnonzero}
Let $Z$ be an irreducible subvariety of $X$.
\begin{enumerate}
\item If each $f_i$ is identically zero, or $Z = X$, or $Z \not\subseteq V(f_1, \ldots, f_\eta)$, then there is a nonempty open subset $Z^*$ of $Z$ such that no $z \in Z^*$ satisfies \eqref{condition:nonzero} for any choice of $g_i, r_j, \rho, i_1, \ldots, i_\rho$ such that $1 \leq \rho \leq \eta$ and $1 \leq i_1 < \cdots < i_\rho \leq \eta$.
\item Otherwise, let $i^*$ be as in \Cref{observation:ord} the smallest index such that $f_{i^*}$ is not identically zero. Then there is $\nu \in \rr$ and a nonempty open subset $Z^*$ of $Z$ such that if there are $z \in Z^*$ and $r_1, \ldots, r_\eta$ which satisfy \eqref{condition:nonzero>nu} for some $g_1, \ldots, g_\eta$, then the number of distinct $j$ such that $i_j > i^*$ is at most $\codim(Z) - 1$.
\end{enumerate}
\end{prop}

\begin{proof}
If each $f_i$ is identically zero, then the first assertion holds with $Z^* := Z$ (\Cref{observation:zero}). So assume there is $i$ such that $f_i$ is not identically zero, or in other words, $i^* \leq \eta$. If $Z = X$ (respectively, $Z \not\subseteq V(f_1, \ldots, f_\eta)$), then again the first assertion holds with $Z^* := X \setminus V(f_{i^*})$ (respectively, $Z^*: = Z \setminus V(f_1, \ldots, f_\eta)$). So we may assume
\begin{align*}
Z \subseteq V(f_1, \ldots, f_\eta) \subsetneqq X
\end{align*}

\subsubsection{} \label{nnonzero:assumptions}
Now we assume the proposition is not true. Then, since there are only finitely many possibilities for the collections of $i_j$, there is a {\em fixed} collection $i^* < i_1 < \cdots < i_\rho \leq \eta$ such that
\begin{prooflist}
\item $\rho = \codim(Z) \geq 1$, and
\item \label{nnonzero:assumption:all} for {\em all} $\nu \in \rr$ and {\em all} nonempty open subsets $Z^*$ of $Z$, there are $z \in Z^*$ and $g_i,r_j$ which satisfy \eqref{condition:nonzero>nu}.
\end{prooflist}

\subsubsection{} \label{nnonzero:rho=1}
We claim that at least one irreducible component of $V(f_{i^*})$ containing $Z$ is not contained in $\bigcup_{j=1}^\rho V(f_{i_j})$. Indeed, otherwise $\prod_{j=1}^\rho f_{i_j} \in \sqrt{f_{i^*}\local{X}{Z}}$, and we get a contradiction to property \ref{nnonzero:assumption:all} due to \Cref{prop:zero}. Choose an open subset $U_1 \subseteq X$ intersecting $Z$ and a regular function $h_1$ on $U_1$ such that
\begin{itemize}
\item $h_1$ identically vanishes on each irreducible component of $V(f_{i^*})$ which contains $Z$ but is not contained in $V(\prod_{j=1}^{\rho} f_{i_j})$,
\item $h_1$ does not vanish identically on any irreducible component of $V(\prod_{j=1}^{\rho} f_{i_j})$.
\end{itemize}
Since $h_1\prod_{j=1}^\rho f_{i_j}$ identically vanishes on each irreducible component of $V(f_{i^*})$ containing $Z$, there is an open subset $U^*_1 \subseteq U_1$ intersecting $Z$ and $m_1 \geq 1$ such that
\begin{align*}
(h_1\prod_{j=1}^\rho f_{i_j})^{m_1} \in f_{i^*}\kk[U^*_1]
\end{align*}
Now, if $B$ is a branch on $C$ centered at $U^*_1 \times \{0\}$ that satisfies \eqref{condition:nonzero>nu} for some $\nu \geq 1$, then \Cref{observation:ord} implies that
\begin{align*}
\ord_B(f_{i_1}) = r_{i_1}\ord_B(t) > r_{i_j}\ord_B(t) = \ord_B(f_{i_j}),\quad j = 2, \ldots, \rho,
\end{align*}
so that
\begin{align*}
m_1(\ord_B(h_1) + \rho r_{i_1}\ord_B(t))
    &\geq m_1(\ord_B(h_1) + \sum_{j=1}^\rho \ord_B(f_{i_j}))
    &\geq \ord_B(f_{i^*})
    \geq r_{i^*}\ord_B(t)
\end{align*}
It follows that
\begin{align*}
\ord_B(h_1)
    &> \frac{r_{i^*} - m_1\rho r_{i_1}}{m_1}\ord_B(t)
\end{align*}
Since $i^* < i_1$, it is then clear that we can find $\nu_1 \geq 1$ and $\mu_1 > 0$ such that if $B$ is a branch on $C$ centered at $U^*_1 \times \{0\}$ that satisfies \eqref{condition:nonzero>nu} for some $\nu \geq \nu_1$, then
\begin{align*}
\ord_B(h_1) &> \frac{r_{i^*}}{\mu_1}\ord_B(t) > r_{i_1}\ord_B(t) = \ord_B(f_{i_1})
\end{align*}
If $\rho = 1$, then we stop. Otherwise we consider $V(h_1) \cap V(f_{i_1})$, which by construction of $h_1$, has pure codimension $2$ in $X$.

\begin{proclaim} \label{nnonzero:h1-cap-fi2}
At least one irreducible component of $V(h_1) \cap V(f_{i_1})$ containing $Z$ is not contained in $\bigcup_{j=2}^\rho V(f_{i_j})$.
\end{proclaim}

\begin{proof}
Indeed, otherwise $\prod_{j=2}^\rho f_{i_j} \in \sqrt{\langle h_1, f_{i_1} \rangle \local{X}{Z}}$. Then there is $U'_2 \subseteq U^*_1$ and $m'_2 \geq 1$ such that
\begin{align*}
(\prod_{j=2}^\rho f_{i_j})^{m'_2} \in \langle h_1, f_{i_1} \rangle \kk[U_2]
\end{align*}
If $B$ is a branch on $C$ centered at $U'_2 \times \{0\}$ which satisfies \eqref{condition:nonzero>nu} for some $\nu \geq \nu_1$ as in the end of \Cref{nnonzero:rho=1}, then
\begin{align*}
m'_2(\rho-1)r_{i_2}\ord_B(t)
    &= m'_2(\rho - 1) \ord_B(f_{i_2})) \\
    &\geq m'_2 \sum_{j=2}^\rho \ord_B(f_{i_j}) \\
    &\geq \min\{\ord_B(h_1), \ord_B(f_{i_1})\} \\
    &\geq \ord_B(f_{i_1})) \\
    &= r_{i_1}\ord_B(t)
\end{align*}
which implies
\begin{align*}
r_{i_2}
    &\geq \frac{r_{i_1}}{m'_2(\rho -1)}
\end{align*}
However, this is impossible if we take $\nu > m'_2(\rho - 1)$ (it is permissible to take arbitrarily large $\nu$ due to assumption \ref{nnonzero:assumption:all} in \Cref{nnonzero:assumptions}). This contradiction proves the claim.
\end{proof}

\subsubsection{} Choose an open subset $U_2 \subseteq U^*_1$ intersecting $Z$ and a regular function $h_2$ on $U_2$ such that
\begin{itemize}
\item $h_2$ identically vanishes on each irreducible component of $V(h_1, f_{i_1})$ which contains $Z$ but is not contained in $V(\prod_{j=2}^{\rho} f_{i_j})$,
\item $h_2$ does not vanish identically on any irreducible component of $V(h_1\prod_{j=2}^{\rho} f_{i_j})$.
\end{itemize}
Since $h_2\prod_{j=2}^\rho f_{i_j}$ identically vanishes on each irreducible component of $V(h_1, f_{i_1})$ containing $Z$, there is an open subset $U^*_2 \subseteq U_2$ intersecting $Z$ and $m_2 \geq 1$ such that
\begin{align*}
(h_2\prod_{j=2}^\rho f_{i_j})^{m_2} \in \langle h_1, f_{i_1} \rangle \kk[U^*_2]
\end{align*}
It follows as in the proof of \Cref{nnonzero:h1-cap-fi2} that if $B$ is a branch on $C$ centered at $U^*_2 \times \{0\}$ which satisfies \eqref{condition:nonzero>nu} for some $\nu \geq \nu_1$, then
\begin{align*}
m_2(\ord_B(h_2) + (\rho - 1) r_{i_2}\ord_B(t))
    &\geq r_{i_1}\ord_B(t)
\end{align*}
so that
\begin{align*}
\ord_B(h_2)
    &\geq \frac{r_{i_1} - m_2(\rho - 1) r_{i_2}}{m_2}\ord_B(t)
\end{align*}
Consequently we can choose $\nu_2 \geq \nu_1$ and a positive number $\mu_2$ such that if $B$ is a branch on $C$ centered at $U^*_2 \times \{0\}$ which satisfies \eqref{condition:nonzero>nu} for some $\nu \geq \nu_2$, then
\begin{align*}
\ord_B(h_2) &> \frac{r_{i_1}}{\mu_2}\ord_B(t) > r_{i_2}\ord_B(t) = \ord_B(f_{i_2})
\end{align*}

\subsubsection{} If $\rho = 2$, then we stop. Otherwise we proceed as above, and after $\rho$ steps we end up with an open subset $U^*_\rho$ of $X$ such that $U^*_\rho \cap Z \neq \emptyset$, and regular functions $h_1, \ldots, h_\rho$ on $U^*_\rho$, and positive real numbers $\nu_1, \ldots, \nu_\rho, \mu_1, \ldots, \mu_\rho$ such that for each $k$, $1 \leq k \leq \rho$,
\begin{enumerate}
\item $Y_k := V(h_1, h_2, \ldots, h_k) \cap U^*_\rho$ is a pure codimension $k$ (possibly reducible) subvariety of $X$ containing $Z$, and
\item if $B$ is a branch on $C$ centered at $U^*_\rho \times \{0\}$ which, together with $r_1, \ldots, r_\eta \in \rr$, satisfies \eqref{condition:nonzero>nu} for some $\nu \geq \nu_k$, then
\begin{align*}
\ord_B(h_k) > \frac{1}{\mu_k} r_{i_{k-1}}\ord_B(t) 
\end{align*}
where we set $i_0 := i^*$ for consistency.
\end{enumerate}
Now since $Y_\rho \supseteq Z$ has pure codimension $\rho = \codim(Z)$, there is an open subset $U^*$ of $U^*_\rho$ such that $Y_\rho \cap U^* = Z \cap U^* \neq \emptyset$. Since $Z \subseteq V(f_1, \ldots, f_\eta) \subseteq V(f_{i_\rho})$, it follows that $(f_{i_\rho})^{m_\rho} \in \langle h_1, \ldots, h_{\rho} \rangle \kk[U^*]$ for some $m_\rho \geq 1$. By assumption \ref{nnonzero:assumption:all} from \Cref{nnonzero:assumptions}, there are $z \in U^* \cap Z$ and $g_i, r_j$ which satisfy \eqref{condition:nonzero>nu} for some $\nu > \max\{\nu_1, \ldots, \nu_\rho, m_\rho\mu_1, \ldots, m_\rho\mu_\rho\}$. Then for each branch $B$ of $C$ centered at $(z,0)$,
\begin{align*}
\ord_B(f_{i_\rho})
    &\geq \min_{k=1}^\rho \frac{\ord_B(h_k)}{m_\rho}
    \geq \min_{k=1}^\rho \frac{r_{i_{k-1}}}{m_\rho\mu_k} \ord_B(t)
    > r_{i_{\rho}}\ord_B(t)
\end{align*}
which contradicts \Cref{observation:ord} and completes the proof of \Cref{prop:nnonzero}.
\end{proof}

\section{Distinguished components of an ordered intersection} \label{sec:distinguished}
In this section we determine the ``distinguished components'' of ordered intersection. The main result of this section is \Cref{thm:Z-existence} which shows that the property of ``having zero (or nonzero) ordered intersection multiplicity'' is well defined, and there can be at most finitely many irreducible subvarieties with nonzero ordered intersection multiplicity, and all of them must be of the form described in \Cref{thm:Z}. The arguments used in this section are elementary (but involved) except for one place in the proof of \Cref{prop:nonzero} where de Jong's theorem on alteration \cite[Theorem 4.1]{dejong} is used\footnote{We believe the results of this article could have been proved without using  de Jong's theorem on alteration; however, that would have made the arguments much more involved.}.


\subsection{Set up} \label{sec:distinguished-setup}
We start with a set up which is a bit more general than that of \Cref{problem0}. We work locally on an irreducible affine variety $X$ of dimension $n$ with fixed regular functions $f_1, \ldots, f_n$ and subsets $L_1, \ldots, L_n$ of $\kk[X]$ such that for each $i$,
\begin{enumerate}
\item $L_i$ is an {\em irreducible constructible} set of dimension $\geq 1$, and
\item $L_i$ is {\em base-point free}, i.e.\ for each $x \in X$, there is $g \in L_i$ such that $g(x) \neq 0$.
\end{enumerate}
%
Given $g_i \in L_i$ and positive integers $r_i$, $1 \leq i \leq n$, we consider $C = C(g_1, \ldots, g_n; r_1, \ldots, r_n)$ as in \Cref{sec:prelim}.
Building on \Cref{prop:zero,prop:nnonzero}, the following result shows how the ordering of $f_1, \ldots, f_n$ dictates the subvarieties $Z$ of $X$ for which $C = C(g_1, \ldots, g_n;\allowbreak r_1, \ldots, r_n)$ may intersect $Z \times \{0\}$ under \eqref{condition:>>} in the case that $\eta = n$.


\begin{thm} \label{prop:nnonzero-2}
Assume we are in the set up of \Cref{sec:distinguished-setup}. Let $Z$ be a subvariety of $X$ with $\rho := \codim(Z) > 0$. Then there are $\nu^* \in \rr$, a nonempty Zariski open subset $Z^*$ of $Z$, and a nonempty Zariski open subset $L^*$ of $\prod_i L_i$ such that for each $(g_1, \ldots, g_n) \in L^*$, one (and only one) of the following holds for every $\nu \geq \nu^*$:
\begin{enumerate}
\item either $C$ does not intersect $Z^* \times \{0\}$ under \eqref{condition:>nu},
\item or there are $z \in Z^*$ and $r_1, \ldots, r_n \in \rr$ such that \eqref{condition:>nu} holds and $(z,0) \in C$. In this case there are $\rho$ choices for $i$, say $1 \leq i_1 < \cdots < i_\rho \leq n$, such that
    \begin{enumerate}[resume]
    \item $i_1 = i^*$, where $i^*$, as in \Cref{observation:ord}, is the smallest index such that $f_{i^*}$ is not identically zero, and
    \item $g_i(z) \neq 0$ if and only if $i = i_j$ for some $j = 1, \ldots, \rho$.
    \end{enumerate}
\end{enumerate}
\end{thm}


\begin{prorem}
Note that \Cref{prop:nnonzero-2} leaves open the possibility that some $(g_1, \ldots, g_n) \in L^*$ satisfies the first assertion, whereas some other $(g_1, \ldots, g_n) \in L^*$ satisfies the second. Moreover, it does {\em not} give any control on whether any of these options holds ``uniformly''. \Cref{thm:Z-existence} below provides a uniform version.
\end{prorem}

\begin{prorem}
It is possible (when $1 < \rho < n$) that the same $Z$ satisfies the second assertion of \Cref{prop:nnonzero-2} for more than one choices of $i_1, \ldots, i_\rho$ (see e.g.\ identity \eqref{eq:3-2:2:m:0} in \Cref{ex:3-2}).
\end{prorem}

\subsubsection{Proof of \Cref{prop:nnonzero-2}}
Let $L'$ be the subset of $\prod_i L_i$ consisting of all $(g_1, \ldots, g_n)$ such that $Z \cap V(g_{i'_1}, \ldots, g_{i'_{n-\rho+1}}) = \emptyset$ for all choices of $1 \leq i'_1 < \cdots < i'_{n-\rho+1} \leq n$. Then $L'$ is a full-dimensional constructible subset of $\prod_i L_i$, and therefore contains a nonempty Zariski open subset $L^*$ of $\prod_i L_i$. Now due to \Cref{prop:nnonzero} there is $\nu^* \in \rr$ and a nonempty open subset $Z^*$ of $Z$ such that the following holds for all $\nu \geq \nu^*$: if $z \in Z^*$ and $r_1, \ldots, r_n \in \rr$ are such that $(z,0) \in C$ and \eqref{condition:>nu} holds, then there are at most $\rho$ choices for $k$ such that $g_k(z) \neq 0$. Consequently, for each $(g_1, \ldots, g_n) \in L^*$ and $\nu \geq \nu^*$ there are two possible scenarios:
\begin{prooflist}
\item either $C$ does not intersect $Z^* \times \{0\}$ under \eqref{condition:>nu},
\item or there are $z \in Z^*$ and $r_1, \ldots, r_n \in \rr$ such that \eqref{condition:>nu} holds and $(z,0) \in C$.
\end{prooflist}
The first scenario is precisely the first possibility stated in \Cref{prop:nnonzero-2}, and under the second scenario, the second assertion of \Cref{prop:nnonzero-2} holds due to \Cref{prop:nnonzero} and the construction of $L^*$. \qed
%

\subsection{} \label{sec:nonzero:+}
In light of \Cref{prop:nnonzero-2}, from now on we consider conditions \eqref{condition:nonzero} and \eqref{condition:nonzero>nu} only in the special case $i_1 = i^*$. Before proceeding further, we explicitly state a somewhat technical, but straightforward, implication of the condition \eqref{condition:nonzero>nu} that was implicitly used several times in the proof of \Cref{prop:nnonzero}. Fix $i^* = i_1 < i_2 < \cdots < i_\rho \leq n$, a nonempty open subset $Z^*$ of $Z$ and a collection $\scrH^*$ of pairs $(h,i)$ such that $1 \leq i \leq n$, and $h$ is a regular function on an open subset of $X$ containing $Z^*$. We say that $\scrH^*, Z^*$ satisfy \eqref{nonzero:+} for $\mu^*, \nu^* \in \rr$ if the following holds:
\begin{align}
\parbox{0.85\textwidth}{
there are $\nu \geq \nu^*$, $r_1, \ldots, r_n \in \rr$ and a branch $B$ of $C = C(g_1, \ldots, g_n;r_1, \ldots, r_n)$ centered at $Z^* \times \{0\}$ such that \eqref{condition:nonzero>nu} holds, and in addition, $\ord_B(h) > \mu^* \ord_B(f_{i}) \geq \mu^* r_i\ord_B(t)$ for each $(h,i) \in \scrH^*$.
}
\tag{$\cancel{0}(\mu^*, \nu^*)$}
\label{nonzero:+}
\end{align}

\begin{lemma} \label{prop:nonzero:+j}
Fix $i^* = i_1 < i_2 < \cdots < i_\rho \leq n$ and an open subset $U^*$ of $X$ such that $U^* \cap Z \neq \emptyset$. Let $\scrH$ be a collection of pairs $(h, i)$ of $h \in \kk[U^*]$ and $1 \leq i \leq n$ such that
\begin{prooflist}
\item property \eqref{nonzero:+} holds with $\scrH^* = \scrH$ and $Z^* = U^* \cap Z$ for {\em each} $\mu^*, \nu^* \in \rr$;
\item there is $j$, $1 \leq j \leq \rho - 1$, such that $i \leq i_{j+1}$ for each $(h,i) \in \scrH$.
\end{prooflist}
Let $Y := V(h: (h,i) \in \scrH)$. Then
\begin{enumerate}
\item \label{nonzero:+j:notsubset} $Y \cap U^* \not\subseteq \bigcup_{j'>j}V(f_{i_{j'}})$.
\item \label{nonzero:+j:order} If $h^* \in \kk[U^*]$ is such that $h^*\prod_{j' > j}f_{i_{j'}}$ vanishes on $Y \cap U^*$, then for each $\mu^*, \nu^* \in \rr$, then $\scrH^* = \scrH \cup\{(h^*, i_{j+1})\}$ also satisfies \eqref{nonzero:+} with $Z^* = U^* \cap Z$.
\end{enumerate}
\end{lemma}

\begin{proof}
First we prove assertion \eqref{nonzero:+j:order}. By assumption (and Hilbert's Nullstellensatz) there is $m \geq 1$ such that $(h^*\prod_{j' > j}f_{i_{j'}})^m$ is in the ideal of $\kk[U^*]$ generated by all $h$ such that $(h,i) \in \scrH$ for some $i$. If $B$ is any branch $B$ of $C$ centered at $U^* \times \{0\}$ which satisfies \eqref{condition:nonzero>nu} for some $\nu \geq 1$, then
\begin{align*}
\ord_B(h^*)
    &\geq \frac{1}{m} (\min_{(h,i) \in \scrH} \ord_B(h) - m\sum_{j'>j} \ord_B(f_{i_{j'}}))
    \geq \frac{1}{m} (\min_{(h,i) \in \scrH} \ord_B(h) - m(\rho-j) \ord_B(f_{i_{j+1}}))
\end{align*}
Now, given $\mu, \nu$, choose $r_1, \ldots, r_n \in \rr$ and a branch $B$ of $C$ centered at $(U^* \cap Z) \times \{0\}$ such that \eqref{nonzero:+} holds for $\mu^* = m(\rho-j + \mu)$ and $\nu^* = \nu$. Then $\ord_B(h^*) > \mu\ord_B(f_{i_{j+1}})$, as required to complete the proof of assertion \eqref{nonzero:+j:order}. Now, if assertion \eqref{nonzero:+j:notsubset} is not true, then $\prod_{j'>j}f_{i_{j'}}$ vanishes on $Y \cap U^*$, and assertion \eqref{nonzero:+j:order} applied with $h^* := 1$ gives a contradiction. This completes the proof.
\end{proof}

\subsection{}
Our next result shows that any subvariety of $X$ not listed in \Cref{thm:Z} has zero ordered intersection multiplicity. 
Let $Z$ be an irreducible subvariety of $X$. Assume at least one $f_i$ is not identically zero; as in \Cref{observation:ord}, let $i^*$ be the smallest such $i$. Fix $i^* = i_1 < \cdots < i_\rho \leq n$, where $\rho := \codim(Z)$ and define subvarieties $\dist{Z}{0}, \ldots, \dist{Z}{\rho}$ of $X$ as follows:
\begin{defnlist}
    \item $\dist{Z}{0} := X$,
\item for each $j = 1, \ldots, \rho-1$, $\dist{Z}{j}$ is the union of all irreducible components of $\dist{Z}{j-1} \cap V(f_{i_j})$ containing $Z$ which are not contained in $\bigcup_{j' = j+1}^\rho V(f_{i_{j'}})$, but are contained in $V(f_i)$ for each $i$, $i_j < i < i_{j+1}$.
\item for $j = \rho$,
\begin{align*}
\dist{Z}{\rho}
    &=
    \begin{cases}
    Z
        & \text{if}\ \dist{Z}{\rho-1} \neq \emptyset\ \text{and}\ Z \subseteq V(f_1, \ldots, f_n), \\
    \emptyset
        & \text{otherwise}.
    \end{cases}
\end{align*}
\end{defnlist}

\begin{lemma} \label{prop:nonzero}
If $\dist{Z}{j} = \emptyset$ for some $j = 1, \ldots, \rho$, then there is $\nu \in \rr$, a nonempty open subset $Z^*$ of $Z$, and a nonempty Zariski open subset $L^*$ of $\prod_i L_i$ such that no branch $B$ of $C = C(g_1, \ldots, g_n;r_1, \ldots, r_n)$ centered at $Z^* \times \{0\}$ satisfies \eqref{condition:nonzero>nu} for {\em any} $(g_1, \ldots, g_n) \in L^*$. In other words, for each $(g_1, \ldots, g_n) \in L^*$, and each choice of positive $r_1, \ldots, r_n$ such that $r_i/r_{i+1} > \nu$, $1 \leq i \leq n-1$, and each $z \in Z^*$ such that $(z,0) \in C$, there is $j$, $1 \leq j \leq \rho$, such that $g_{i_j}(z) = 0$.
\end{lemma}

\begin{proof}
We prove the lemma in several steps. First we reduce it to the case that the union of the set of zeroes of the $f_i$ is the support of a {\em strict normal crossing divisor}\footnote{An effective Cartier divisor $D$ on a nonsingular variety $X$ is a {\em strict normal crossing divisor} if near each $x \in \supp(D)$, $D$ can be expressed as a sum of divisors of regular functions which collectively form a part of a regular system of parameters in $\local{X}{x}$.} using de Jong's {\em alterations}\footnote{We use the following weak version of de Jong's theorem on ``alterations'' \cite[Theorem 4.1]{dejong}: Let $Z$ be a proper closed subset of a variety $X$. There is a surjective generically finite morphism $\pi: X' \to X$ such that $X'$ is nonsingular, and $\pi^{-1}(Z)$ is the support of a strict normal crossing divisor on $X'$.}\cite[Theorem 4.1]{dejong}.

\begin{proprop} \label{nonzero:normal-reduction}
Let $\pi: X' \to X$ be a surjective morphism between varieties of the same dimension $n$. Assume \Cref{prop:nonzero} holds for the ordered intersection of $V(\pi^*(f_i))$, $i = 1, \ldots, n$, on $X'$. Then \Cref{prop:nonzero} holds for the ordered intersection of $V(f_i)$, $i = 1, \ldots, n$, on $X$.
\end{proprop}

\begin{proof}
It suffices to consider the case that $Z \subseteq V(f_1, \ldots, f_n)$, and $\Phi^{(j)}_Z$, as defined in \Cref{prop:nonzero}, is empty for some $j = 1, \ldots, \rho - 1$, where $\rho := \codim(Z)$.
Due to \Cref{prop:nnonzero-2} then it suffices to show that there is no subvariety $Z'$ of $\pi^{-1}(Z)$ such that
\begin{prooflist}
\item the codimension of $Z'$ in $X'$ equals the codimension $\rho$ of $Z$ in $X$,
\item $\pi(Z') = Z$,
\item $\dist{Z'}{j}$, as defined in \Cref{prop:nonzero}, are nonempty for each $j = 1, \ldots, \rho$.
\end{prooflist}
Assume to the contrary that there is a subvariety $Z'$ of $\pi^{-1}(Z)$ satisfying the above properties. Note the following (obvious) observation:
\begin{prooflist} [resume]
\item \label{nonzero:normal-reduction:obvious} given a subset $W'$ of $X'$, $W' \subseteq V(\pi^*(f_i))$ if and only if $\pi(W') \subseteq V(f_i)$. In particular, if $W'$ is irreducible and $W' \not \subseteq V(\pi^*(f_i))$, and $\dim(W') > \dim(\pi(W'))$, then
    \begin{prooflist}
    \item either $W' \cap V(\pi^*(f_i)) = \emptyset$,
    \item or $\dim(\pi(W') \cap V(f_i)) = \dim(\pi(W')) - 1 < \dim(W') - 1 = \dim(W' \cap V(\pi^*(f_i)))$.
    \end{prooflist}
\end{prooflist}
It follows from the definition of $\Phi^{(\cdot)}_\cdot$ and applications of the preceding observation that for each $j = 1, \ldots, \rho$,
\begin{prooflist} [resume]
\item $\pi(\Phi^{(j)}_{Z'})$ is the union of $\Phi^{(j)}_Z$ and a (possibly empty) subvariety $W_j$ of codimension $> j$ in $X$. 
\end{prooflist}
Since $\Phi^{(\rho-1)}_Z = \emptyset$ by assumption, it follows that $\pi(\Phi^{(\rho-1)}_{Z'}) = Z \subseteq V(f_{i_\rho})$. However, then observation \ref{nonzero:normal-reduction:obvious} implies that $\Phi^{(\rho-1)}_{Z'} \subseteq V(\pi^*(f_{i_\rho}))$. This contradicts the definition of $\Phi^{(\rho-1)}_\cdot$, and therefore, proves \Cref{nonzero:normal-reduction}.
\end{proof}

\subsubsection{} \label{nonzero:assumptions}
We go back to the proof of \Cref{prop:nonzero}. It suffices to consider the case that
\begin{prooflist}
\item \label{nonzero:assumption:subset} $Z \subseteq V(f_1, \ldots, f_n)$, for otherwise \Cref{prop:nonzero} is obvious with $Z^* := Z \setminus V(f_1, \ldots, f_n)$,
\item \label{nonzero:assumption:codim} $\rho := \codim(Z) > 1$, since \Cref{prop:nonzero} is clearly true if $\rho \leq 1$,
\item \label{nonzero:assumption:normal} and $\bigcup_i V(f_i)$ with the reduced scheme structure is a strict normal crossing divisor (due to \Cref{nonzero:normal-reduction}).
\end{prooflist}
We will prove the contrapositive of \Cref{prop:nonzero}. In particular, we assume that
\begin{prooflist}[resume]
\item \label{nonzero:assumption:order} for each $\nu \in \rr$, each nonempty open subset $Z^*$ of $Z$, and each nonempty Zariski open subset $L^*$ of $\prod_i L_i$, there are $r_1, \ldots, r_n \in \rr$, $(g_1, \ldots, g_n) \in L^*$, and a branch $B$ of $C = C(g_1, \ldots, g_n;r_1, \ldots, r_n)$ centered at $Z^* \times \{0\}$ such that \eqref{condition:nonzero>nu} holds.
\end{prooflist}
Under these assumptions, we will show that $\dist{Z}{j}$ is nonempty for each $j = 1, \ldots, \rho$.

\subsubsection{}
By assumption \ref{nonzero:assumption:normal}, there is an open subset $U$ of $X$ with local system of coordinates $(x_1, \ldots, x_n)$ such that $U \cap Z \neq \emptyset$ and each $f_i$ which is not identically zero is an invertible function on $U$ multiplied by a monomial in the $x_i$. Define subvarieties $\distprime{Z}{j}$ of $X$ containing $\dist{Z}{j}$, $j = 0, \ldots, \rho$, as follows:
\begin{defnlist}
    \item $\distprime{Z}{0} := X$,
\item for each $j = 1, \ldots, \rho$, $\distprime{Z}{j}$ is the union of all irreducible components of $\distprime{Z}{j-1} \cap V(f_{i_j})$ containing $Z$ which are not contained in $\bigcup_{j' = j+1}^\rho V(f_{i_{j'}})$.
\end{defnlist}

\begin{proclaim} \label{claim:nonzero:nozero'}
Under the assumptions of \Cref{nonzero:assumptions}, $\distprime{Z}{j} \neq \emptyset$, $j = 0, \ldots, \rho$. There are indices $k_1, \ldots, k_\rho$ and a nonempty open subset $U'$ of $U$ such that
\begin{enumerate}
\item $U' \cap Z \neq \emptyset$,
\item for each $j$,
\begin{enumerate}
\item $x_{k_j}$ divides $f_{i_j}$, but does not divide $f_{i_{j'}}$ for any $j' > j$,
\item $\distprime{Z}{j} \cap U' \supseteq V(x_{k_1}, \ldots, x_{k_j}) \cap U' \neq \emptyset$,
\end{enumerate}
\item In particular,  $\distprime{Z}{\rho} \cap U' = V(x_{k_1}, \ldots, x_{k_\rho}) \cap U' \neq \emptyset$.
\end{enumerate}
\end{proclaim}

\begin{proof}
Write $f_{i_1} = u_{i_1}\prod_k x_k^{\alpha_k}$, where $u_{i_1}$ is invertible on $U$. Let $f'_{i_1}$ be the product of all $x_k^{\alpha_k}$ such that $x_k$ does {\em not} divide $f_{i_j}$ for any $j > 1$. Due to assumption \ref{nonzero:assumption:order} of \Cref{nonzero:assumptions}, we may apply \Cref{prop:nonzero:+j} with $j = 1$, $\scrH = \{(f_{i_1}, i_2)\}$ and $h^* = f'_{i_1}$ to obtain that property \eqref{nonzero:+} is satisfied for all $\mu^*, \nu^*$ with $\scrH^* = (f'_{i_1}, i_2)$ and $Z^* = Z \cap U$; in particular, $V(f'_{i_1}) \cap U \neq \emptyset$. Now assume we have inductively defined monomials $f'_{i_1}, \ldots, f'_{i_j}$ in the $x_i$ such that
\begin{prooflist}
\item $V(f'_{i_1}, \ldots, f'_{i_j}) \cap U \neq \emptyset$,
\item each $f'_{i_k}$ divides $f_{i_k}$, and is relatively prime to $\prod_{j' > k} f_{i_{j'}}$,
\item property \eqref{nonzero:+} is satisfied for all $\mu^*, \nu^*$ with $\scrH^* = \{(f'_{i_1}, i_2), \ldots, (f'_{i_j}, i_{j+1})\}$ and $Z^* = Z \cap U$.
\end{prooflist}
If $j = \rho - 1$, then we define $f'_{i_{j+1}} := f_{i_\rho}$, and stop. Otherwise, write $f_{i_{j+1}} = u_{i_{j+1}}\prod_k x_k^{\alpha'_k}$, where $u_{i_{j+1}}$ is invertible on $U$, and define $f'_{i_{j+1}}$ to be the product of all $x_k^{\alpha'_k}$ such that $x_k$ does {\em not} divide $f_{i_{j'}}$ for any $j' > j+1$. It is then straightforward to see by an application of \Cref{prop:nonzero:+j} for $j + 1$ with $\scrH = \{(f'_{i_1}, i_2), \ldots, (f'_{i_j}, i_{j+1})\}$ and $h^* = f'_{i_{j+1}}$ that the process can be continued to the next step with $f'_{i_{j+1}}$. This implies in particular that $V(f'_{i_1}, \ldots, f'_{i_\rho})$ is a nonempty union of ``coordinate subspaces'' of codimension $\rho$, i.e.\ varieties of the form $x_{k_1} = \cdots = x_{k_\rho} = 0$ for $1 \leq k_1 < \cdots k_\rho \leq n$. Since $(x_1, \ldots, x_n)$ is a local system of coordinates on $U$, and $Z$ is irreducible of codimension $\rho$, there is an open subset $U'$ of $U$ such that $U' \cap Z$ is nonempty and equals the intersection with $U'$ of precisely one such coordinate subspace. All assertions of the claim follow from this observation.
\end{proof}

Consider $k_1, \ldots, k_\rho$ from \Cref{claim:nonzero:nozero'}. Pick $j$, $1 \leq j \leq \rho - 1$, and $i$ such that $i_j < i < i_{j+1}$.

\begin{proclaim} \label{claim:nonzero:nozero''}
There is $j' \leq j$ such that $x_{k_{j'}}$ divides $f_i$.
\end{proclaim}

\begin{proof}
Write $f_i = u_i \prod_k x_k^{\alpha_k}$ where $u_i$ is invertible on $U$. Assume the claim does not hold, so that $\alpha_{k_{j'}} = 0$ for all $j' \leq j$. Since $(x_1, \ldots, x_n)$ form a local system of coordinates everywhere on $U$, it follows that $Z^* := (Z \cap U) \setminus V(x_k: k \neq k_{j'}$ for any $j')$ is nonempty. If $r_1, \ldots, r_n$ and $B$ satisfy property \ref{nonzero:assumption:order} of \Cref{nonzero:assumptions} for $Z^*$ with $\nu \geq 1$, then
\begin{align*}
\ord_B(f_i) &= \sum_{j' > j}\alpha_{k_{j'}}\ord_B(x_{k_{j'}})
    \leq \sum_{j' > j}\alpha_{k_{j'}}\ord_B(f_{i_{j'}})
    \leq (\rho - j)(\max_{j'=j+1}^\rho \alpha_{k_{j'}})r_{i_{j+1}}\ord_B(t)
\end{align*}
However, then $\ord_B(f_i) < r_i\ord_B(t)$ for $\nu \gg 1$, contradicting property \eqref{condition:>nu}. This proves the claim.
\end{proof}

If $U'$ is as in \Cref{claim:nonzero:nozero'}, then $\dist{Z}{j} \cap U' \supseteq V(x_{k_1},\ldots, x_{k_j}) \cap U' \neq \emptyset$ for each $j = 1, \ldots, \rho$ (\Cref{claim:nonzero:nozero',claim:nonzero:nozero''}). This completes the proof of \Cref{prop:nonzero}.
\end{proof}

\subsection{} \label{sec:Z}
We now show how \Cref{prop:nonzero} implies that the ``ordered intersection multiplicity'' can be nonzero for at most finitely many subvarieties of $X$. %
As in \Cref{observation:ord}, let $i^*$ be the smallest index such that $f_{i^*}$ is not identically zero. For each $i^* = i_1 < \cdots < i_\rho \leq n$, where $1\leq \rho \leq n := \dim(X)$, let $Z^{(j)}_{i_1, \ldots, i_\rho}$, $0 \leq j \leq \rho$, be as in \Cref{thm:Z}, i.e.\
\begin{defnlist}
\item $Z^{(0)}_{i_1, \ldots, i_\rho} = X$,
\item for $1 \leq j \leq \rho-1$, $Z^{(j)}_{i_1, \ldots, i_\rho}$ is the union of all irreducible components of $Z^{(j-1)}_{i_1, \ldots, i_\rho} \cap \supp(F_{i_j})$ which are not contained in $\bigcup_{j' >j} \supp(F_{i_{j'}})$, but are contained in $\supp(F_i)$ for each $i$, $i_j < i < i_{j+1}$,
\item $Z^{(\rho)}_{i_1, \ldots, i_\rho}$ is the union of all irreducible components of $Z^{(\rho - 1)}_{i_1, \ldots, i_\rho} \cap \supp(F_{i_\rho})$ which are contained in $\bigcap_{i=1}^n \supp(F_i)$.
\end{defnlist}
We write $\scrZ^{(\rho)}$ for the union of $Z^{(\rho)}_{i_1, \ldots, i_\rho}$ over all choices of $i_1, \ldots, i_\rho$.

\begin{thm} \label{thm:Z-existence}
Assume we are in the set up of \Cref{sec:distinguished-setup}. 
\begin{enumerate}
\item \label{Z-existence:open} For each $\rho$ and each irreducible component $\scrZ^{(\rho)}_k$ of $\scrZ^{(\rho)}$, choose a nonempty open subset $\scrZ^{(\rho, *)}_k$ of $\scrZ^{(\rho)}_k$. Then there is $\nu \in \rr$, and a nonempty open subset $L^*$ of $\prod_i L_i$ such that under \eqref{condition:>nu} for each $(g_1, \ldots, g_n) \in L^*$ and each point $(z,0)$ on $C$ at $t = 0$:
\begin{enumerate}
\item \label{Z-existence:location} $z \in \bigcup_{\rho, k} \scrZ^{(\rho, *)}_k$,
\item \label{Z-existence:nonzero} if $z \in \scrZ^{(\rho, *)}_k$, then there are $i^* =i_1 < \cdots < i_\rho \leq n$, such that $\scrZ^{(\rho)}_k$ is an irreducible component of $Z^{(\rho)}_{i_1, \ldots, i_{\rho}}$, and in addition, $g_i(z) \neq 0$ if and only if $i = i_j$ for some $j$.
\end{enumerate}
\item \label{Z-existence:empty} In particular, given an irreducible subvariety $Z$ of $X$, if $Z$ is not an irreducible component of any $\scrZ^{(\rho)}$, then there is $\nu \in \rr$, a nonempty open subset $Z^*$ of $Z$ and a nonempty open subset $L^*$ of $\prod_i L_i$ such that under \eqref{condition:>nu}, $C \cap (Z^* \times \{0\}) = \emptyset$ for each $(g_1, \ldots, g_n) \in L^*$.
\end{enumerate}
\end{thm}

\begin{prorem} \label{Z-existence:zero-remark}
In the case that $Z$ satisfies (respectively, does not satisfy) the conclusion of assertion \eqref{Z-existence:empty} of \Cref{thm:Z-existence}, we say that the ``ordered intersection multiplicity of $f_1, \ldots, f_n$ at $Z$ is zero (respectively, nonzero).''
\end{prorem}

\subsubsection{Proof of assertion \eqref{Z-existence:empty} of \Cref{thm:Z-existence}}
The second assertion of \Cref{thm:Z-existence} follows from the first. Indeed, given $Z$ as in assertion \eqref{Z-existence:empty}, apply assertion \eqref{Z-existence:open} with
\begin{align*}
\scrZ^{(\rho, *)}_k &:=
    \begin{cases}
      \scrZ^{(\rho)}_k \setminus Z
            & \text{if}\  \scrZ^{(\rho)}_k \supsetneq Z,\\
      \scrZ^{(\rho)}_k
            & \text{otherwise}.
    \end{cases}
\end{align*}
Then assertion \eqref{Z-existence:empty} holds with $Z^*$ being the complement in $Z$ of the union of all $\scrZ^{(\rho)}_k$ such that $\scrZ^{(\rho)}_k \not\supseteq Z$.

\subsubsection{Proof of assertion \eqref{Z-existence:location} of \Cref{thm:Z-existence} - step 1} \label{Z-existence:location:proof:1}
Starting with $Z'_0 := V(f_1, \ldots, f_n)$, inductively define $Z'_i$, $i \geq 1$, as follows: for every irreducible component $Z'_{i-1;j}$ of $Z'_{i-1}$, let
\begin{align*}
Z''_{i-1;j}
    &:=
    \begin{cases}
    Z'_{i-1;j} \setminus \scrZ^{(\rho,*)}_k
        & \text{if}\ Z'_{i-1;j} = \scrZ^{(\rho)}_k,\\
    Z'_{i-1;j}
        & \text{if there are no $\rho, k$ such that}\ Z'_{i-1;j} = \scrZ^{(\rho)}_k.
    \end{cases}
\end{align*}
Then define $Z'_i$ to be the union of $Z''_{i-1;j}$ over all $j$. This process can be repeated only finitely many times, and at the end we arrive at $Z'_{j_1}$ such that
\begin{prooflist}
\item \label{Z-existence:irr} no irreducible component of $Z'_{j_1}$ is an irreducible component of some $\scrZ^{(\rho)}$, and
\item $V(f_1, \ldots, f_n) = Z'_{j_1} \cup \bigcup_{\rho,k} \scrZ^{(\rho,*)}_k$.
\end{prooflist}
Now for each irreducible component $Z'_{j_1;k}$ of $Z'_{j_1}$, apply \Cref{prop:nonzero} with $Z = Z'_{j_1;k}$ and {\em every} possible choice of $i_1, \ldots, i_\rho$, where $\rho := \codim(Z'_{j_1;k})$. Due to property \ref{Z-existence:irr} this yields $\nu'_{j_1;k} \in \rr$ and nonempty open subsets $Z'^*_{j_1;k}$ of $Z'_{j_1;k}$ and $L'^*_{j_1;k}$ of $\prod_i L_i$ such that for each such choice of $i_1, \ldots, i_\rho$, each $(g_1, \ldots, g_n) \in L'^*_{j_1;k}$ and each $r_1, \ldots, r_n$ such that \eqref{condition:>nu} holds with $\nu = \nu'_{j_1;k}$, no $z \in Z'^*_{j_1;k}$ satisfies \eqref{condition:nonzero}. It then follows (due to \Cref{prop:nnonzero-2}) that there are $\nu''_{j_1;k} \geq \nu'_{j_1;k}$ and nonempty open subsets $Z''^*_{j_1;k}$ of $Z'^*_{j_1;k}$ and $L''^*_{j_1;k}$ of $L'^*_{j_1;k}$ such that $(Z''^*_{j_1;k} \times \{0\}) \cap C = \emptyset$ under \eqref{condition:>nu} with $\nu = \nu''_{j_1;k}$ for each $(g_1, \ldots, g_n) \in L''^*_{j_1;k}$. Define
\begin{align*}
Z''^{1,*} &:= \bigcup_k Z''^*_{j_1;k} \\
Z'^2_0 &:= \bigcup_k (Z'_{j_1;k} \setminus Z''^*_{j_1;k})
\end{align*}

\subsubsection{Proof of assertion \eqref{Z-existence:location} of \Cref{thm:Z-existence} - step 2}
Now, as in \Cref{Z-existence:location:proof:1}, for $i \geq 1$, inductively define $Z'^2_i$ from $Z'^2_{i-1}$ by taking each irreducible component $Z'^2_{i-1;j}$ of $Z'^2_{i-1}$ which equals $\scrZ^{(\rho)}_k$ for some $\rho, k$, and replacing it by $Z'^2_{i-1;j} \setminus \scrZ^{(\rho,*)}_k$. This process stops as above with $Z'^2_{j_2}$ such that
\begin{prooflist}
\item no irreducible component of $Z'^2_{j_2}$ is an irreducible component of any $\scrZ^{(\rho)}$, and
\item $V(f_1, \ldots, f_n) = Z'^2_{j_2} \cup Z'^{1,*} \cup \bigcup_{\rho, k} \scrZ^{(\rho, *)}_k$.
\end{prooflist}
As in \Cref{Z-existence:location:proof:1}, we apply \Cref{prop:nonzero} and \Cref{prop:nnonzero-2} to each irreducible component $Z'^2_{j_2;k}$ of $Z'^2_{j_2}$ to obtain a nonempty open subset $Z''^{2,*}_{j_2;k}$ of $Z'^2_{j_2;k}$ and $L''^{2,*}_{j_2;k}$ of $\prod_i L_i$ such that under \eqref{condition:>nu} for some $\nu \gg 1$, $(Z''^{2,*}_{j_2;k} \times \{0\}) \cap C = \emptyset$ for each $(g_1, \ldots, g_n) \in L''^{2,*}_{j_2;k}$. Define
\begin{align*}
Z''^{2,*} &:= \bigcup_k Z''^{2,*}_{j_2;k} \\
Z'^3_0 &:= \bigcup_k (Z'^2_{j_2;k} \setminus Z''^{2,*}_{j_2;k})
\end{align*}
and continue as above. Since $\dim(Z'^{i+1}_0) < \dim(Z'^{i})$ for each $i$, this process stops after finitely many steps and we end up with a decomposition
\begin{align*}
V(f_1, \ldots, f_n)
    &= Z''^{1,*} \cup \cdots \cup Z''^{j,*} \cup \bigcup_{\rho, k} \scrZ^{(\rho, *)}_k
\end{align*}
and an open subset $L^* := \bigcap_{i,k} L''^{i,*}_{j_i;k}$ of $\prod_i L_i$ such that under \eqref{condition:>nu} for a sufficiently large $\nu$,
\begin{align*}
(\bigcup_i Z''^{i,*} \times \{0\}) \cap C = \emptyset
\end{align*}
for each $(g_1, \ldots, g_n) \in L^*$. This immediately implies assertion \eqref{Z-existence:location} of \Cref{thm:Z-existence}.

\subsubsection{Proof of assertion \eqref{Z-existence:nonzero} of \Cref{thm:Z-existence}}
For each $\scrZ^{(\rho)}_k$, consider the collection $\scrI^{(\rho)}_k$ of all $i^* = i_1 < \cdots < i_\rho \leq n$ such that
\begin{itemize}
\item either $Z^{(\rho)}_{i_1, \ldots, i_{\rho}} = \emptyset$, or
\item $Z^{(\rho)}_{i_1, \ldots, i_{\rho}} \neq \emptyset$, but $\scrZ^{(\rho)}_k$ is {\em not} an irreducible component of $Z^{(\rho)}_{i_1, \ldots, i_\rho}$.
\end{itemize}
For each such choice of $i_1, \ldots, i_\rho$, due to \Cref{prop:nonzero} there is a nonempty open subset $\scrZ'^{(\rho)}_k(i_1, \ldots, i_\rho)$ of $\scrZ^{(\rho)}_k$, and a nonempty open subset $L'^{(\rho, *)}_k(i_1, \ldots, i_\rho)$ of $\prod_i L_i$ such that the following holds under \eqref{condition:>nu} for sufficiently large $\nu$:
\begin{prooflist}
\item \label{Z-existence:nonzero:zero} for each $(g_1, \ldots, g_n) \in L'^{(\rho, *)}_k(i_1, \ldots, i_\rho)$, if $z \in \scrZ'^*_k(i_1, \ldots, i_\rho)$ is such that $(z,0) \in C$, then there is $j$ such that $g_{i_j}(z) = 0$.
\end{prooflist}
Moreover, due to \Cref{prop:nnonzero-2} there is a nonempty open subset $\scrZ'^{(\rho, *)}_k$ of $\scrZ^{(\rho)}_k$, and a nonempty open subset $L'^{(\rho, *)}_k$ of $\prod_i L_i$ such that the following holds under \eqref{condition:>nu} for sufficiently large $\nu$:
\begin{prooflist}[resume]
\item \label{Z-existence:nonzero:nonzero} for each $(g_1, \ldots, g_n) \in L'^{(\rho, *)}_k$, if $z \in \scrZ'^{(\rho, *)}_k$ is such that $(z,0) \in C$, then there are $i^* = i_1 < \cdots < i_\rho \leq n$ such that $g_i(z) \neq 0$ if and only if $i = i_j$ for some $j$.
\end{prooflist}
Now let
\begin{align*}
\scrZ''^{(\rho, *)}_k &:= \scrZ^{(\rho, *)}_k \cap \scrZ'^{(\rho, *)}_k \cap \bigcap_{(i_1, \ldots, i_\rho) \in \scrI^{(\rho)}_k} \scrZ'^{(\rho, *)}_k(i_1, \ldots, i_\rho)
\end{align*}
where $\scrZ^*_k$ are the open subsets given in assertion \eqref{Z-existence:open} of \Cref{thm:Z-existence}. Applying the arguments of the proof of assertion \eqref{Z-existence:location} with $\scrZ''^{(\rho, *)}_k$ instead of $\scrZ^{(\rho, *)}_k$, we obtain an open subset $L'^*$ of $\prod_i L_i$ such that for sufficiently large $\nu$ the following holds under \eqref{condition:>nu}:
\begin{prooflist}[resume]
\item \label{Z-existence:nonzero:location} for each $(g_1, \ldots, g_n) \in L'^*$, if $(z,0) \in C$ at $t = 0$, then $z \in \bigcup_{\rho, k} \scrZ''^{(\rho, *)}_k$.
\end{prooflist}
Let
\begin{align*}
L''^* &:= L'^* \cap \bigcap_{\rho, k} L'^{(\rho, *)}_k \cap \bigcap_{\rho, k,i_1, \ldots, i_\rho} L'^{(\rho, *)}_k(i_1, \ldots, i_\rho)
\end{align*}
where the $\rho, k$ vary over all irreducible components $\scrZ^{(\rho)}_k$ of all $\scrZ^{(\rho)}$, and $(i_1, \ldots, i_\rho)$ vary over $\scrI^{(\rho)}_k$.
Then property \ref{Z-existence:nonzero:location} implies that the following holds under \eqref{condition:>nu} for sufficiently large $\nu$: for each $(g_1, \ldots, g_n) \in L''^*$, if $(z,0) \in C$ at $t = 0$, then $z \in \bigcup_k \scrZ''^{(\rho, *)}_k$. Moreover, for each such $z \in \scrZ''^{(\rho, *)}_k$, property \ref{Z-existence:nonzero:nonzero} implies that there are $i^* = i_1 < \cdots < i_\rho \leq n$ such that $g_i(z) \neq 0$ if and only if $i = i_j$ for some $j$. However, then $Z$ must be an irreducible component of $Z^{(\rho)}_{i_1, \ldots, i_\rho}$ due to property \ref{Z-existence:nonzero:zero}. This completes the proof of assertion \eqref{Z-existence:nonzero}, and consequently, the proof of \Cref{thm:Z-existence}.
%
\qed
%

\section[Guessing the formula for the ordered intersection multiplicity]{Combinatorics of intersections of divisors - guessing the formula for the ordered intersection multiplicity} \label{sec:combinatorics}
\subsection{} \label{sec:comb:step:0} Consider effective Cartier divisors $F_1, \ldots, F_n$ on a variety $X$ of pure dimension $n$. We allow for the scenario that $\bigcap \supp(F_i)$ has nonzero dimension, and try to express the intersection product $\mult{F_1}{F_n}$ as a sum of contributions from the ``distinguished components'', i.e.\ the irreducible components of $Z^{(\rho)}_{i_1, \ldots, i_\rho}$ from \Cref{thm:Z}. Recall from the definition of the $Z^{(\rho)}_{i_1, \ldots, i_\rho}$ that for all such components, $i_1$ takes a fixed value, which we denoted as $i^*$. Moreover, $Z^{(\rho)}_{i_1, \ldots, i_\rho}$ is constructed via an inductive step starting from irreducible components of $\supp(F_{i_1})$ which are contained in $\supp(F_{i_2})$, but {\em not} contained in $\supp(F_{i})$ for any $i$, $i_1 < i < i_2$; in other words, for each such component $Z$ of $\supp(F_{i_1})$, $i_2$ is the {\em smallest} value of $i$ such that $Z \not\subseteq \supp(F_i)$.

\subsection{} \label{sec:comb:step:1} Motivated by the observations in the preceding paragraph, we do the following:
\begin{prooflist}
\item Fix $i_1 := i^*$, $1 \leq i^* \leq n$. Denote the irreducible components of $F_{i_1}$ as $F_{i_1;j}$, so that as a Weil divisor $F_{i_1}$ is of the form $\sum_j m_{i_1;j} F_{i_1;j}$.
\item For each $j$ we consider the smallest index $k_j > i_1$ such that $F_{i_1;j} \not\subseteq \supp(F_{k_j})$
\end{prooflist}
Then
\begin{align*}
\mult{F_1}{F_n}
    &= \mult{F_{i_1}}{F_n, F_1, \ldots, F_{i_1-1}}
    = \mult{\sum_j m_{i_1;j} F_{i_1;j},F_2}{F_n, F_1, \ldots, F_{i_1-1}} \\
    &= \sum_j m_{i_1;j} \mult{F_{i_1;j}, F_{k_j}, F_{k_j+1}, \ldots, F_n, F_1, \ldots, F_{i_1-1}, F_2}{F_{k_j -1}}
\end{align*}
In line with the construction of the $Z^{(j)}_{i_1, \ldots, i_\rho}$ from \Cref{thm:Z}, we write $i_2 = k_j$, and rearrange the terms by fixing first a value of $i_2$ and summing over all $j$ such that $k_j = i_2$ to obtain
\begin{align*}
\mult{F_1}{F_n}
    &= \sum_{i_2}\sum_{k_j = i_2} m_{1;j} \mult{F_{i_1;j}, F_{i_2}, F_{i_2+1}, \ldots, F_n, F_1, \ldots, F_{i_1-1}, F_{i_1+1}}{F_{i_2 -1}} \\
    &= \sum_{i_2} \mult{F_{i_1;\Sigma i_2}, F_{i_2}, F_{i_2+1}, \ldots, F_n, F_1, \ldots, F_{i_1-1}, F_{i_1+1}}{F_{i_2 -1}}
\end{align*}
where $F_{i_1;\Sigma i_2}$ is the Weil divisor $\sum_{j:k_j = i_2}m_{1;j}F_{i_1;j}$.

\subsection{} \label{sec:comb:step:2}
Note that each irreducible component of $\supp(F_{i_1;\Sigma i_2}) \cap \supp(F_{i_2})$ has codimension two, so that the intersection
\begin{align*}
F_{i_1, i_2} := F_{i_1;\Sigma i_2} \cap F_{i_2}
\end{align*}
can be represented as a codimension two cycle $\sum_j m_{i_1,i_2;j}F_{i_1,i_2;j}$ consisting of linear combination of irreducible components of $\supp(F_{i_1;\Sigma i_2}) \cap \supp(F_{i_2})$. Then
\begin{align*}
\mult{F_1}{F_n}
    &= \sum_{i_2} \mult{F_{i_1,i_2}, F_{i_2+1}, \ldots, F_n, F_1, \ldots, F_{i_1-1}, F_{i_1+1}}{F_{i_2 -1}} \\
    &= \sum_{i_2, j} m_{i_1,i_2;j} \mult{F_{i_1,i_2;j}, F_{i_2+1}, \ldots, F_n, F_1, \ldots, F_{i_1-1}, F_{i_1+1}}{F_{i_2 -1}}
\end{align*}
We now apply the same procedure as in \Cref{sec:comb:step:1}: for each $j$, write
\begin{align*}
&\mult{F_{i_1,i_2;j}, F_{i_2+1}, \ldots, F_n, F_1, \ldots, F_{i_1-1}, F_{i_1+1}}{F_{i_2 -1}} \\
&\quad\qquad = \mult{F_{i_1,i_2;j}, F_{k'_j}, F_{k'_j + 1}, \ldots, F_n, F_1, \ldots, F_{i_1-1}, F_{i_1+1}, \ldots, F_{i_2-1}, F_{i_2+1}} {F_{k'_j -1}}
\end{align*}
where $k'_j$ is the smallest index such that $F_{i_1,i_2;j} \not\subseteq \supp(F_{k'_j})$. Then fix first a value of $i_3$ and sum over all $j$ such that $k'_j = i_3$ to obtain
\begin{align*}
& \mult{F_1}{F_n} \\
&\qquad = \sum_{i_2,i_3}
    \mult{F_{i_1,i_2;\Sigma i_3}, F_{i_3}, F_{i_3+1}, \ldots, F_n, F_1, \ldots, F_{i_1-1}, F_{i_1+1}, \ldots, F_{i_2 -1}, F_{i_2+1}} {F_{i_3-1}} \\
&\qquad = \sum_{i_2,i_3}
    \mult{F_{i_1,i_2,i_3}, F_{i_3+1}, \ldots, F_n, F_1, \ldots, F_{i_1-1}, F_{i_1+1}, \ldots, F_{i_2 -1}, F_{i_2+1}} {F_{i_3-1}}
\end{align*}
where
\begin{align*}
F_{i_1,i_2;\Sigma i_3}
    &:= \sum_{k'_j = i_3} m_{i_1,i_2;j} F_{i_1,i_2;j},\ \text{and} \\
F_{i_1,i_2,i_3}
    &:= F_{i_1,i_2;\Sigma i_3} \cap F_{i_3}
\end{align*}

\subsection{} \label{sec:comb:step:n}
Repeating this process as many times as possible we obtain
\begin{align*}
\mult{F_1}{F_n}
    &= \sum_{\rho=1}^n \sum_{i^* = i_1 < \cdots < i_\rho \leq n}
        \mult{F_{i_1, \ldots, i_\rho}, F_{i'_1}}{F_{i'_{n-\rho}}} \\
    &= \sum_{\rho=1}^n \sum_{i^* = i_1 < \cdots < i_\rho \leq n} \sum_j
        m_{i_1, \ldots, i_\rho;j}
        \mult{F_{i_1, \ldots, i_\rho;j}, F_{i'_1}}{F_{i'_{n-\rho}}}
\end{align*}
where $1 \leq i'_1 < \cdots < i'_{n-\rho} \leq n$ are the elements of $\{1, \ldots, n\} \setminus \{i_1, \ldots, i_\rho\}$. It follows from the construction of the $F_{i_1, \ldots, i_\rho;j}$ that for each $j$,
\begin{prooflist}
\item \label{comb:step:n:subset} $F_{i_1, \ldots, i_\rho; j} \subseteq \bigcap_{j'=1}^{n-\rho} \supp(F_{i'_{j'}})$ for each $j$, and
\item \label{comb:step:n:codim} $\codim(F_{i_1, \ldots, i_\rho;j}) = \rho$.
\end{prooflist}
Due to observation \ref{comb:step:n:codim}, we can rearrange the terms of the above sum to obtain
\begin{align*}
\mult{F_1}{F_n}
    &= \sum_Z \sum_{F_{i_1, \ldots, i_\rho;j} = Z}
        m_{i_1, \ldots, i_\rho;j}
        \deg_{i'_1, \ldots, i'_{n-\rho}}(Z)
\end{align*}
where $\deg_{i'_1, \ldots, i'_{n-\rho}}(Z) := \mult{\scrL_{i'_1}|_Z}{\scrL_{i'_{n-\rho}}|_Z}$, the outer sum is over all irreducible subvarieties $Z$ of $X$, and the inner sum is over all $\rho, i_1, \ldots, i_\rho, j$ such that $\rho = \codim(Z)$ and $Z = F_{i_1, \ldots, i_\rho;j}$.

\subsection{} \label{sec:comb:Z} We claim that $F_{i_1, \ldots, i_\rho;j}$ from \Cref{sec:comb:step:n} are precisely the irreducible components of $Z^{(\rho)}_{i_1, \ldots, i_\rho}$ from \Cref{thm:Z}. Indeed, recall the definition of $Z^{(j)}_{i_1, \ldots, i_\rho}$:

\begin{defnlist}
\item $Z^{(0)}_{i_1, \ldots, i_\rho} = X$,
\item for $1 \leq j \leq \rho-1$, $Z^{(j)}_{i_1, \ldots, i_\rho}$ is the union of all irreducible components of $Z^{(j-1)}_{i_1, \ldots, i_\rho} \cap \supp(F_{i_j})$ which are not contained in $\bigcup_{j' >j} \supp(F_{i_{j'}})$, but are contained in $\supp(F_i)$ for each $i$, $i_j < i < i_{j+1}$,
\item $Z^{(\rho)}_{i_1, \ldots, i_\rho}$ is the union of all irreducible components of $Z^{(\rho-1)}_{i_1, \ldots, i_\rho} \cap \supp(F_{i_\rho})$ which are contained in $\bigcap_{i=1}^n \supp(F_i)$.
\end{defnlist}
It is clear from the definition that $\supp(F_{i_1; \Sigma i_2})$ is the union of $Z^{(1)}_{i_1, \ldots, i_\rho}$ and a (possibly empty) collection of irreducible components of $\supp(F_{i_1})$ which are contained in $\bigcup_{j' = 2}^\rho \supp(F_{i_{j'}})$. By an easy induction then it follows that $\supp(F_{i_1, \ldots, i_{k-1}; \Sigma i_k})$ is the union of $Z^{(k-1)}_{i_1, \ldots, i_\rho}$ and a (possibly empty) collection of irreducible components which are contained in $\bigcup_{j' = k+1}^\rho \supp(F_{i_{j'}})$. However, this means that $\supp(F_{i_1, \ldots, i_{\rho-1}; \Sigma i_\rho})$ is precisely $Z^{(\rho-1)}_{i_1, \ldots, i_\rho}$. This in turn implies (due to observation \ref{comb:step:n:subset} from \Cref{sec:comb:step:n}) that $Z^{(\rho)}_{i_1, \ldots, i_\rho}$ is precisely the union of the $F_{i_1, \ldots, i_{\rho}; j}$, as claimed.

\subsection{} \label{sec:comb:m} Now we compute the numbers $m_{i_1, \ldots, i_\rho;j}$ from \Cref{sec:comb:step:n}. Take an affine open subset $U$ of $X$ such that $U$ intersects $F_{i_1, \ldots, i_\rho;j}$, and each $F_i$ is the divisor of some $f_i \in \kk[U]$. Recall the definition of $I^{(k)}_{i_1, \ldots, i_\rho}$ from \Cref{thm:length}:
\begin{defnlist}
\item $I^{(1)}_{i_1, \ldots, i_\rho}$ is the principal ideal generated by $f_{i_1}$,
\item for $2 \leq k \leq \rho$, $I^{(k)}_{i_1, \ldots, i_\rho}$ is the ideal generated by $f_{i_k}$ and the intersection of all primary components $\qqq$ of $I^{(k-1)}_{i_1, \ldots, i_\rho}$ such that the corresponding prime ideal $\sqrt{\qqq}$ contains $f_i$ for each $i < i_j$, but does {\em not} contain $f_{i_k}$.
\end{defnlist}
It is straightforward to check from the definitions and \Cref{prop:length_f} that for each $k \leq \rho$, the subscheme of $U$ defined by $I^{(k)}_{i_1, \ldots, i_\rho}$ is precisely the cycle $F_{i_1, \ldots, i_k}$ constructed via the inductive procedure in \Crefrange{sec:comb:step:1}{sec:comb:step:n}. In particular, $m_{i_1, \ldots, i_\rho;j}$ is precisely the multiplicity of $Z = F_{i_1, \ldots, i_\rho;j}$ in $V(I^{(\rho)}_{i_1, \ldots, i_\rho})$. Consequently,
\begin{align}
m_{i_1, \ldots, i_\rho;j}
    &= \len_{\local{V(I^{(\rho)}_{i_1, \ldots, i_\rho})}{Z}}(\local{V(I^{(\rho)}_{i_1, \ldots, i_\rho})}{Z})
    = \len(\local{X}{Z}/I^{(\rho)}_{i_1, \ldots, i_\rho}\local{X}{Z})
    \label{eq:comb:m}
\end{align}
We summarize the observations from \Crefrange{sec:comb:step:n}{sec:comb:m} in the following theorem:

\begin{thm} \label{thm:comb:intersection}
Let $F_1, \ldots, F_n$ be effective Cartier divisors on a variety $X$ of pure dimension $n$. Fix $i^*$, $1 \leq i^* \leq n$. Then
\begin{align*}
\mult{F_1}{F_n}
    &= \sum_Z \sum_{i_1, \ldots, i_\rho}
        \len(\local{X}{Z}/I^{(\rho)}_{i_1, \ldots, i_\rho}\local{X}{Z})
        \deg_{i'_1, \ldots, i'_{n-\rho}}(Z)
\end{align*}
where the outer sum is over all irreducible subvarieties $Z$ of $X$, the inner sum is over all $i_1, \ldots, i_\rho$ where $\rho := \codim(Z)$ and $i^* = i_1 < \cdots < i_\rho \leq n$ are such that $Z$ is an irreducible component of $Z^{(\rho)}_{i_1, \ldots, i_\rho}$ from \Cref{sec:comb:Z}, the $I^{(\rho)}_{i_1, \ldots, i_\rho}$ are as in \Cref{sec:comb:m}, and finally, $1 \leq i'_1 < \cdots < i'_{n-\rho} \leq n$ are the elements of $\{1, \ldots, n\} \setminus \{i_1, \ldots, i_\rho\}$, and $\deg_{i'_1, \ldots, i'_{n-\rho}}(Z)$ is the number (counted with appropriate multiplicities) of points in the intersection of $Z$ and generic Cartier divisors linearly equivalent to $F_{i'_j}$, $j = 1, \ldots, n-\rho$. \qed
\end{thm}

\subsection{} \label{sec:comb:guess}
Based on \Cref{thm:comb:intersection} we may guess that the ordered intersection multiplicity of $F_1, \ldots, F_n$ along an irreducible subvariety $Z$ equals
\begin{align*}
    \sum_{i_1, \ldots, i_\rho}
        \len(\local{X}{Z}/I^{(\rho)}_{i_1, \ldots, i_\rho}\local{X}{Z})
        \deg_{i'_1, \ldots, i'_{n-\rho}}(Z)
\end{align*}
where $\rho := \codim(Z)$, and the sum is over all $i_1, \ldots, i_\rho$ such that $Z$ is an irreducible component of $Z^{(\rho)}_{i_1, \ldots, i_\rho}$. In the next section we show that this guess is indeed correct.

\section{Proof of the formula for ordered intersection multiplicity} \label{sec:formula}

The main result of this section is the computation of the formula of the ordered intersection multiplicity (\Cref{thm:length:confirmation}). All our main results on ordered intersection multiplicity, namely \Cref{thm:defn,thm:Z,thm:length}, are straightforward consequences of \Cref{thm:length:confirmation} combined with \Cref{thm:Z-existence,thm:comb:intersection}.

\subsection{} \label{sec:formula:set-up:0}
We continue with the set up from \Cref{sec:distinguished-setup}. In addition 
we assume that for each $i$,
\begin{enumerate}
\item $L_i$ is a base-point free (finite dimensional) vector subspace of $\kk[X]$ containing $f_i$.
\end{enumerate}
By $\bar X$ we denote a compactification\footnote{i.e.\ a complete variety containing $X$ as a dense subset.} of $X$ such that
\begin{enumerate}[resume]
\item \label{length:confirmation:assumption:Xbar} there are linear systems $\scrL_i$ on $\bar X$ such that $L_i = \scrL_i|_X$, $i = 1, \ldots, n$; in particular, for each $i$, either $f_i \equiv 0$, or there is a Cartier divisor $F_i \in \scrL_i$ such that $f_i$ is a local representative of $F_i$ on $X$.
\end{enumerate}
(For property \eqref{length:confirmation:assumption:Xbar} one can e.g.\ take an arbitrary completion $\bar X_0$ of $X$, and then define $\bar X$ to be the closure in $\bar X_0 \times \prod_{i=1}^n \pp^{n_i}$ where $d_i := \dim(L_i)$, of the graph of the morphism $x \mapsto (x, [1: g_{1,1}(x): \cdots : g_{1, d_1}(x)], \ldots, [1: g_{n,1}(x): \cdots : g_{n, d_n}(x)])$, where $g_{i,1}, \ldots, g_{i,d_i}$ form a basis of $L_i$.)

\subsection{} \label{sec:formula:set-up:1}
Consider subvarieties $Z^{(\rho')}_{i_1, \ldots, i_{\rho'}}$ defined in the same way as in \Cref{thm:Z}, but with $\bar X$ in place of $X$, and as usual, let $\scrZ^{(\rho')}$ be the union of $Z^{(\rho')}_{i_1, \ldots, i_{\rho'}}$ over all $i_1, \ldots, i_{\rho'}$, and let $\scrZ^{(\rho')}_k$, $k = 1, 2, \ldots$, be the irreducible components of $\scrZ^{(\rho')}$. Pick nonempty Zariski open subsets $\scrZ^{(\rho', *)}_k$ of $\scrZ^{(\rho')}_k$ such that
\begin{enumerate}[resume]
\item $\scrZ^{(\rho'_1, *)}_{k_1} \cap  \scrZ^{(\rho'_2, *)}_{k_2} = \emptyset$ unless $\rho'_1 = \rho'_2$ and $k_1 = k_2$.
\end{enumerate}
Due to \Cref{thm:Z-existence} we can pick $\nu^*_1 \in \rr$ and a nonempty Zariski open subset $L^*_1$ of $\prod_i L_i$ such that if \eqref{condition:>nu} holds with $\nu \geq \nu^*_1$, then for each $(g_1, \ldots, g_n) \in L^*_1$ and each point $(z,0)$ on $C = C(g_1, \ldots, g_n; r_1, \ldots, r_n)$ at $t = 0$, there are unique integers $\rho_z$, $k_z$, and $i^* =i_1(z) < \cdots < i_{\rho_z}(z) \leq n$, and $h_{\rho_z, k_z, i} \in \kk[X]$, $i = 1, \ldots, n$, such that
\begin{enumerate}[resume]
\item $z \in \scrZ^{(\rho_z, *)}_{k_z}$,
\item $\scrZ^{(\rho_z)}_{k_z}$ is an irreducible component of $Z^{(\rho_z)}_{i_1(z), \ldots, i_{\rho_z}(z)}$,
\item \label{length:confirmation:assumption:F_i} there is a neighborhood of $\scrZ^{(\rho_z, *)}_{k_z}$ in $\bar X$ such that for each $i = 1, \ldots, n$,
\begin{enumerate}
\item $F_i$ is locally represented by $f_i/h_{\rho_z, k_z, i}$,
\item $h/h_{\rho_z, k_z, i}$ is a regular function for each $h \in L_i$,
\end{enumerate}
\item on a punctured neighborhood of $(z,0)$ in $\bar X \times \kk$, $C$ is defined by $f_i/h_{\rho_z, k_z, i} - t^{r_i}g_i/h_{\rho_z, k_z, i} = 0$, $i = 1, \ldots, n$,
\item \label{length:confirmation:assumption:nonzero} $(g_i/h_{\rho_z, k_z, i})(z) \neq 0$ if and only if $i = i_j(z)$ for some $j. $
\end{enumerate}
Due to property \eqref{length:confirmation:assumption:Xbar} (from \Cref{sec:formula:set-up:0}) we may assume that
\begin{enumerate}[resume]
\item \label{length:confirmation:assumption:denom} if $\scrZ^{(\rho')}_{k}$ is the closure of a subvariety of $X$, then $\scrZ^{(\rho', *)}_{k} \subseteq X \cap Z$ and $h_{\rho', k, i} = 1$, $i = 1, \ldots, n$.
\end{enumerate}
Restricting $L^*_1$ to an appropriate open subset if necessary, we may further ensure that
\begin{enumerate}[resume]
\item \label{length:confirmation:assumption:finite} for each $\scrZ^{(\rho')}_{k}$ which is an irreducible component of $Z^{(\rho')}_{i_1, \ldots, i_{\rho'}}$,
\begin{align*}
|\scrZ^{(\rho', *)}_{k} \cap V(\frac{g_{i'_1}}{h_{\rho', k, i'_1}}, \ldots, \frac{g_{i'_{n-\rho'}}}{h_{\rho', k, i'_{n-\rho}}})| &< \infty
\end{align*}
where $i'_1, \ldots, i'_{n-\rho'}$ are the elements in the complement of $\{1, \ldots, n\} \setminus \{i_1, \ldots, i_{\rho'}\}$. Define
\begin{align*}
\scrZ'_0 := \bigcup_{\rho', k}
    \left(
    \scrZ^{(\rho', *)}_{k} \cap
        V(\frac{g_{i'_1}}{h_{\rho', k, i'_1}}, \ldots, \frac{g_{i'_{n-\rho'}}}{h_{\rho', k, i'_{n-\rho}}})
    \right)
\end{align*}
\end{enumerate}

\subsection{} \label{sec:formula:blow-up}
Given $\vec g = (g_1, \ldots, g_n) \in \prod_i L_i$, we write below $C(\vec g)$ for $C$ to explicitly mention the dependence on $\vec g$. For $\vec a = (a_1, \ldots, a_n) \in \kk^n$, we write $\vec{ag} := (a_1g_1, \ldots, a_ng_n)$. We now introduce a birational modification $\tilde X$ of $\bar X \times \kk$ on which the image of $C(\vec{ag})$ can be explicitly described for generic $\vec a \in \kk^n$. Indeed, let $W := \pp^n(r_1, \ldots, r_n, 1)$ be the $n$-dimensional weighted projective space with weights $(r_1, \ldots, r_n, 1)$, and $\tilde X$ be the closure in $\bar X \times \kk \times W$ of the graph of the map
\begin{align*}
\bar X \times \kk \ni (x, t) &\mapsto [\frac{f_1}{g_1}(x): \cdots : \frac{f_n}{g_n}(x): t] \in W
\end{align*}
Given $\vec a \in (\kk \setminus \{0\})^n$, let $\tilde C(\vec{ag})$ be the ``scheme-theoretic strict transform''\footnote{The scheme-structure on $\tilde C(\vec{ag})$ is given as follows: near all but finitely many points of $C(\vec{ag})$, $X$ is locally isomorphic to $\tilde X$, so that the scheme-structure on $C(\vec{ag})$ can be uniquely transferred to $\tilde C(\vec{ag})$. Over the remaining points the scheme-structure can be extended by taking the scheme-theoretic closure.\label{footnote:scheme-transform}} of $C(\vec{ag})$ on $\tilde X$. If $(z,0) \in C(\vec{ag})$ and $\tilde z \in \tilde C(\vec{ag})$ is a pre-image of $(z,0)$, then it is straightforward to see that $\tilde z \in \tilde U_{n+1} := \pi_W^{-1}(W \setminus V(w_{n+1}))$, where $\pi_W: \tilde X \to W$ is the natural projection and $[w_1: \cdots :w_{n+1}]$ are (weighted) homogeneous coordinates on $W$. Note that $u_i := w_i/(w_{n+1})^{r_i}$, $i = 1, \ldots, n$, are well-defined regular functions on $\tilde U_{n+1}$.

\begin{proclaim} \label{prop:formula:blow-up}
Consider the set up in \Cref{sec:formula:set-up:0,sec:formula:set-up:1}. Pick $\nu \geq \nu^*_1$ and $r_1, \ldots, r_n$ satisfying \eqref{condition:>nu}, and $\vec g \in L^*_1$ (where $\nu^*_1$ and $L^*_1$ are from \Cref{sec:formula:set-up:1}). There is a nonempty Zariski open subset $\scrA^*$ of $(\kk \setminus \{0\})^n$ such that for each $\vec a = (a_1, \ldots, a_n) \in \scrA^*$,
\begin{enumerate}
\item \label{formula:blow-up:set} $C(\vec{ag}) \cap (\bar X \times \{0\})$ is set-theoretically the same as $\scrZ'_0 \times \{0\}$ (where $\scrZ'_0$ is from property \eqref{length:confirmation:assumption:finite} of \Cref{sec:formula:set-up:1}),
\item \label{formula:blow-up:eq} $\tilde C(\vec{ag})$ is scheme-theoretically equal to $V(u_1 - a_1, \ldots, u_n - a_n)$ near $\pi^{-1}(X \times \{0\}) \cap \tilde U_{n+1}$ (where $\pi: \tilde X \to X \times \kk$ is the projection map).
\end{enumerate}
Let $\tilde \scrE_0$ be the (finite) collection of $n$-dimensional subvarieties $E$ of $\pi^{-1}(X \times \{0\})$ such that $E \cap \tilde U_{n+1} \neq \emptyset$, and $\pi_W|_E: E \to W$ is generically finite. Then
\begin{enumerate}[resume]
\item \label{formula:blow-up:ord} for each $\vec a = (a_1, \ldots, a_n) \in \scrA^*$ and for each $z \in \scrZ'_0$,
\begin{align*}
\len(\local{C(\vec {ag})}{(z,0)}/t\local{C(\vec {ag})}{(z,0)})
    &=
    \sum_{E \in \tilde \scrE_0}
    \ord_E(t)\deg(\pi_W|_E)
\end{align*}
\end{enumerate}
\end{proclaim}

\begin{proof}
Since $L^*_1$ is open (and nonempty), there is a nonempty open subset $\scrA^*_1$ of $(\kk \setminus \{0\})^n$ such that $\vec{ag}\in L^*_1$ for all $\vec a \in \scrA^*_1$. Assertion \eqref{formula:blow-up:set} follows (with $\scrA^* = \scrA^*_1$) from properties \eqref{length:confirmation:assumption:nonzero} and \eqref{length:confirmation:assumption:finite} of $L^*_1$ from \Cref{sec:formula:set-up:1}. Since $\pi_W|_{\tilde U_{n+1}}$ is defined by $(u_1, \ldots, u_n)$, by definition of $\tilde \scrE_0$ there is a nonempty open subset $\scrA^*$ of $\scrA^*_1$ such that for each $\vec a \in \scrA^*$, the subscheme $\tilde C_{\vec a} := V(u_1 - a_1, \ldots, u_n - a_n)$ of $\tilde U_{n+1}$
\begin{prooflist}
\item intersects each $E \in \tilde \scrE_0$ at finitely many points,
\item is purely one dimensional near every such point,
\item does {\em not} intersect any points in $E \cap E'$ for any $E \neq E' \in \tilde \scrE_0$,
\item does {\em not} intersect any irreducible component of $\pi^{-1}(X \times \{0\})$ other than those in $\tilde \scrE_0$.
\end{prooflist}
For each $\vec a \in \scrA^*$ and each $z \in X$, it follows that $\tilde C(\vec{ag}) = \tilde C_{\vec a}$ near $\pi^{-1}(z,0) \cap \tilde U_{n+1}$, and
\begin{align*}
\len(\local{C(\vec {ag})}{(z,0)}/t\local{C(\vec {ag})}{(z,0)})
    &=
    \sum_{\tilde z \in \pi^{-1}(z,0)}
    \len(\local{\tilde C(\vec {ag})}{\tilde z}/t\local{\tilde C(\vec {ag})}{\tilde z}) \\
    &=
    \sum_{E \in \tilde \scrE_0}
    \ord_E(t)\deg((\pi_W)|_E)
\end{align*}
which completes the proof of \Cref{prop:formula:blow-up}.
\end{proof}

\subsubsection{} \label{sec:length:confirmation:constant-m}
Let $\nu, r_1, \ldots, r_n, \vec g, \scrA^*$ be as in \Cref{prop:formula:blow-up}. For a positive integer $m$, define
\begin{align*}
\scrA^*_m
    := \{\vec a \in \scrA^* :~
    & \vec{\mu a} := (\mu_1 a_1, \ldots, \mu_n a_n) \in \scrA^*\ \text{for all}\ \vec{\mu} := (\mu_1, \ldots, \mu_n)\\
    & \text{such that}\ \mu_1^m = \cdots = \mu_n^m = 1\}
\end{align*}
Then $\scrA^*_m$ is a nonempty Zariski open subset of $\scrA^*$. Pick $\vec a \in \scrA^*_m$ and $z \in \bar X$ such that $(z,0) \in C(\vec{ag})$. For convenience write $\rho' = \rho_z$, $k = k_z$ and $i_j = i_j(z)$, $j = 1, \ldots, \rho'$ (where $\rho_z, k_z, i_j(z)$ are as in \Cref{sec:formula:set-up:1}). Let $i'_1 < \cdots < i'_{n-\rho'}$ be the elements of $\{1, \ldots, n\}\setminus \{i_1, \ldots, i_{\rho'}\}$.

\begin{procor} \label{prop:length:confirmation:constant-m}
Let $C_{m, i_1, \ldots, i_{\rho'}}(\vec{ag};\vec r)$ (or in short, $C_{m, i_1, \ldots, i_{\rho'}}(\vec{ag})$) be the scheme-theoretic closure in $X \times \kk$ of the subscheme of $X \times (\kk \setminus \{0\})$ defined by the following equations:
\begin{alignat*}{2}
&f_{i_j} - a_{i_j}t^{r_{i_j}}g_{i_j}
    = 0,\quad
    &&j = 1, \ldots, \rho', \\
&(f_{i'_j})^m - (a_{i'_j})^m t^{mr_{i'_j}} (g_{i'_j})^m
    = 0,\quad
    &&j = 1, \ldots, n - \rho',
\end{alignat*}
where $i'_1 < \cdots < i'_{n-\rho'}$ are the elements of $\{1, \ldots, n\}\setminus \{i_1, \ldots, i_{\rho'}\}$. 
Then
\begin{align*}
\len(\local{C(\vec {ag})}{(z,0)}/t\local{C(\vec {ag})}{(z,0)})
    &=  \frac{1}{m^{n-\rho'}}
        \len(\local{C_{m, i_1, \ldots, i_{\rho'}}(\vec{ag})}{(z,0)}/t\local{C_{m, i_1, \ldots, i_{\rho'}}(\vec{ag})}{(z,0)})
\end{align*}
\end{procor}

\begin{proof}
Let $\scrM_{i_1, \ldots, i_{\rho'}}$ be the set of all $(\mu_1, \ldots, \mu_n) \in (\kk \setminus \{0\})^n$ such that $\mu_1^m = \cdots = \mu_n^m = 1$, and in addition, $\mu_{i_j} = 1$, $j = 1, \ldots, \rho'$. Since $|\scrM_{i_1, \ldots, i_{\rho'}}| = m^{n-\rho'}$, \Cref{prop:formula:blow-up} implies that
\begin{align*}
\len(\local{C(\vec {ag})}{(z,0)}/t\local{C(\vec {ag})}{(z,0)})
    &=  \frac{1}{m^{n-\rho'}}
        \sum_{\vec \mu \in \scrM_{i_1, \ldots, i_{\rho'}}}
        \len(\local{C(\vec{\mu ag})}{(z,0)}/t\local{C(\vec{\mu ag})}{(z,0)})
\end{align*}
The right hand side of the identity above equals $\frac{1}{{m^{n-\rho'}}} \len(\local{C_{m, i_1, \ldots, i_{\rho'}}(\vec{ag})}{(z,0)}/t\local{C_{m, i_1, \ldots, i_{\rho'}}(\vec{ag})}{(z,0)})$ due to the additivity of the ``order'' function of a non zero-divisor over branches of a curve (see e.g.\ \cite[Theorem IV.24]{howmanyzeroes}).
\end{proof}

\begin{proprop} \label{prop:formula:constant}
Let $Z', i_1, \ldots, i_{\rho'}$ be such that $Z'$ is an irreducible component of $Z^{(\rho')}_{i_1, \ldots, i_{\rho'}}$, and let $r_1, \ldots, r_n$ be as in \Cref{prop:formula:blow-up}. Then there is a nonempty open subset $L^*$ of $\prod_i L_i$ such that the following sum is constant for each $g \in L^*$:
\begin{align*}
\sum_{\mathclap{\substack{
        z \in Z' \\
        \scrZ^{(\rho_z)}_{k_z} = Z' \\
        (i_1(z), \ldots, i_{\rho_z}(z)) = (i_1, \ldots, i_{\rho'})
    }}} 
    \len(\local{C(\vec {g})}{(z,0)}/t\local{C(\vec {g})}{(z,0)})
\end{align*}
\end{proprop}


\begin{proof}
Let $\nu, r_1, \ldots, r_n$ and $W := \pp^n(r_1, \ldots, r_n, 1)$ be as in \Cref{prop:formula:blow-up}, and let $Y$ be the closure in $\prod_i L_i \times \bar X \times \kk \times W$ of the graph of the map
\begin{align*}
\prod_i L_i \times \bar X \times \kk \ni (\vec g, x, t) &\mapsto [\frac{f_1}{g_1}(x): \cdots : \frac{f_n}{g_n}(x): t] \in W
\end{align*}
Fix $i_1, \ldots, i_{\rho'}$ such that $Z'$ is an irreducible component of $Z^{(\rho')}_{i_1, \ldots, i_{\rho'}}$, and consider
\begin{align*}
Y_0 &:= \{(\vec g, x, t) \in Y: x \in Z,\ g_{i'_1}(x) = \cdots = g_{i'_{n-\rho'}}(x) = t = 0\}, 
\end{align*}
where the $i'_j$ are the elements of $\{1, \ldots, n\}\setminus \{i_1, \ldots, i_{\rho'}\}$. Note that $Y_0$ is a (possibly reducible) closed subvariety of $Y$. Let $[w_1: \cdots :w_{n+1}]$ be (weighted) homogeneous coordinates on $W$, and $u_i := w_i/(w_{n+1})^{r_i}$, $i = 1, \ldots, n$, be the coordinates on $U_{n+1} := W \setminus V(w_{n+1})$. Let
\begin{align*}
\tilde L_0 &:= Y_0 \cap V(u_1 - 1, \ldots, u_n - 1) \cap \pi_W^{-1}(U_{n+1})
\end{align*}
where $\pi_W: Y \to W$ is the natural projection. Let $\pi_L: Y \to \prod_i L_i$ is the natural projection. If $\pi_L|_{\tilde L_0}$ is {\em not} dominant, then \Cref{prop:formula:constant} holds with the required sum equal to zero. So assume $\pi_L|_{\tilde L_0}$ is dominant. Then the arguments from the proof of \Cref{prop:formula:blow-up} imply that
\begin{prooflist}
\item there is a nonempty open subset $L^*_0$ of $\prod_i L_i$ over which each fiber of $\pi_L|_{\tilde L_0}$ is zero dimensional, and each fiber of $\pi_L|_{Y_0}$ has pure dimension $n$.
\end{prooflist}
In particular, $Y_0$ has codimension one in $Y$, and consequently, some irreducible components of $Y_0$ are also irreducible components of $V(t) \subseteq Y$. The observations from the proof of \Cref{prop:formula:blow-up} then imply that
\begin{prooflist}[resume]
\item there is a nonempty open subset $L^*_1$ of $L^*_0$ such that on $\tilde U^1_{n+1} := \pi_W^{-1}(U_{n+1}) \cap \pi_L^{-1}(L^*_1)$, we can consider $\tilde L_0$ as the (restriction of the) closed subscheme $V(t, u_1 - 1, \ldots, u_n - 1)$.
\item Denote the irreducible components of $\tilde L_0 \cap \tilde U^1_{n+1}$ by $\tilde L_{0,i}$. Then for each $\vec g$ in a nonempty open subset $L^*_2$ of $L^*_1$, the ``scheme-theoretic strict transform'' (see \Cref{footnote:scheme-transform} on \Cpageref{footnote:scheme-transform}) $\tilde C(\vec g)$ of $C(\vec{g})$ on $Y$ is the same as the subscheme $V(u_1 - 1, \ldots, u_n - 1)$ near $\{\vec g\} \times Z \times \{0\}$, and
\begin{align*}
\sum_{\substack{
        z \in Z' \\
        \scrZ^{(\rho_z)}_{k_z} = Z' \\
        (i_1(z), \ldots, i_{\rho_z}(z)) = (i_1, \ldots, i_{\rho'})
    }} \hspace{-3em}
    \len(\local{C(\vec {g})}{(z,0)}/t\local{C(\vec {g})}{(z,0)})
    &=
    \sum_i \len(\local{\tilde L_0}{\tilde L_{0,i}}) \deg(\pi_L|_{\tilde L_{0,i}})
\end{align*}
\end{prooflist}
Since the right hand side does {\em not} depend on $\vec{g} \in L^*_2$, this completes the proof of \Cref{prop:formula:constant} with $L^* = L^*_2$.
\end{proof}

The following theorem is the main result of this section. It confirms that the formula for the ordered intersection multiplicity guessed in \Cref{sec:comb:guess} is the correct one. 
For this result $X$ is an $n$-dimensional affine variety, $f_1, \ldots, f_n \in \kk[X]$, and $L_i$ are base-point free (finite dimensional) vector subspace of $\kk[X]$ containing $f_i$, $i = 1, \ldots, n$. Let $Z$ be a distinguished component for (the ordered intersection of the hypersurfaces defined by) $V(f_1), \ldots, V(f_n)$ (\Cref{defn:distinguished}), and $\check{Z}$ be the intersection of $Z$ with the union of all other distinguished components $Z'$ with $\codim(Z') \geq \codim(Z)$.

\begin{thm} \label{thm:length:confirmation}
There is $\nu \in \rr$ with the following property: for each $\vec{r} = (r_1, \ldots, r_n)$ satisfying \eqref{condition:>nu} and each nonempty open subset $Z^*$ of $Z$, there is a nonempty open subset $L^*$ of $\prod_i L_i$ such that for each $\vec g = (g_1, \ldots, g_n) \in L^*$,
\begin{align*}
\sum_{\mathclap{\substack{\text{branch}\ B\ \text{of}\ C(\vec{g}; \vec r)\\ \Center(B) \in (Z^* \setminus \check{Z}) \times \{0\}}}}
    \ord_B(t)
    &= \sum_{\mathclap{i_1, \ldots, i_\rho}}
        \len(\local{X}{Z}/I^{(\rho)}_{i_1, \ldots, i_\rho}\local{X}{Z})
        \deg_{i'_1, \ldots, i'_{n-\rho}}(Z)
\end{align*}
where the $i'_j$ on the right hand side are the elements of $\{1, \ldots, n\}\setminus \{i_1, \ldots, i_\rho\}$, and the sum is over all $i_1, \ldots, i_\rho$ such that $Z$ is an irreducible component of $Z^{(\rho)}_{i_1, \ldots, i_\rho}$.
\end{thm}

\begin{proof}
The proof of \Cref{thm:length:confirmation} straddles from \Crefrange{sec:length:confirmation:I'I''}{sec:length:confirmation:conclusion} below.

\subsubsection{} \label{sec:length:confirmation:I'I''}
Fix $i_1, \ldots, i_\rho$ such that $Z$ is an irreducible component of $Z^{(\rho)}_{i_1, \ldots, i_\rho}$. For each $j = 1, \ldots, \rho-1$,
\begin{itemize}
\item let $I'^{(j)}_{i_1, \ldots, i_\rho}$ be the intersection of all primary components $\qqq$ of $I^{(j)}_{i_1, \ldots, i_\rho}$ which are used in the construction of $I^{(j+1)}_{i_1, \ldots, i_\rho}$, i.e.\ $f_{i_{j+1}} \not\in \sqrt{\qqq}$, and $f_i \in \sqrt{\qqq}$ for each $i < i_j$;
\item let $I''^{(j)}_{i_1, \ldots, i_\rho}$ be the intersection of all other primary components $\qqq$ of $I^{(j)}_{i_1, \ldots, i_\rho}$.
\end{itemize}
It is then clear that for each $j = 1, \ldots, \rho -1$,
\begin{align*}
I^{(j)}_{i_1, \ldots, i_\rho} &= I'^{(j)}_{i_1, \ldots, i_\rho} \cap I''^{(j)}_{i_1, \ldots, i_\rho},\ \text{and} \\
I^{(j+1)}_{i_1, \ldots, i_\rho} &= I'^{(j)}_{i_1, \ldots, i_\rho} + f_{i_{j+1}}\kk[X]
\end{align*}
We will prove the following result in \Cref{proof:length:confirmation:I-g-subset} below.

\begin{proclaim} \label{length:confirmation:I-g-subset}
There is a nonempty open subset $Z^*_0$ of $Z$ and $\nu^*_0 \in \rr$ such that if $z \in Z^*_0$ and $g_1, \ldots, g_n, r_1, \ldots, r_n$ satisfy \eqref{condition:nonzero>nu} with $\nu \geq \nu^*_0$, then
\begin{enumerate}
\item for each $j$, $1 \leq j \leq \rho$,
\begin{align*}
I^{(j)}_{i_1, \ldots, i_\rho} \hatlocal{C}{(z,0)} \subseteq t^{r_{i_j}}\hatlocal{C}{(z,0)}
\end{align*}
In addition, $Z^*_0$ does {\em not} change (but $\nu^*_0$ possibly changes) if the $f_i$ are replaced by $(f_i)^{m_i}$ for $m_i \geq 1$.
\end{enumerate}
Pick $i' \in \{1, \ldots, n\}\setminus \{i_1, \ldots, i_\rho\}$. Recall that $i_1 = i^*$, where $i^*$ is the smallest index such that $F_{i^*}$ is not identically zero.
\begin{enumerate}[resume]
\item If $i' < i^*$, then $g_{i'}|_C \equiv 0 \in t\hatlocal{C}{(z,0)}$.
\item Assume $i' > i^*$. Pick the largest $j$ such that $i_j < i'$. Then $f_{i'} \in \sqrt{I'^{(j)}_{i_1, \ldots, i_\rho}}$. If $f_{i'} \in I'^{(j)}_{i_1, \ldots, i_\rho}$, then $g_{i'} \in t\hatlocal{C}{(z,0)}$.
\end{enumerate}
\end{proclaim}

\subsubsection{} \label{sec:length:confirmation:I}
We now apply the constructions of \Cref{sec:formula:set-up:0,sec:formula:set-up:1}. In particular, we assume $\bar X$ is a compactification of $X$, and for each $i = 1, \ldots, n$, $F_i$ is either identically zero (in the case that $f_i$ is identically zero), or a Cartier divisor on $\bar X$ extending $V(f_i) \subseteq X$. Let $\scrZ^{(\rho', *)}_k$ be defined as in \Cref{sec:formula:set-up:1} with the following additional constraint:
\begin{prooflist}
\item $\scrZ^{(\rho^*, *)}_{k^*} \subseteq Z^*$, where $\rho^*, k^*$ are the unique indices such that $\scrZ^{(\rho^*)}_{k^*}$ is the closure of $Z$ in $\bar X$ (recall that $Z^*$ is the nonempty open subset of $Z$ from the statement of \Cref{thm:length:confirmation}).
\end{prooflist}
For each $\rho', k$, recall that $F_i$ is locally represented near $\scrZ^{(\rho', *)}_k$ by $f_i/h_{\rho',k,i}$ (Property \eqref{length:confirmation:assumption:F_i} of \Cref{sec:formula:set-up:1}). Consider the ideals $I^{(j)}_{\rho', k; i_1, \ldots, i_{\rho'}}$ obtained by applying the construction of $I^{(j)}_{i_1, \ldots, i_{\rho'}}$ from \Cref{thm:length} with $f_i/h_{\rho',k,i}$ in place of $f_i$, $i = 1, \ldots, n$. Replacing the $\scrZ^{(\rho', *)}_k$ by appropriate open subsets if necessary, we may ensure that for each $\rho', k$,
\begin{prooflist}[resume]
\item \label{length:confirmation:assumption:sep} $\scrZ^{(\rho', *)}_{k}$ does not intersect any irreducible component of $\bigcup_{i_1, \ldots, i_{\rho'}} V(I^{(\rho')}_{i_1, \ldots, i_{\rho'}})$ other than $\scrZ^{(\rho')}_{k}$ itself.
\end{prooflist}

\subsubsection{} \label{sec:length:confirmation:M}
Analogously, define $I'^{(j)}_{\rho', k; i_1, \ldots, i_{\rho'}}$ as in \Cref{sec:length:confirmation:I'I''}, and pick $M \geq 1$ such that
$(f_{i'}/h_{\rho', k, i'})^M \in I'^{(j)}_{\rho', k; i_1, \ldots, i_{\rho'}}$ for each $k, \rho', i_1, \ldots, i_{\rho'}$, and each $j, l$ such that
    \begin{itemize}
    \item either $i_j < i' < i_{j+1}$, or
    \item $j = \rho$ and $i' > i_{\rho'}$.
    \end{itemize}
For each $k, \rho', i_1, \ldots, i_{\rho'}$, and $m \geq M$, define $C_{m, i_1, \ldots, i_{\rho'}}(\vec{g}; \vec r)$ as in \Cref{prop:length:confirmation:constant-m}. \Cref{length:confirmation:I-g-subset} implies that there are nonempty open subset $\scrZ^{(\rho', *)}_{k, 0} \subseteq \scrZ^{(\rho', *)}_k$ and $\nu'^*_m \in \rr$ such that for each $k, \rho', i_1, \ldots, i_{\rho'}$, if \eqref{condition:>nu} holds with $\nu \geq \nu'^*_m$, and $z \in \scrZ^{(\rho', *)}_{k, 0}$ is such that $(z,0) \in C_{m, i_1, \ldots, i_{\rho'}}(\vec{g}; \vec r)$ and $(g_{i_j}/h_{\rho', k, i_j})(z) \neq 0$, $1 \leq j \leq \rho'$, then
\begin{prooflist}
\item \label{length:confirmation:M:I} $I^{(\rho')}_{\rho', k; i_1, \ldots, i_{\rho'}} \hatlocal{C_{m, i_1, \ldots, i_{\rho'}}(\vec{g}; \vec r)}{(z,0)} \subseteq t\hatlocal{C_{m, i_1, \ldots, i_{\rho'}}(\vec{g}; \vec r)}{(z,0)}$,
\item \label{length:confirmation:M:g} $(g_{i'_j}/h_{\rho', k, i'_j})^m \in t\hatlocal{C_{m, i_1, \ldots, i_{\rho'}}(\vec{g}; \vec r)}{(z,0)}$ for each $j = 1, \ldots, n - \rho'$, where $i'_1 < \cdots < i'_{n-\rho'}$ are the elements of $\{1, \ldots, n\}\setminus \{i_1, \ldots, i_{\rho'}\}$.
\end{prooflist}

\subsubsection{} \label{sec:length:confirmation:intprod}
Now run the constructions of \Cref{sec:formula:set-up:1} with all $\scrZ^{(\rho', *)}_k$ replaced by $\scrZ^{(\rho', *)}_{k,0}$ from \Cref{sec:length:confirmation:M}. Consider resulting $\nu^*_1 \in \rr$ and $L^*_1 \subseteq \prod_i L_i$ from \Cref{sec:formula:set-up:1}. Pick $\nu \geq \nu^*_1$ and $\vec r = (r_1, \ldots, r_n)$ satisfying \eqref{condition:>nu}. Given $\vec g = (g_1, \ldots, g_n) \in L^*_1$ and $m \geq M$, consider the nonempty open subset $\scrA^*_{m, \vec r}(\vec g)$ of $(\kk \setminus \{0\})^n$ constructed in \Cref{sec:length:confirmation:constant-m} (here we write $\scrA^*_{m, \vec r}(\vec g)$ instead of $\scrA^*_m$ in order to explicitly denote the dependence on $\vec r, \vec g$). For each $\vec a = (a_1, \ldots, a_n) \in \scrA^*_{m, \vec r}(\vec g)$, consider $C(\vec {ag}; \vec r)$ where $\vec{ag} = (a_1g_1, \ldots, a_ng_n) \in L^*_1$. Let $\scrS$ be the collection of all the $\scrZ^{(\rho')}_k$ over all $\rho', k$. Then
\begin{align*}
\mult{\scrL_1}{\scrL_n}
    &= \sum_{z \in \bar X}
    \len(\local{C(\vec {ag}; \vec r)}{(z,0)}/t\local{C(\vec {ag}; \vec r)}{(z,0)}) \\
    &= \sum_{Z' \in \scrS}
    \sum_{\substack{
        i_1, \ldots, i_{\rho'} \\
        \rho' = \codim(Z') \\
        Z' \subseteq Z^{(\rho')}_{i_1, \ldots, i_{\rho'}}
    }}
    \sum_{\substack{
        z \in Z' \\
        \scrZ^{(\rho_z)}_{k_z} = Z' \\
        (i_1(z), \ldots, i_{\rho_z}(z)) = (i_1, \ldots, i_{\rho'})
    }}
    \len(\local{C(\vec {ag}; \vec r)}{(z,0)}/t\local{C(\vec {ag}; \vec r)}{(z,0)})
\end{align*}

\subsubsection{} \label{sec:length:confirmation:constant-m:M}
For each summand of the (innermost) sum in the above identity, \Cref{prop:length:confirmation:constant-m} implies that
\begin{align*}
\len(\local{C(\vec {ag}; \vec r)}{(z,0)}/t\local{C(\vec {ag}; \vec r)}{(z,0)})
    &=  \frac{1}{m^{n-\rho'}}
        \len(\local{C_{m, i_1, \ldots, i_{\rho'}}(\vec{ag}; \vec r)}{(z,0)}
            / t\local{C_{m, i_1, \ldots, i_{\rho'}}(\vec{ag}; \vec r)}{(z,0)})
\end{align*}
Now assume $\nu$ (from \Cref{sec:length:confirmation:intprod}) is greater than or equal to $\nu'^*_m$ from \Cref{sec:length:confirmation:M}. Then writing $C_* := C_{m, i_1, \ldots, i_{\rho'}}(\vec{ag}; \vec r)$ as a shorthand, we have
\begin{align*}
\len(\local{C_*}{(z,0)}/t\local{C_*}{(z,0)} )
     &=\len(\hatlocal{C_*}{(z,0)}/t\hatlocal{C_*}{(z,0)}) \\
     &= \len(
        \hatlocal{\bar X \times \kk}{(z,0)}/
            \langle
                I(C_*), t
            \rangle \hatlocal{\bar X \times \kk}{(z,0)}
    ) \\
     &= \len(
        \hatlocal{\bar X \times \kk}{(z,0)}/
            \langle
                I(C_*), t,
                I^{(\rho')}_{\rho',k; i_1, \ldots, i_{\rho'}},
                (\frac{g_{i'_1}}{h_{\rho', k, i'_1}})^m , \ldots, (\frac{g_{i'_{n-\rho'}}}{h_{\rho', k, i'_{n-\rho'}}})^m
            \rangle \hatlocal{\bar X \times \kk}{(z,0)}
    ) \\
    & \qquad \text{(due to properties \ref{length:confirmation:M:I} and \ref{length:confirmation:M:g} from \Cref{sec:length:confirmation:M})} \\
     &\leq \len(
        \hatlocal{\bar X \times \kk}{(z,0)}/
            \langle
                t,
                I^{(\rho')}_{\rho',k; i_1, \ldots, i_{\rho'}},
                (\frac{g_{i'_1}}{h_{\rho', k, i'_1}})^m, \ldots, (\frac{g_{i'_{n-\rho'}}}{h_{\rho', k, i'_{n-\rho'}}})^m
            \rangle \hatlocal{\bar X \times \kk}{(z,0)}
    ) \\
    &= \len(
        \hatlocal{\bar X}{z}/
            \langle
                I^{(\rho')}_{\rho',k; i_1, \ldots, i_{\rho'}},
                (\frac{g_{i'_1}}{h_{\rho', k, i'_1}})^m , \ldots, (\frac{g_{i'_{n-\rho'}}}{h_{\rho', k, i'_{n-\rho'}}})^m
            \rangle \hatlocal{\bar X}{z}
    ) 
\end{align*}

\begin{proclaim} \label{claim:length:confirmation:limit:0}
For each $m \geq M$, fix $\nu_m \geq \max\{\nu'^*_m, \nu^*_1\}$ (where $\nu^*_1$ is from \Cref{sec:length:confirmation:intprod}) and $\vec r_m$ such that \eqref{condition:>nu} holds with $\vec r = \vec r_m$ and $\nu = \nu_m$. If $\scrA^*_{\vec g} := \bigcap_{m \geq M}\scrA^*_{m, \vec r_m}(\vec g)$ is nonempty, then for all sufficiently large $m$,
\begin{align*}
\len(\local{C(\vec{ag}; \vec r_m)}{(z,0))}/t\local{C(\vec{ag}; \vec r_m)}{(z,0)})
    &\leq
    \len(\local{\bar X}{Z'}/I^{(\rho')}_{\rho',k; i_1, \ldots, i_{\rho'}}\local{\bar X}{Z'})
        \multsub{\frac{g_{i'_1}}{h_{\rho', k, i'_1}}|_{Z'}}{\frac{g_{i'_{n-\rho'}}}{h_{\rho', k, i'_{n-\rho'}}}|_{Z'}}{z}
\end{align*}
for each $\vec a \in \scrA^*_{\vec g}$ and $z \in Z' \cap V(\frac{g_{i'_1}}{h_{\rho', k, i'_1}}, \ldots, \frac{g_{i'_{n-\rho'}}}{h_{\rho', k, i'_{n-\rho'}}})$, where $\multonlysub{\cdot,\cdots,\cdot}{z}$ denotes the intersection multiplicity at $z$ (of Cartier divisors on $Z'$).
\end{proclaim}

\begin{proof}
Indeed, it follows from \Cref{sec:length:confirmation:constant-m:M} that
\begin{align*}
\len(\local{C(\vec{ag}; \vec r_m)}{(z,0))}/t\local{C(\vec{ag}; \vec r_m)}{(z,0)})
    & \leq \frac{1}{m^{n-\rho}}\len(
        \local{\bar X}{z}/
            \langle
                I^{(\rho')}_{\rho',k; i_1, \ldots, i_{\rho'}},
                (\frac{g_{i'_1}}{h_{\rho', k, i'_1}})^m, \ldots, (\frac{g_{i'_{n-\rho'}}}{h_{\rho', k, i'_{n-\rho'}}})^m
            \rangle \local{\bar X}{z}
        )
\end{align*}
The {\em limit formula of Lech} \cite[Theorem 7.5.10]{northcott-lessons} then implies that the limit of the right hand side as $m \to \infty$ is $e_{R_z}(\frac{g_{i'_1}}{h_{\rho', k, i'_1}}, \ldots, \frac{g_{i'_{n-\rho'}}}{h_{\rho', k, i'_{n-\rho'}}} \mid R_z)$, where $R_z := \local{\bar X}{z}/ I^{(\rho')}_{\rho',k; i_1, \ldots, i_{\rho'}} \local{\bar X}{z}$, and $e_\cdot(\cdot \mid \cdot)$ is the ``multiplicity'' symbol. Since the left hand side is always an integer, it follows that for all $m \gg 1$,
\begin{align*}
\len(\local{C(\vec{ag}; \vec r_m)}{(z,0))}/t\local{C(\vec{ag}; \vec r_m)}{(z,0)})
    & \leq e_{R_z}(\frac{g_{i'_1}}{h_{\rho', k, i'_1}}, \ldots, \frac{g_{i'_{n-\rho'}}}{h_{\rho', k, i'_{n-\rho'}}} \mid R_z)
\end{align*}
for each $z \in Z' \cap V(\frac{g_{i'_1}}{h_{\rho', k, i'_1}}, \ldots, \frac{g_{i'_{n-\rho'}}}{h_{\rho', k, i'_{n-\rho'}}})$ (recall that the latter set is finite due to condition \eqref{length:confirmation:assumption:finite} of \Cref{sec:formula:set-up:1}). It follows from the definitions that in this case the multiplicity symbol is same as the intersection multiplicity at $z$ of the Cartier divisor defined by $g_{i'_j}/h_{\rho', k, i'_j}$, $j = 1, \ldots, n - \rho'$, on the subscheme $\tilde Z'$ of $\bar X$ defined by $I^{(\rho')}_{\rho',k; i_1, \ldots, i_{\rho'}}$. Since near each $z \in Z' \cap V(\frac{g_{i'_1}}{h_{\rho', k, i'_1}}, \ldots, \frac{g_{i'_{n-\rho'}}}{h_{\rho', k, i'_{n-\rho'}}})$, the only irreducible component of $\tilde Z'$ is $Z'$ itself (property \ref{length:confirmation:assumption:sep} of \Cref{sec:length:confirmation:I}), it follows that
\begin{align*}
e_{R_z}(\frac{g_{i'_1}}{h_{\rho', k, i'_1}}, \ldots, \frac{g_{i'_{n-\rho'}}}{h_{\rho', k, i'_{n-\rho'}}} \mid R_z)
    &= \multsub{\frac{g_{i'_1}}{h_{\rho', k, i'_1}}|_{\tilde Z'}}{\frac{g_{i'_{n-\rho'}}}{h_{\rho', k, i'_{n-\rho'}}}|_{\tilde Z'}}{z} \\
    &= \len(\local{\bar X}{Z'}/I^{(\rho')}_{\rho',k; i_1, \ldots, i_{\rho'}}\local{\bar X}{Z'})
        \multsub{\frac{g_{i'_1}}{h_{\rho', k, i'_1}}|_{Z'}}{\frac{g_{i'_{n-\rho'}}}{h_{\rho', k, i'_{n-\rho'}}}|_{Z'}}{z}
\end{align*}
This completes the proof of the claim.
\end{proof}

\begin{proprop} \label{claim:length:confirmation:limit:1}
Fix $i_1, \ldots, i_{\rho'}$ such that $Z' = \scrZ^{(\rho')}_k$ is an irreducible component of $Z^{(\rho')}_{i_1, \ldots, i_{\rho'}}$. Then there is $\nu'^*$ with the following property: for all $\nu \geq \nu'^*$ and $\vec r$ satisfying \eqref{condition:>nu}, there is a nonempty open subset $L'^*$ of $\prod_i L_i$ such that for each $\vec g = (g_1, \ldots, g_n) \in L'^*$,
\begin{align*}
    \sum_{\mathclap{\substack{
            z \in Z' \\
            \scrZ^{(\rho_z)}_{k_z} = Z' \\
            (i_1(z), \ldots, i_{\rho_z}(z)) = (i_1, \ldots, i_{\rho'})
    }}}
    \len(\local{C(\vec {g}; \vec r)}{(z,0)}/t\local{C(\vec {g}; \vec r)}{(z,0)})
    &
    \leq 
        \len(\local{\bar X}{Z'}/I^{(\rho')}_{\rho',k; i_1, \ldots, i_{\rho'}}\local{\bar X}{Z'})
        \deg_{i'_1, \ldots, i'_{n-\rho'}}(Z)
\end{align*}
\end{proprop}

\begin{proof}
We may assume \woutlog\ that $\kk$ is uncountable\footnote{Since one can proceed after embedding $\kk$ into the algebraic closure of power series in a new variable.}. \Cref{prop:formula:constant} implies that for each choice of sufficiently large $\nu$ and $r_1, \ldots, r_n$ satisfying \eqref{condition:>nu}, the sum on the left hand side of the above relation is a constant $e_{\nu, \vec r}$ for all $\vec g$ in a nonempty open subset $L^*_{\nu, \vec r}$ of $\prod_i L_i$. It suffices to show that $e_{\nu, \vec r} \leq \lambda := \len(\local{\bar X}{Z'}/I^{(\rho')}_{\rho',k; i_1, \ldots, i_{\rho'}}\local{\bar X}{Z'}) \deg_{i'_1, \ldots, i'_{n-\rho'}}(Z)$ if $\nu$ is sufficiently large. Indeed, otherwise for each $m$, there are $\nu_m, \vec r_m$ as in \Cref{claim:length:confirmation:limit:0} such that $e_{\nu_m, \vec r_m} > \lambda$. Since $\kk$ is uncountable, there is $\vec g \in L^*_1 \cap \bigcap_m L^*_{\nu_m, \vec r_m}$ (where $L^*_1$ is from \Cref{sec:length:confirmation:intprod}). Now for each $m$, let $\scrA'^*_{m, \vec r_m}(\vec g)$ be the nonempty open subset of $(\kk \setminus \{0\})^n$ consisting of all $\vec a \in \scrA^*_{m, \vec r_m}(\vec g)$ such that $\vec{ag} \in L^*_{\nu_m, \vec r_m}$. The uncountability of $\kk$ again implies that there is $\vec a \in \bigcap_m \scrA'^*_{m, \vec r_m}(\vec g)$, which yields a contradiction via \Cref{claim:length:confirmation:limit:0}.
\end{proof}

\subsubsection{}\label{sec:length:confirmation:conclusion}
Combining the observations from \Cref{sec:length:confirmation:intprod}, \Cref{claim:length:confirmation:limit:1} and \Cref{thm:comb:intersection} we obtain that for each $\vec r$ satisfying \eqref{condition:>nu} with sufficiently large $\nu$, there is a nonempty open subset $L^*$ of $\prod_i L_i$ such that for all $\vec g \in L^*$,
\begin{align*}
\mult{\scrL_1}{\scrL_n}
    &=
    \sum_{Z' \in \scrS}
    \sum_{\substack{
        i_1, \ldots, i_{\rho'} \\
        \rho' = \codim(Z') \\
        Z' \subseteq Z^{(\rho')}_{i_1, \ldots, i_{\rho'}}
    }}
    \sum_{\substack{
        z \in Z' \\
        \scrZ^{(\rho_z)}_{k_z} = Z' \\
        (i_1(z), \ldots, i_{\rho_z}(z)) = (i_1, \ldots, i_{\rho'})
    }} \hspace{-3em}
    \len(\local{C(\vec{g}; \vec r)}{(z,0))}/t\local{C(\vec{g}; \vec r)}{(z,0)}) \\
    &\leq
    \sum_{Z' \in \scrS}
    \sum_{\substack{
        i_1, \ldots, i_{\rho'} \\
        \rho' = \codim(Z') \\
        Z' \subseteq Z^{(\rho')}_{i_1, \ldots, i_{\rho'}}
    }}
    \len(\local{\bar X}{Z'}/I^{(\rho')}_{\rho',k;i_1, \ldots, i_{\rho'}}\local{\bar X}{Z'})
        \deg_{i'_1, \ldots, i'_{n-\rho}}(Z') \\
    &= \mult{\scrL_1}{\scrL_n}
\end{align*}
Then the inequality in the middle must be an equality, and for each $Z', i_1, \ldots, i_{\rho'}$ and each $\vec g \in L^*$ we must have
\begin{align*}
\sum_{\mathclap{\substack{
        z \in Z' \\
        \scrZ^{(\rho_z)}_{k_z} = Z' \\
        (i_1(z), \ldots, i_{\rho_z}(z)) = (i_1, \ldots, i_{\rho'})
    }}}
    \len(\local{C(\vec{g}; \vec r)}{(z,0)}/t\local{C(\vec{g}; \vec r)}{(z,0)})
    &=
    \len(\local{\bar X}{Z'}/I^{(\rho')}_{\rho',k;i_1, \ldots, i_\rho}\local{\bar X}{Z'})
        \deg_{i'_1, \ldots, i'_{n-\rho}}(Z')
\end{align*}
\Cref{thm:length:confirmation} now follows from taking $Z'$ equal to the closure $\bar{Z}$ of $Z$ in $\bar X$.
\end{proof}

\part{Computation: estimates and examples} \label{part:bezout}

\section{Local factors of ordered intersection as (usual) intersection multiplicities} \label{sec:local}
\subsection{} \label{sec:local-set-up}
Recall the basic set up: we are given base-point free linear systems $\scrL_1, \ldots,\scrL_n$ on a purely $n$-dimensional variety $X$, and for each $i$, we have a fixed element $F_i \in \scrL_i$, which is either an effective (Cartier) divisor, or the zero element in $\scrL_i$. \Cref{thm:comb:intersection} shows that\footnote{In \Cref{thm:comb:intersection} no $F_i$ is allowed to be identically zero. If say $k$ of the $F_i$ are identically zero, then one obtains identities \eqref{eq:Fintersection-len} and \eqref{eq:multord} as follows: first reorder the $F_i$ such that all the identically zero elements appear first, and then appear the remaining ones, but with the {\em ordering among the non identically-zero elements unchanged}. Then replace the identically zero elements by generic elements of the corresponding $\scrL_i$, and apply \Cref{thm:comb:intersection} with $i^* = k+1$.}
\begin{align}
\mult{\scrL_1}{\scrL_n}
    &= \sum_Z \multord{F_1}{F_n}{Z}
  \label{eq:Fintersection-len}
\end{align}
where $\multord{F_1}{F_n}{Z}$ is the ``ordered intersection multiplicity'' of $F_1, \ldots, F_n$ along $Z$, which is zero unless $Z$ is a ``distinguished component'' of the ordered intersection of $F_1, \ldots, F_n$, in which case
\begin{align}
\multord{F_1}{F_n}{Z}
 &=
\sum_{i_1, \ldots, i_\rho}
        \len(\local{X}{Z}/I^{(\rho)}_{i_1, \ldots, i_\rho}\local{X}{Z})
        \deg_{i'_1, \ldots, i'_{n-\rho}}(Z)
        \label{eq:multord}
\end{align}
where $\rho := \codim(Z)$, and the sum is over all $i_1, \ldots, i_\rho$ are such that $Z$ is an irreducible component of $Z^{(\rho)}_{i_1, \ldots, i_\rho}$ from \Cref{thm:Z}, and $I^{(\rho)}_{i_1, \ldots, i_\rho}$ are as in \Cref{thm:length}. In this section we describe a few cases where the ``local factors'' $\len(\local{X}{Z}/I^{(\rho)}_{i_1, \ldots, i_\rho} \local{X}{Z})$ can be interpreted as ``intersection multiplicities of regular sequences'' at nonsingular points, and work out basic examples in dimension $\leq 3$.

\subsection{} First we recall the relevant definitions:
\begin{itemize}
\item Let $h_1, \ldots, h_k$ be regular functions on an open subset $U$ of $X$, and $W$ be an irreducible component of $V(h_1, \ldots, h_k)$. If $\codim(W) = k$, then we say that $h_1, \ldots, h_k$ form a {\em regular sequence near} $W$; if in addition $W$ is not contained in the set $\sing(X)$ of singular points of $X$, then the {\em intersection multiplicity of $h_1, \ldots, h_k$ along $W$} is defined as:
\begin{align*}
\multsub{h_1}{h_k}{W}
    &:= \len(\local{X}{W}/\langle h_1, \ldots, h_k \rangle \local{X}{W})
\end{align*}

\item If $D_1, \ldots, D_k$ are effective Cartier divisors on $X$ and $W$ is an irreducible component of $\bigcap_{i=1}^k \supp(D_i)$ such that $\codim(W) = k$ and $W \not\subseteq \sing(X)$, then the intersection multiplicity of $D_1, \ldots, D_k$ along $W$ is:
\begin{align*}
\multsub{D_1}{D_k}{W}
    &:= \multsub{h_1}{h_k}{W \cap U}
\end{align*}
where $h_i$ are local representatives of $D_i$ on an open subset $U$ of $X$ such that $U \cap W \neq \emptyset$.
\end{itemize}


\subsection{}
We continue to work in the set up of \Cref{sec:local-set-up}. By a {\em distinguished component} of (the ordered intersection of) $F_1, \ldots, F_n$ we mean an irreducible component of some  $Z^{(\rho)}_{i_1, \ldots, i_\rho}$ from \Cref{thm:Z}. As in \Cref{sec:Z} we write $\scrZ^{(\rho)}$ for the union of all distinguished components of codimension $\rho$. 

\begin{prop} \label{prop:local=mult:1}
Let $Z$ be a distinguished component of $F_1, \ldots, F_n$ such that $Z \not\subseteq \sing(X)$. Let $i^*$ be the smallest index such that $F_{i^*}$ is not identically zero, and $\rho := \codim(Z)$.
\begin{enumerate}
\item \label{local=mult:1:0} If $Z$ is an irreducible component of $\bigcap_{j=1}^\rho \supp(F_{i_j})$, then
\begin{align*}
\len(\local{X}{Z}/I^{(\rho)}_{i_1, \ldots, i_\rho} \local{X}{Z})
    &= \multsub{F_{i_1}}{F_{i_\rho}}{Z}
\end{align*}
\item In particular,
\begin{enumerate}
\item \label{local=mult:1:1} If $\rho = 1$, then
\begin{align*}
\len(\local{X}{Z}/I^{(1)}_{i^*} \local{X}{Z})
    &= \multonlysub{F_{i^*}}{Z} \\
\multord{F_1}{F_n}{Z}
    &= \multonlysub{F_{i^*}}{Z}
        \deg_{1, \ldots, \hat{i^*}, \ldots, n}(Z)
\end{align*}
\item If $\rho = n$ and $Z$ is an isolated point in $\bigcap_{i=1}^n \supp(F_i)$, then
\begin{align*}
\multord{F_1}{F_n}{Z}
    &= \len(\local{X}{Z}/I^{(n)}_{1, \ldots, n} \local{X}{Z})
    = \multsub{F_1}{F_n}{Z}
\end{align*}
\end{enumerate}
\item \label{local=mult:1:2} If $\rho = 2$, then pick an open affine subset $U$ of $X$ such that $U \cap Z \neq \emptyset$, and each $F_i$ is defined on $U$ by some $f_i \in \kk[U]$. Since $\local{X}{Z}$ is a UFD\footnote{Due to Auslander–Buchsbaum theorem.}, it contains (up to multiplication by units) a unique element $f_{i^*, i_2}$ which is the least common multiple of all $f \in \local{X}{Z}$ such that
    \begin{enumerate}
    \item $f$ divides $f_{i^*}$,
    \item each irreducible factor of $f$ divides each $f_i$, $i < i_2$, and
    \item no irreducible factor of $f$ divides $f_{i_2}$.
    \end{enumerate}
    Then
    \begin{align*}
    \len(\local{X}{Z}/I^{(2)}_{i^*, i_2} \local{X}{Z})
        &= \multonlysub{f_{i^*, i_2},f_{i_2}}{Z} \\
    \multord{F_1}{F_n}{Z}
        &= \sum_{i_2}
            \multonlysub{f_{i^*, i_2},f_{i_2}}{Z}
            \deg_{i'_1, \ldots, i'_{n-2}}(Z)
    \end{align*}
    where the sum is over all $i_2$ such that $Z$ is an irreducible component of $Z^{(\rho)}_{i^*, i_2}$, and $i'_1, \ldots, i'_{n-2}$ are the elements of $\{1, \ldots, n\} \setminus \{i^*, i_2\}$.
\end{enumerate}
\end{prop}

\begin{proof}
It follows in a straightforward way from the definitions.
\end{proof}

\subsection{} \label{sec:local-curves}
For the next result, we pick $i_1, \ldots, i_\rho$ such that $Z$ is an irreducible component of $Z^{(\rho)}_{i_1, \ldots, i_\rho}$. In \Cref{thm:length} we constructed ideals $I^{(j)}_{i_1, \ldots, i_\rho}$ on an affine neighborhood $U$ in $X$ of some point in $Z$. Recall that $V(I^{(j}_{i_1, \ldots, i_\rho})$ has pure codimension $j$. In particular, if $i'_1, \ldots, i'_{n-\rho}$ are the elements of $\{1, \ldots, n\}\setminus\{i_1, \ldots, i_\rho\}$, and $g_{i'_k}$ are local representatives on $U$ of generic elements in $\scrL_{i'_k}$, then $V(I^{(\rho-1)}_{i_1, \ldots, i_\rho}, g_{i'_1}, \ldots, g_{i'_{n-\rho}})$ has pure dimension one (or is empty); to simplify the notation we write $\vec i := (i_1, \ldots, i_\rho)$, $\vec g := (g_{i'_1}, \ldots, g_{i'_{n-\rho}})$ and $C_{\vec i, \vec g} := V(I^{(\rho-1)}_{i_1, \ldots, i_\rho}, g_{i'_1}, \ldots, g_{i'_{n-\rho}})$.

\begin{prop} \label{prop:local-curves}
With the above notations,
\begin{align*}
\len(\local{X}{Z}/I^{(\rho)}_{i_1, \ldots, i_\rho} \local{X}{Z})\deg_{i'_1, \ldots,i'_{n-\rho}}(Z)
    &= \sum_{P \in Z \cap C_{\vec i, \vec g}} \ord_P(f_{i_\rho}|_{C_{\vec i, \vec g}})
\end{align*}
where $\ord_P(\cdot)$ is either zero (if $C_{\vec i, \vec g} = \emptyset$), or the {\em order of vanishing}\footnote{If $w$ is a (closed) point of a purely one dimensional scheme $W$ and $f \in \local{W}{w}$, then $\ord_w(f)$ is the length of $\local{W}{w}/f\local{W}{w}$ as a module over $\local{W}{w}$ \cite[Section 1.2]{fultersection}.} of $f_{i_\rho}|_{C_{\vec i, \vec g}}$ at $P$. Consequently,
\begin{align*}
\multord{F_1}{F_n}{Z}
    &= \sum_{i_1, \ldots, i_\rho} \sum_{P \in Z \cap C_{\vec i, \vec g}} \ord_P(f_{i_\rho}|_{C_{\vec i, \vec g}})
\end{align*}
where the sum is over all $i_1 = i^* < i_2 < \cdots < i_\rho$ such that $Z$ is an irreducible component of $Z^{(\rho)}_{i_1, \ldots, i_\rho}$.
\end{prop}

\begin{proof}
This follows from identity \eqref{eq:comb:m} on page \pageref{eq:comb:m} and \Cref{thm:length}.
\end{proof}

\begin{example} \label{ex:2-1}
Let $F_i = \sum_j m_{i,j}C_j + \sum_k m'_{i,k} C'_{i,k}$, $i = 1, 2$, be curves (i.e.\ purely one dimensional schemes) on a surface $X$, where $m_{i,j}, m'_{i,k}$ are positive integers, and $C_j, C'_{i,k}$ are pairwise distinct irreducible curves; in particular, $C_j$ are precisely the irreducible components common to both $F_1$ and $F_2$. Then
\begin{align}
\begin{split} \label{eq:2-1:supp}
\bigcap_i \supp(F_i)
    &= \bigcup_j C_j \cup \bigcup_{k_1,k_2} (C'_{1,k_1} \cap C'_{2,k_2})\\
\scrZ^{(1)}
    &= \bigcup_j C_j \\
\scrZ^{(2)}
    &=  \bigcup_k (C'_{1,k} \cap \supp(F_2))
\end{split}
\end{align}
For each $j$,
\begin{align*}
\multordonly{F_1,F_2}{C_j}
    &= m_{1,j}\multonly{C_j, F_2}
    = m_{1,j} \deg_2(C_j)
\end{align*}
If $C_j \not\subseteq \sing(X)$, then $F_1$ can be represented by some $f_1$ near a generic (nonsingular) point of $C_j$, and $m_{1,j} = \multonlysub{f_1}{C_j}$ is the order of vanishing of $f_1$ along $C_j$, so that
\begin{align}
\multordonly{F_1,F_2}{C_j}
    &= \multonlysub{f_1}{C_j} \deg_2(C_j)
    \label{eq:2-1:1:m}
\end{align}
Similarly, if $P \in \scrZ^{(2)}$ is a nonsingular point of $X$, then
\begin{align}
\multordonly{F_1,F_2}{P}
    &= \multonlysub{f'_1, f_2}{P}
    \label{eq:2-1:2:m}
\end{align}
where $f'_1, f_2$ are local equations of respectively $\sum_k m'_{1,k} C'_{1,k}$ and $F_2$ near $P$.
\end{example}

\begin{example} \label{ex:n-1}
More generally, assume $F_1, \ldots, F_{n-1}$ intersect properly, and the scheme theoretic intersection of $F_1, \ldots, F_{n-1}$ is $\sum_j m_j C_j + \sum_k m'_k C'_k$, where the $C_j, C'_k$ are irreducible (reduced) curves such that $\supp(F_n)$ contains each $C_j$, and does {\em not} contain any $C'_k$. Then
\begin{align*}
\scrZ^{(\rho)}
    &=
    \begin{cases}
    \emptyset
        &\text{if}\ \rho < n-1, \\
    \bigcup_j C_j
        &\text{if}\ \rho = n-1, \\
    \bigcup_k (C'_k \cap \supp(F_n))
        &\text{if}\ \rho = n
\end{cases} \\
\multord{F_1}{F_n}{Z}
    &=
    \begin{cases}
    m_j \deg_n(C_j)
        &\text{if}\ Z = C_j, \\
    m'_k \ord_P(F_n|_{C'_k})
        &\text{if}\ Z = P \in C'_k, \\
    0
        &\text{otherwise} \\
\end{cases}
\end{align*}
\end{example}

\begin{example} \label{ex:3-2}
Let $X$ be a variety of pure dimension $3$, and $F_i = \sum_j m_{i,j}H_j + F'_i$, where $m_{i,j}$ are positive integers, $H_j$ are irreducible hypersurfaces which are not contained in $\bigcup_i \supp(F'_i)$, and the $F'_i$ are hypersurfaces with no irreducible components common to all $F'_i$. We will describe the distinguished components of the ordered intersection of the $F_i$ and the corresponding ordered intersection multiplicity.

\subsubsection{}
Since $H_j$ are codimension one irreducible components of each $F_i$, it follows that
\begin{align}
\scrZ^{(1)} &= \bigcup_j \supp(H_j)
    \label{eq:3-2:1:supp}
\end{align}
Since
\begin{align}
\multonly{F_1, F_2, F_3}
    &= \multonly{\sum_j m_{1,j}H_j + F'_1, F_2, F_3}
    = \sum_j m_{1,j}\multonly{H_j, F_2, F_3} + \multonly{F'_1, F_2, F_3}
    \label{eq:3-2:1}
\end{align}
it follows that
\begin{align*}
\multordonly{F_1, F_2, F_3}{H_j}
    &= m_{1,j}\multonly{H_j, F_2, F_3}
    = m_{1,j} \deg_{2,3}(H_j)
\end{align*}
If $H_j \not\subseteq \sing(X)$, then
\begin{align}
\multordonly{F_1, F_2, F_3}{H_j}
    &= \multonlysub{f_1}{H_j} \deg_{2,3}(H_j)
    \label{eq:3-2:1:m}
\end{align}
where $f_1$ is a local equation of $F_1$ near a generic (nonsingular) point of $H_j$.

\subsubsection{}
Now we consider $\multonly{F'_1, F_2, F_3}$, the second summand of the right hand side of \eqref{eq:3-2:1}. Write $F'_1 = F'_{1,2} + F''_1$, where $F'_{1,2}$ is the sum of the irreducible components of $F'_1$ (with appropriate multiplicity) which are common to $F_2$, and $F''_1$ does not have any irreducible component in common with $F'_2$. It follows that
\begin{align*}
\multonly{F'_1, F_2, F_3}
    &= \multonly{F'_{1,2} + F''_1, F_2, F_3}
    = \multonly{F'_{1,2}, F_3, F_2} + \multonly{F''_1, F_2, F_3}
\end{align*}
where the intersection of the first two components in each triplet has codimension two. Write the scheme theoretic intersections
\begin{align*}
F'_{1,2} \cap F_3 &= \sum_k m'_k C'_k, \\
F''_1 \cap F_2 &= \sum_i m''_{1,i} C''_{1,i} + \sum_j m''_{2,j} C''_{2,j}
\end{align*}
where $C''_{1,i}$ are irreducible curves which are contained in $\supp(F_3)$, and $C''_{2,j}$ are irreducible curves which are not contained in $\supp(F_3)$. Then
\begin{align}
\multonly{F'_1, F_2, F_3}
    &= \sum_k m'_k \multonly{C'_k, F_2}
    + \sum_i m''_{1,i} \multonly{C''_{1,i}, F_3} + \sum_j m''_{2,j} \multonly{C''_{2,j}, F_3}
    \label{eq:3-2:2}
\end{align}
It follows that
\begin{align}
\scrZ^{(2)} &= \bigcup_{k} C'_k \cup \bigcup_i C''_{1,i}
    \label{eq:3-2:2:supp}
\end{align}
and if $C$ is one of the $C'_k$ or $C''_{1,i}$, then
\begin{align}
\multordonly{F_1, F_2, F_3}{C}
    &= m'_k \deg_2(C) + m''_{1,i}\deg_3(C)
    \label{eq:3-2:2:m:0}
\end{align}
where $k$ and $i$ are such that $C'_k = C$ and $C''_{1,i} = C$ (if there is no such $k$ or $i$, then we accordingly set $m'_k = 0$ or $m'_{1,i} = 0$). If $C \not\subseteq \sing(X)$, then this multiplicity can also be expressed as:
\begin{align}
\multordonly{F_1, F_2, F_3}{C}
    &= \multonlysub{F'_{1,2}, F_3}{C}\deg_2(C) + \multonlysub{F''_1, F_2}{C}\deg_3(C)
    \label{eq:3-2:2:m}
\end{align}
which is in line with assertion \eqref{local=mult:1:2} of \Cref{prop:local=mult:1}.

\subsubsection{} \label{3-2:3}
Finally, from the last sum in the right hand side of \eqref{eq:3-2:2} we deduce
\begin{align}
\scrZ^{(3)} &= \bigcup_j C''_{2, j} \cap \supp(F_3)
    \label{eq:3-2:3:supp}
\end{align}
and if $P \in \scrZ^{(3)}$, then
\begin{align}
\multordonly{F_1, F_2, F_3}{P}
    &= \sum_{C''_{2,j} \ni P} m''_{2,j}\ord_P(f_{3,P}|_{C''_{2,j}})
    \label{eq:3-2:3:m}
\end{align}
where $f_{3,P}$ is a local equation of $F_3$ near $P$. If $C''_2 := \sum_{C''_{2,j} \ni P} m''_{2,j}C''_{2,j}$ is a local complete intersection near $P$ defined by say $f''_1, f''_2$, then of course $\multordonly{F_1, F_2, F_3}{P}$ would equal the intersection multiplicity of $f''_1, f''_2, f_{3, P}$ at $P$; however, $C''_2$ may {\em not} be a complete intersection at $P$ even when $X$ is nonsingular at $P$. It follows that unlike codimension $\leq 2$ cases, at a codimension three distinguished component, even if it is not contained in $\sing(X)$, the ordered intersection multiplicity can {\em not} in general be represented as a (sum of) product(s) of degree times a usual intersection multiplicity.
\end{example}

\begin{example} \label{ex:3-2:specific}
As an example of the scenario described at the end of the preceding example, consider
\begin{align*}
f_1 &= x^3-yz\\
f_2 &= y^2-xz\\
f_3 &= x^2 + xz^2 + y^3 + yz^3
\end{align*}
on $\kk^3$, and let $F_i$ be the closure of $V(f_i)$ on $\pp^3$. Scheme theoretically,
\begin{align}
F_1 \cap F_2 &= C + Z \label{eq:3-2:specific:12}
\end{align}
where $Z := V(x,y)$ is the ``$z$-axis'', and $C$ is the (closure of the) monomial curve $\{(t^3, t^4, t^5): t \in \kk\}$. %
It follows that
\begin{align*}
\bigcap_{i=1}^3 \supp(F_i) = Z \cup S
\end{align*}
for a set $S$ of isolated points. We will compute the number $\multisoonly{F_1, F_2, F_3}{\pp^3}$ (counted with appropriate multiplicities) of isolated points in $\bigcap_i \supp(F_i)$. First note that $\scrZ^{(2)} = \emptyset$, $\scrZ^{(1)} = Z$, and consequently, it follows from identity \eqref{eq:=multiso} on page \pageref{eq:=multiso} that
\begin{align}
\multisoonly{F_1, F_2, F_3}{\pp^3}
    &= \multonly{F_1, F_2, F_3} - \multordonly{F_1, F_2, F_3}{Z} - \sum_{P \in Z} \multordonly{F_1, F_2, F_3}{P}
    \label{eq:3-2:specific:0}
\end{align}
Identities \eqref{eq:3-2:2:m} and \eqref{eq:3-2:specific:12} imply that
\begin{align*}
\multordonly{F_1, F_2, F_3}{Z} &= \multonlysub{F_1, F_2}{Z}\deg_3(Z) = 1 \times 4 = 4
\end{align*}
It follows from \eqref{eq:3-2:3:supp} that the codimension three distinguished components of the ordered intersection of the $F_i$ are simply the points in $C \cap \supp(F_3)$. Consequently, the only nonzero contributions to the last sum of the right hand side of \eqref{eq:3-2:specific:0} come from the points in $Z \cap C$. There are two such points in $\pp^3$, namely the origin $P_0$ in $\kk^3$, and the point $P_\infty$ at infinity on the $z$-axis. Now, it is well known that $C$ is {\em not} a complete intersection near $P_0$ (see e.g.\ \cite[Exercise 1.11]{hart}), so that $\multordonly{F_1, F_2, F_3}{P_0}$, as expressed in identity \eqref{eq:3-2:3:m}, can {\em not} be readily expressed as the usual intersection multiplicity of polynomials at $P_0$ (in \Cref{prop:newton1*} below we provide an estimate which is sometimes useful in these scenarios). In any event, for the present example, we can compute via the explicit parametrization $\phi:t \mapsto (t^3, t^4, t^5)$ of $C$ near $P_0$ that
\begin{align*}
\multordonly{F_1, F_2, F_3}{P_0}
    &= \multonlysub{F_1,F_2}{C}\ord_t(\phi^*(f_3)) = 1 \times 6 = 6
\end{align*}
Similarly, one can compute that
\begin{align*}
\multordonly{F_1, F_2, F_3}{P_\infty}
    & = 1
\end{align*}
It then follows from B\'ezout's theorem and \eqref{eq:3-2:specific:0} that
\begin{align*}
\multisoonly{F_1, F_2, F_3}{\pp^3}
    &
    = 3\times 2 \times 4 - 4 - 6 - 1
    = 13
\end{align*}

\end{example}

\section{Computation in terms of Newton diagrams} \label{sec:newton}
We continue to work in the set up of \Cref{sec:local-set-up}. \Cref{prop:local=mult:1} shows that the ``local factors'' of ordered intersection multiplicity at a subvariety $Z$ of $X$ can sometimes be represented in terms of intersection multiplicities of irreducible components of complete intersection. In this section we use the results of \cite{howmanyzeroes} to estimate this intersection multiplicity in terms of related ``Newton diagrams'' (provided $Z$ contains nonsingular points of $X$), and give a few examples of how this can be applied to the affine B\'ezout problem of counting isolated points of intersections.

\subsection{(Usual) intersection multiplicity at a nonsingular point} \label{sec:newton:0}
In this section we recall basic results {\cite[Theorems IX.1, IX.8]{howmanyzeroes}} on the computation of intersection multiplicity of $n$ hypersurfaces at a nonsingular point on an $n$-dimensional variety and some relevant notation that we use in subsequent sections.

\subsubsection{}
The $n$-dimensional {\em mixed volume} is the unique symmetric multiadditive functional $\mv$ on $n$-tuples of polytopes in $\rr^n$ such that $\mv(\scrP, \ldots, \scrP) = n!\vol_n(\scrP)$ for each polytope $\scrP$, where $\vol_n$ is the $n$-dimensional Euclidean volume.

\subsubsection{} \label{notation:nd}
Let $\scrA$ be a (possibly infinite) subset of $\znonnegg{n}$. The convex hull of $\scrA + \rnonnegg{n}$ in $\rr^n$ is a convex polyhedron; the {\em Newton diagram} $\nd(\scrA)$ of $\scrA$ is the union of the compact faces of this polyhedron. Given variables $x_1, \ldots, x_n$, the {\em support} of a power series $f = \sum_\alpha c_\alpha x^\alpha$ in $(x_1, \ldots, x_n)$, denoted $\supp(f)$, is the set of all $\alpha$ such that $c_\alpha \neq 0$. The Newton diagram of $f$, denoted $\nd(f)$, is the Newton diagram of $\supp(f)$.

\subsubsection{} \label{notation:[]I}
We write $[n] := \{1, \ldots, n\}$. If $I \subseteq [n]$ and $k$ is a field, we write $k^I$ for the ``$|I|$-dimensional coordinate subspace'' $\{(x_1, \ldots, x_n) \in k^n: x_i =  0\ \text{if}\ i \not\in I\}$ of $k^n$. By $\pi_I: k^n \to k^I$ we denote the natural projection in the coordinates indexed by $I$.

\subsubsection{}
Let $\nu$ be a weighted order on $\kk[x_1, \ldots, x_n]$ corresponding to weights $\nu_j$ for $x_j$, $j = 1, \ldots, n$. We identify $\nu$ with the element in $\rnstar$ with coordinates $(\nu_1, \ldots, \nu_n)$ with respect to the dual basis. For $\scrS \subseteq \rr^n$, we write $\In_\nu(\scrS) := \{\alpha \in \scrS: \langle \nu, \alpha \rangle = \inf_\scrS(\nu)\}$. The ``initial form'' of $f = \sum_\alpha c_\alpha x^\alpha$ with respect to $\nu$ is $\In_\nu(f) := \sum_{\alpha \in \In_\nu(\supp(f))} c_\alpha x^\alpha$. More generally, if $\supp(f) \subseteq \scrS$, we write
\begin{align*}
\In_{\scrS, \nu}(f)
	&:= \sum_{\alpha \in \In_\nu(\scrS)} c_\alpha x^\alpha
    =
	\begin{cases}
	\In_\nu(f) & \text{if}\ \supp(f) \cap \In_\nu(\scrS) \neq \emptyset,\\
	0 & \text{otherwise.}
	\end{cases}
\end{align*}
We say that $\nu$ is {\em centered at the origin} if each $\nu_i$ is positive and that $\nu$ is {\em primitive} if it is nonzero and the greatest common divisor of $\nu_1, \ldots, \nu_n$ is $1$. If $\nu$ is centered at the origin, then it also extends to a weighted order on the ring of power series in $(x_1, \ldots, x_n)$.

\subsubsection{} \label{notation:mult*}
Given Newton diagrams $\Gamma_1, \ldots, \Gamma_n$ in $\rr^n$, define
\begin{align*}
\multGammazero
    &:= \min\{\multfzero: \supp(f_j) \subseteq \Gamma_j  + \rnonnegg{n},\ j = 1, \ldots, n \}
\end{align*}
We now recall the basic result regarding $\multGammazero$. It gives an expression of $\multGammazero$ in terms of the following quantity:
\begin{align*}
\multzerostar{\Gamma_1}{\Gamma_n}
	&:=  \sum_{\nu \in \scrV'_0} \min_{\Gamma_1}(\nu) ~
		\mv'_\nu(\In_\nu(\Gamma_2), \ldots, \In_\nu(\Gamma_n))
\end{align*}
where $\scrV'_0$ is the set of primitive weighted orders centered at the origin, and $\mv'_\nu$ is the $(n-1)$-dimensional mixed volume of $\psi_\nu(\In_\nu(\Gamma_j) + \alpha_j)$, $j = 2, \ldots, n$, where $\alpha_j$ is an arbitrary element in $\zz^n$ such that $\In_\nu(\Gamma_j) + \alpha_j \subseteq \rnnuperp := \{\alpha \in \rr^n: \langle \nu, \alpha \rangle = 0\}$, and $\psi_\nu:\rnnuperp \cap \zz^n \to \zz^{n-1}$ is an arbitrary isomorphism of $\zz$-modules.

\begin{prothm}[{\cite[Theorems IX.1, IX.8]{howmanyzeroes}}] \label{thm:mult0}
Let $\Gamma := (\Gamma_1, \ldots, \Gamma_n)$ be a collection of Newton diagrams in $\znonnegg{n}$. For each $I \subseteq [n]$, let $\tiGamma := \{j: \Gamma_j \cap \rr^I \neq \emptyset\}$ be the set of all indices $j$ such that $\Gamma_j$ touches $\rr^I$. Define
\begin{align*}
\tGammaOne:= \{I \subseteq [n]: I \neq \emptyset,\ |\tiGamma| = |I|,\ 1 \in \tiGamma\}
\end{align*}
Then
\begin{enumerate}
\item If $0 \not\in \bigcup_j \Gamma_j$ and there is $I \subset [n]$ such that $|\tiGamma| < |I|$, then $\multGammazero = \infty$.
\item \label{mult0:formula} Otherwise
\begin{align*}
\multGammazero
	&= \sum_{I \in \tGammaOne}
							\multzerostar{\Gamma_1 \cap \rr^I, \Gamma_{j_2} \cap \rr^I}{\Gamma_{j_{|I|}} \cap \rr^I} \\
    &\qquad \qquad \quad
							\times
							\multzero{\pi_{[n]\setminus I}(\Gamma_{j'_1})}{\pi_{[n]\setminus I}(\Gamma_{j'_{n-|I|}})}
\end{align*}
where for each $I \in \tGammaOne$, $j_1 = 1, j_2, \ldots, j_{|I|}$ are elements of $\tiGamma$, and $j'_1, \ldots, j'_{n-|I|}$ are elements of $[n]\setminus \tiGamma$.
\end{enumerate}
Pick $f_1, \ldots, f_n \in \kk[[x_1, \ldots, x_n]]$ such that $\supp(f_j) \subseteq \Gamma_j  + \rnonnegg{n}$, $j = 1, \ldots, n$. Then it is clear that $\multfzero \geq \multGammazero$.
\begin{enumerate}[resume]
\item \label{mult0:nondegeneracy} In the situation of assertion \eqref{mult0:formula}, $\multfzero = \multGammazero$ if and only if the following ``non-degeneracy'' condition holds\footnote{There is a more ``efficient'' version of the non-degeneracy condition which allows to disregard certain subsets of $[n]$ {\cite[Theorem IX.9]{howmanyzeroes}}.}: for each nonempty subset $I$ of $[n]$ and each weighted order $\nu$ centered at the origin, there is no common root of $\In_{\Gamma_j \cap \rr^I, \nu}(f_j|_{\kk^I})$, $j = 1, \ldots, n$, on $(\kk\setminus\{0\})^n$. \qed
\end{enumerate}
\end{prothm}


\begin{proexample} \label{ex:3-1-detailed:p0}
Consider
\begin{align*}
f_1 &= x^4 + xy + y^2(1-y+7yz) \\
f_2 &= 2xz^2 + y(1-y)(4+z^2) \\
f_3 &= x^2(5+8xz) + y(1-y+3z^3)
\end{align*}
We will compute the number $\multisoonly{f_1, f_2, f_3}{\kk^3}$ of isolated points in $V(f_1, f_2, f_3)$ (with appropriate multiplicities). For convenience in this example we {\em fix} the characteristic of $\kk$ to zero (essentially since the Newton diagrams depend on characteristics). 
Let $F_i$ be the closure of $V(f_i)$ on $\pp^3$. Then $\bigcap_i \supp(F_i)$ consists of the ``$z$-axis'' $Z := V(x,y)$, and finitely many points. Since $\dim(Z) = 1$, it follows that $\scrZ^{(2)} = \emptyset$ and $\scrZ^{(1)} = Z$ (where $\scrZ^{(\rho)}$ denotes the union of the codimension $\rho$ distinguished components of the ordered intersection of $F_1, F_2, F_3$), and consequently, identity \eqref{eq:=multiso} on page \pageref{eq:=multiso} implies that
\begin{align*}
\multisoonly{F_1, F_2, F_3}{\pp^3}
    &= \multonly{F_1, F_2, F_3}
        - \multordonly{F_1, F_2, F_3}{Z}
        - \sum_{P \in Z} \multordonly{F_1, F_2, F_3}{P}
\end{align*}
Note that
\begin{align*}
\multisoonly{f_1, f_2, f_3}{\kk^3}
    &= \multisoonly{F_1, F_2, F_3}{\pp^3}
        - \sum_{
            \mathclap{
                \substack{
                    P \in \pp^3 \setminus \kk^3\\
                    P\ \text{isolated in}\ \bigcap_i \supp(F_i)
                }
            }
        } \multonlysub{F_1, F_2, F_3}{P}
\end{align*}
An examination of the leading forms of the $f_i$ shows that $\bigcap_i\supp(F_i) \setminus \kk^3$ consists of two points: $P_z := [0:0:1:0]$ and $P_y := [0:1:0:0]$ with respect to homogeneous coordinates $[x:y:z:w]$. Since $P_z \in Z$, it follows that
\begin{align}
 \multisoonly{f_1, f_2, f_3}{\kk^3}
    &= \multonly{F_1, F_2, F_3}
        - \multordonly{F_1, F_2, F_3}{Z}
        - \sum_{P \in Z} \multordonly{F_1, F_2, F_3}{P}
        - \multonlysub{F_1, F_2, F_3}{P_y} \notag \\
    &= 64
        - \multordonly{F_1, F_2, F_3}{Z}
        - \sum_{P \in Z} \multordonly{F_1, F_2, F_3}{P}
        - \multonlysub{F_1, F_2, F_3}{P_y}
        \label{eq:3-1-detailed:0}
\end{align}
since $\multonly{F_1, F_2, F_3} = \prod_{i=1}^3 \deg(f_i) = 64$ due to B\'ezout's theorem. We now compute $\multonlysub{F_1, F_2, F_3}{P_y}$ using \Cref{thm:mult0}. In coordinates $(x',z',w') := (x/y, z/y, 1/y)$, the local equations of the $F_i$ at $P$ are:
\begin{align*}
f_{1,y} &= x'^4 + x'w'^2 + w'^2 - w' + 7z' \\
f_{2,y} &= 2x'z'^2w' + (w'-1)(4w'^2+z'^2) \\
f_{3,y} &= x'^2(5w'^2+8x'z') + w'^3-w'^2+3z'^3
\end{align*}

\def\shiftone{7.5}
\def\colorzero{blue}
\def\colorone{red}
\def\colortwo{yellow}
\def\colorfour{green}
\def\opazero{0.5}
\def\viewx{75}
\def\viewy{30}
\def\tx{1}
\def\ty{-1}

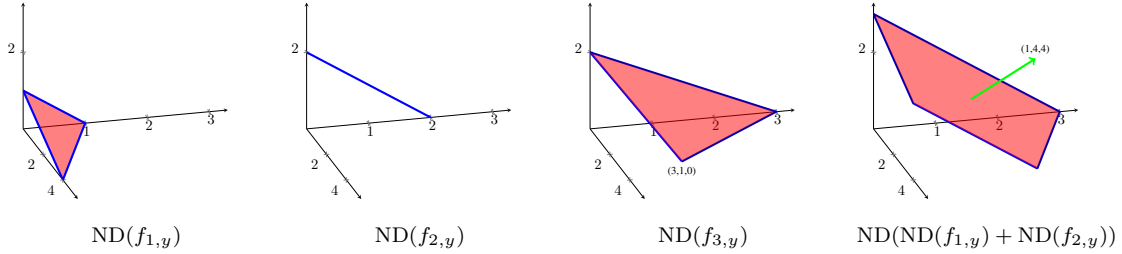
\begin{figure}[h]
\begin{center}
\begin{tikzpicture}[scale=0.5]
\pgfplotsset{every axis title/.append style={at={(0,-0.2)}}, view={\viewx}{\viewy}, axis lines=middle, enlargelimits={upper}}

\begin{scope}
\begin{axis}[
        xmin=0,
        xmax=5,
        ymin=0,
        ymax=3,
        zmin=0,
        zmax=3,
]
	\addplot3[fill=\colorone,opacity=\opazero] coordinates{(4,0,0) (0,1,0) (0,0,1)};
    \addplot3 [\colorzero, ultra thick] coordinates{(4,0,0) (0,1,0) (0,0,1) (4,0,0)};
\end{axis}
\node (A) at (\tx, \ty){};
\node (B) [right of=A] {\picfontsize $\nd(f_{1,y})$};
\end{scope}

\begin{scope}[shift={(\shiftone,0)}]
\begin{axis}[
        xmin=0,
        xmax=5,
        ymin=0,
        ymax=3,
        zmin=0,
        zmax=3,
]
    \addplot3 [\colorzero, ultra thick] coordinates{(0,2,0) (0,0,2)};
\end{axis}
\node (A) at (\tx, \ty){};
\node (B) [right of=A] {\picfontsize $\nd(f_{2,y})$};
\end{scope}

\begin{scope}[shift={(2*\shiftone,0)}]
\begin{axis}[
        xmin=0,
        xmax=5,
        ymin=0,
        ymax=3,
        zmin=0,
        zmax=3,
]
    \addplot3 [\colorzero, ultra thick] coordinates{(3,1,0) (0,3,0) (0,0,2) (3,1,0)};
	\addplot3[fill=\colorone,opacity=\opazero] coordinates{(3,1,0) (0,3,0) (0,0,2)};
	\draw (axis cs:3,1,0) node [below] {\picfontsize (3,1,0)};
\end{axis}
\node (A) at (\tx, \ty){};
\node (B) [right of=A] {\picfontsize $\nd(f_{3,y})$};
\end{scope}

\begin{scope}[shift={(3*\shiftone,0)}]
\begin{axis}[
        xmin=0,
        xmax=5,
        ymin=0,
        ymax=3,
        zmin=0,
        zmax=3,
]
    \addplot3 [\colorzero, ultra thick] coordinates{(4,2,0) (0,3,0) (0,0,3) (4,0,2) (4,2,0)};
	\addplot3[fill=\colorone,opacity=\opazero] coordinates{(4,2,0) (0,3,0) (0,0,3) (4,0,2)};
	\addplot3 [ultra thick, \colorfour, ->] coordinates{(2,1.25,1.25) (2.25, 2.25, 2.25)};
	\draw (axis cs:2.25, 2.25, 2.25) node [above] {\picfontsize (1,4,4)};
\end{axis}
\node (A) at (\tx, \ty){};
\node (B) [right of=A] {\picfontsize $\nd(\nd(f_{1,y}) + \nd(f_{2,y}))$};
\end{scope}

\end{tikzpicture}
\end{center}
\caption{
Newton diagrams of the $f_{i,y}$
} \label{fig:3-1-detailed:0}
\end{figure}

To compute $\multonlysub{f_{1,y}, f_{2,y}, f_{3,y}}{P_y}$ we apply \Cref{thm:mult0} with $(\Gamma_1, \Gamma_2, \Gamma_3) := (\nd(f_{3,y}), \nd(f_{2,y}), \allowbreak \nd(f_{1,y}))$ (we could have used any permutation of the $\nd(f_{i,y})$ - this particular choice leads to a relatively shorter computation). Then 
the set $\tGammaOne$ from \Cref{thm:mult0} consists only of one $I$, namely $\{1, 2, 3\}$ (see \Cref{fig:3-1-detailed:0}), so that
\begin{align*}
\multonlysub{f_{1,y}, f_{2,y}, f_{3,y}}{P_y}
    &\geq \multzeroonly{\nd(f_{3,y}), \nd(f_{2,y}), \nd(f_{1,y})}
	= \multzerostaronly{\nd(f_{3,y}), \nd(f_{2,y}), \nd(f_{1,y})} \\
    &= \min_{\nd(f_{3,y})}(\nu) \times
		\mv'_{\nu}(\In_{\nu}(\nd(f_{1,y})), \In_{\nu}(\nd(f_{2,y})))\quad (\text{where}\ \nu := (1, 4, 4)) \\
    &= 7 \times 2 = 14
\end{align*}
It is not hard to check that $(f_{3,y}, f_{1,y}, f_{2,y})$ satisfy the non-degeneracy condition from assertion \eqref{mult0:nondegeneracy} of \Cref{thm:mult0}, so that
\begin{align}
\multonlysub{f_{1,y}, f_{2,y}, f_{3,y}}{P_y} &= 14 \label{eq:3-1-detailed:0y}
\end{align}
We will compute the other terms of identity \eqref{eq:3-1-detailed:0} in Examples \ref{ex:3-1-detailed:p1} and \ref{ex:3-1-detailed:p2}.
\end{proexample}

\subsection{(possibly non-isolated) complete intersections} \label{sec:newton:complete}
Consider an $(n-k)$-dimensional irreducible component $Z$ of $V(f_1, \ldots, f_k)$, where the $f_j$ are regular functions on an $n$-dimensional affine variety $X$. Assume in addition that $Z \not\subseteq \sing(X)$. Under this condition we show that the intersection multiplicity $\multsub{f_1}{f_k}{Z}$ of $f_1, \ldots, f_k$ along $Z$ can be estimated in terms of certain Newton diagrams. Indeed, since $Z \not\subseteq \sing(X)$, there are ``regular coordinates''\footnote{We say $(g_1, \ldots, g_n)$ are {\em regular coordinates} on a nonsingular variety $Y$ if $\dim(Y) = n$ and for each $y \in Y$, $\hatlocal{Y}{y} \cong \kk[[g_1 - g_1(y), \ldots, g_n - g_n(y)]]$.} $(x_1, \ldots, x_n)$ on an open subset $U$ of $X$ such that $U \cap Z = V(x_1, \ldots, x_k) \cap U \neq \emptyset$. Then there are $f_{j,\alpha} \in \kk[U]$ such that for each $a \in U$,
\begin{align*}
f_j &= \sum_{\alpha \in \znonnegg{n}} f_{j,\alpha} \prod_{i=1}^{n}(x_i - a_i)^{\alpha_i} \in \kk[[x_1 - a_1, \ldots, x_n - a_n]]
\end{align*}
where $a_i := x_i(a)$ (see e.g.\ \cite[Lemma 3.5]{bierstone-milman}). Define
\begin{align*}
f^{[k]}_j &:= \sum_{\alpha \in \zz^{[k]}} f_{j,\alpha} \prod_{i=1}^{k}x_i^{\alpha_i} \in \kk[U][[x_1, \ldots, x_k]]
\end{align*}
(recall from \Cref{notation:[]I} that $\zz^{[k]}$ is the subset of all elements in $\zz^n$ whose last $(n-k)$-coordinates are identically zero). The {\em Newton diagram} $\nd(f^{[k]}_j)$ with respect to $(x_1, \ldots, x_k)$ is the Newton diagram of the set of all $\alpha$ such that $f_{j, \alpha} \not\equiv 0$. 

\begin{proprop} \label{cor:mult0}
Let $\pi: \rr^n \to \rr^k$ be the projection on the first $k$ coordinates. Then
\begin{enumerate}
\item \label{cor:mult0:formula} $\multsub{f_1}{f_k}{Z} \geq \multsub{\pi(\nd(f^{[k]}_1))}{\pi(\nd(f^{[k]}_k))}{\origin}$
\item \label{cor:mult0:nondegen} If the right hand side is finite, then the above relation holds with an equality if and only if there is $a \in U \cap Z$ such that the following pair of ``non-degeneracy'' conditions hold:
\begin{defnlist}
\item \label{cor:mult0:nondegen:vertex} for each $j = 1, \ldots, k$, and each vertex $\alpha$ of $\pi(\nd(f^{[k]}_j))$, the coefficient $f_{j,\alpha}$ is nonzero at $a$, and
\item \label{cor:mult0:nondegen:zero} for each nonempty subset $I$ of $[k]$ and each weighted order $\nu$ centered at the origin, there is no common root of $\In_\nu(f^{[k]}_j|_{\kk^I})$, $j = 1, \ldots, k$, on $\{a\} \times (\kk\setminus\{0\})^k$.
\end{defnlist}
\item \label{cor:mult0:irrcomp} Moreover, the set of all $a \in U \cap Z$ such that the above non-degeneracy conditions hold is a (possibly empty) Zariski open subset of $Z$, and every such $a$ (if exists) has an open neighborhood $U_a$ on $X$ such that $V(f_1, \ldots, f_k) \cap U_a$ is a complete intersection with only one irreducible component, namely $Z \cap U_a$.
\end{enumerate}
\end{proprop}

\begin{proof}
Since $\multsub{f_1}{f_k}{Z} = \multzero{f_1}{f_k,x_{k+1} - a_{k+1}, \ldots, x_n - a_n}$ for generic $a \in U \cap Z$, all statements except for the last one are immediate corollaries of \Cref{thm:mult0}. The Zariski openness of the set of $a \in U \cap Z$ satisfying the non-degeneracy conditions is a consequence of the Zariski openness of the $n$-tuples of polynomials which are non-degenerate at the origin \cite[Theorem IX.8]{howmanyzeroes}. For the local complete intersection property, pick $a \in U \cap Z$ which satisfies condition \ref{cor:mult0:nondegen:vertex} and in addition, also belongs to another irreducible component $Z'$ of $V(f_1, \ldots, f_k)$. Pick an irreducible curve $C$ on $Z'$ such that $a \in C$ and $C \not\subseteq Z$. Let $I$ be the subset of $[k]$ consisting of all $i$ such that $x_i|_C \not\equiv 0$. If $B$ is a branch on $C$ at $a$, then condition \ref{cor:mult0:nondegen:zero} is violated with $I$ and $\nu := (\ord_B(x_i))_{i \in I}$. This completes the proof of the last assertion of \Cref{cor:mult0}.
\end{proof}

\begin{proexample} \label{ex:3-1-detailed:p1}
We continue with the computation in characteristic zero started in \Cref{ex:3-1-detailed:p0}. Now we compute the ordered intersection multiplicity of $F_1, F_2, F_3$ along $Z$, i.e.\ the term $\multordonly{F_1, F_2, F_3}{Z}$ from the right hand side of identity \eqref{eq:3-1-detailed:0}. Since $Z$ is an irreducible component of $V(f_1, f_2)$, assertion \ref{local=mult:1:0} of \Cref{prop:local=mult:1} implies that
\begin{align*}
\len(\local{X}{Z}/I^{(\rho)}_{1,2} \local{X}{Z})
    &= \multonlysub{f_1, f_2}{Z}
\end{align*}

\def\picfontsize{\scriptsize}
\def\colorzero{blue}
\def\colorone{blue}
\def\colortwo{red}

\begin{center}
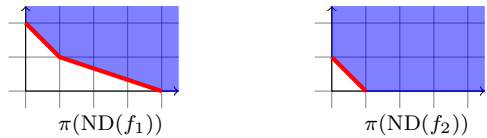
\begin{figure}[htp]
\begin{tikzpicture}[scale=0.45]
\def\shiftone{9}
\def\opazero{0.5}
\def\tx{2.5}
\def\ty{-1}
\def\gridx{4.5}
\def\gridy{2.5}

\draw [gray,  line width=0pt] (-0.5,-0.5) grid (\gridx,\gridy);
\draw [<->] (0, \gridy) |- (\gridx, 0);
\draw[\colorone] (0,\gridy) -- (0,2) -- (1,1) -- (4,0) -- (\gridx,0);
\fill[\colorzero, opacity=\opazero ] (0,\gridy) -- (0,2) -- (1,1) -- (4,0) -- (\gridx,0) -- (\gridx,\gridy) -- cycle;
\draw[ultra thick, \colortwo]  (0,2) -- (1,1) -- (4,0);
\draw (\tx,\ty) node {\picfontsize $\pi(\nd(f_1))$};

\begin{scope}[shift={(\shiftone,0)}]
	\draw [gray,  line width=0pt] (-0.5,-0.5) grid (\gridx,\gridy);
	\draw [<->] (0, \gridy) |- (\gridx, 0);
	\draw[\colorone] (0,\gridy) -- (0,1) -- (1,0) -- (\gridx,0);
    \fill[\colorzero, opacity=\opazero ] (0,\gridy) -- (0,1) -- (1,0) -- (\gridx,0) -- (\gridx,\gridy) -- cycle;
    \draw[ultra thick, \colortwo]  (0,1) -- (1,0);
	\draw (\tx,\ty) node {\picfontsize $\pi(\nd(f_2))$};
\end{scope}

\end{tikzpicture}

\caption{``Generic'' Newton diagrams of $f_1, f_2$ with respect to $(x, y)$-coordinates} \label{fig:3-1-1}
\end{figure}
\end{center}

Since $Z = V(x,y)$, in the notation of \Cref{cor:mult0} we need to consider $f_j^{\{1,2\}}$, $j = 1, 2$, which are simply $f_j$ but where we consider monomials in $z$ as ``coefficients''. \Cref{fig:3-1-1} displays the Newton diagrams of $f_j^{\{1,2\}}$, which are also identical to $\pi(\nd(f_j))$ where $\pi$ is the projection onto $(x,y)$-coordinates. Since $\pi(\nd(f_2))$ has only one edge $E$, and the edge has inner normal $\nu := (1,1)$, the first assertion of \Cref{cor:mult0} and \Cref{thm:mult0} imply that
\begin{align}
\multonlysub{f_1, f_2}{Z}
    &\geq \multonlysub{\pi(\nd(f_1)), \pi(\nd(f_2))}{\origin}
    = \nu(f_1) \times \len(E)
    = 2 \times 1
    = 2
    \label{ineq:3-1-1}
\end{align}
The initial forms of $f^{\{1,2\}}_j$ corresponding to the weighted order $\nu := (1,1)$ are:
\begin{alignat*}{2}
\In_\nu(f^{\{1,2\}}_1) &= xy + y^2 \qquad \quad & \In_\nu(f^{\{1,2\}}_2) &= 2xz^2 + y(4 + z^2)
\end{alignat*}
The second assertion of \Cref{cor:mult0} implies that the inequality \eqref{ineq:3-1-1} is satisfied with an equality if and only if there is $z \in \kk$ such that following non-degeneracy condition is satisfied:
\begin{align}\label{nondegen:3-1-1}
\parbox{0.9\textwidth}{
there is no common solution to $\In_\nu(f^{\{1,2\}}_1) = \In_\nu(f^{\{1,2\}}_2) = 0$ with $xyz(4+z^2) \neq 0$.
}
\end{align}
It is clear that the above condition is satisfied for almost all $z \in \kk$. It follows that
\begin{align}
\begin{split} \label{eq:3-1-1}
\multonlysub{f_1, f_2}{Z}
    & = 2,\ \text{and therefore,} \\
\multordonly{F_1, F_2, F_3}{Z}
    &= \multonlysub{f_1, f_2}{Z} \times \deg_3(Z)
    = 2 \times 4
    = 8
\end{split}
\end{align}
In order to compute $\multisoonly{f_1, f_2, f_3}{\kk^3}$ using \eqref{eq:3-1-detailed:0}, it remains to compute $\sum_{P \in Z} \multordonly{F_1, F_2, F_3}{P}$. Write $F_1 \cap F_2 = mZ + C$, where the $C$ is a (possibly non-reduced) curve whose support does not contain $Z$. Then
\begin{align}\label{eq:3-1-0:0}
\sum_{P \in Z} \multordonly{F_1, F_2, F_3}{P}
    &= \sum_{P \in Z} \ord_P(f_{3,P}|_C)
\end{align}
where $f_{3,P}$ are local equations of $F_3$ near $P$ (see, e.g.\ \Cref{ex:3-2}). Note that the summand on the right hand side is nonzero only for the (finitely many) points $P$ in $Z \cap \supp(C)$. The third assertion of \Cref{cor:mult0} implies that every such $P$ must satisfy one of the following:
\begin{prooflist}
\item $P = (0,0,z)$ where $z$ is such that Condition \eqref{nondegen:3-1-1} does {\em not} hold, i.e.
\begin{prooflist}
\item \label{3-1-detailed:p2:v0:0} $z = 0$,
\item \label{3-1-detailed:p2:v0:2} or $4+z^2 = 0$,
\item \label{3-1-detailed:p2:degen} or $z(4+z^2) \neq 0$, but there is a common solution to $\In_{(1,1)}(f^{\{1,2\}}_1) = \In_{(1,1)}(f^{\{1,2\}}_2) = 0$ at $z$ with $x,y,z \in \kk \setminus \{0\}$ (this is possible if and only if $z = \pm 2$);
\end{prooflist}
\item \label{3-1-detailed:p2:infty} or $P$ does {\em not} belong to the coordinate chart of $X$ with coordinates $(x,y,z)$, i.e.\ $P$ is the point $P_\infty := [0:0:1:0]$ with respect to homogeneous coordinates $[x:y:z:w]$.
\end{prooflist}
We will compute $\ord_P(F_3|_C)$ in each of these cases in \Cref{ex:3-1-detailed:p2} below.
\end{proexample}

\subsection{(a special type of) non set-theoretic complete intersections} \label{sec:newton0}
Consider \Cref{ex:3-1-detailed:p1} above, and take one of the points, say $P= (0,0,0)$, that satisfies one of the two conditions given at the end of the example. We need to compute $\ord_P(f_3|_C)$; at the least we would like to estimate it {\em from below}. If $C$ were a complete intersection of $f_1, f_2$ near $P$, then $\ord_P(f_3|_C)$ would simply be the intersection multiplicity $\multonlysub{f_1, f_2, f_3}{P}$ of $f_1, f_2, f_3$ at $P$, and we could have used \Cref{thm:mult0} to estimate it. However, $C$ is {\em not} a complete intersection of $f_1, f_2$ near $P$. As we have seen, $V(f_1, f_2)$ has another irreducible component $Z$ containing $P$, so that $\ord_P(f_3|_C) < \multonlysub{f_1, f_2, f_3}{P}$, i.e.\ $\multonlysub{f_1, f_2, f_3}{P}$ can {\em not} be used to estimate $\ord_P(f_3|_C)$ from below (in fact in this particular case $f_3|_Z \equiv 0$, so that $\multonlysub{f_1, f_2, f_3}{P} = \infty$). The main result of this section (\Cref{prop:newton1*}) can sometimes be used to give such an estimate for $\ord_P(f_3|_C)$ in the case that the ``punctured'' germ of $C$ at $P$ is a complete intersection in the complement of a union of ``coordinate subspaces''.

\subsubsection{}
First we state a variant of the non-degeneracy condition from \Cref{thm:mult0}: we say that $f_1, \ldots, f_m \in \kk[[x_1, \ldots, x_n]]$ are $(*, k)$-non-degenerate at the origin if they satisfy the following condition:
\begin{align}
\parbox{0.87\textwidth}{
for each weighted order $\nu$ centered at the origin such that $\dim(\In_\nu(\sum \supp(f_j))) \leq k$, 
there is no common root of $\In_\nu(f_j)$, $j = 1, \ldots, m$, on $(\kk\setminus\{0\})^n$.
}
\tag{$\nd^{*,k}_0$}
\label{eq:non-degeneracy-*k}
\end{align}
We say that $f_1, \ldots, f_m$ are $*$-non-degenerate at the origin if they are $(*, k)$-non-degenerate at the origin for each $k \geq 0$.

\subsubsection{} \label{par:newton0*}
Let $z$ be a nonsingular point on a variety $X$ of dimension $n$, and $(x_1, \ldots, x_n)$ be a system of regular coordinates near $z$ such that each $x_i$ vanishes at $z$, i.e.\ $z$ is the ``origin'' with respect to $(x_1, \ldots, x_n)$. Given a subset $I$ of $[n]$, we write $K^I \subseteq X$ for the set of zeroes of all $x_i$ such that $i \not\in I$ (i.e.\ $K^I$ is the analogue of the ``coordinate subspace'' $\kk^I$); in particular,
\begin{align*}
K^I &=
    \begin{cases}
    X &\text{if}\ I = [n],  \\
    V(x_1, \ldots, x_n) &\text{if}\ I = \emptyset.
    \end{cases}
\end{align*}
Let $\scrI$ be a given collection of subsets of $[n]$. 
Consider $f_1, \ldots, f_{n-1} \in \local{X}{z}$. We are interested in the irreducible components of $V(f_1, \ldots, f_{n-1})$ which contain $z$, and in addition, are {\em not} contained in $K^I$ for any $I \in \scrI$; let $V'_\scrI$ be the union of all such components. Write $\bar \scrI := \{I \subseteq [n]: I \subseteq I'$ for some $I' \in \scrI\}$, i.e.\ $\bar \scrI$ is the ``closure'' of $\scrI$ under the operation of taking subsets; note that $\bar \scrI$ is the largest collection of subsets of $[n]$ such that
\begin{align*}
\bigcup_{I \in \bar \scrI} K^I = \bigcup_{I \in \scrI} K^I
\end{align*}

\begin{proprop} \label{prop:newton0*}
For each $I \subseteq [n]$, let $\ti := \{j: f_j|_{K^I} \not\equiv 0$ near $z\}$ (i.e.\ $\ti$ is the analogue of $\tiGamma$ from \Cref{thm:mult0}). With the set up from the preceding paragraph, assume the following conditions hold for all $I \not\in \bar\scrI$:
\begin{enumerate}
\item \label{newton0*:n-nonzero} $|\ti| \geq |I| - 1$,
\item \label{newton0*:non-degen} $f_1|_{K^I}, \ldots, f_{n-1}|_{K^I}$ are $(*, |\ti|-1)$-non-degenerate at the origin.
\end{enumerate}
where we treat the $f_i$ as power series in $(x_1, \ldots, x_n)$. Then either $V'_\scrI = \emptyset$, or $V'_\scrI$ is a curve (i.e.\ $V'_\scrI$ is purely one dimensional).
\end{proprop}

\begin{proof}
Assume $V'_\scrI \neq \emptyset$. Pick an irreducible component $V'$ of $V'_\scrI$. First note that $\dim(V') \geq 1$. Let $I$ be the smallest subset of $[n]$ such that $V' \subseteq K^I$. Then $I \not\in \bar \scrI$. \Woutlog\ we may assume that $I = \{1, \ldots, m\}$. Consider a subset $\scrA$ of $\znonnegg{m}$ such that
\begin{prooflist}
\item the radical of the ideal $\qqq_\scrA$ of $\local{K^I}{z}$ generated by monomials $x^\alpha$, $\alpha \in \scrA$, is the maximal ideal of $\local{K^I}{z}$, and
\item \label{newton1*:face} every $m-1$ dimensional face of $\nd(\sum_{i=1}^{n-1} \supp(f_i|_{K^I}))$ (if there is one) is also a face of $\nd(\scrA)$.
\end{prooflist}
Consider the blow up $\tilde K^I$ of $K^I$ at $z$ with respect to $\qqq_\scrA$. The irreducible components of the exceptional divisor $E$ on $\tilde K^I$ are toric varieties $X_\scrQ$ corresponding to $m-1$ dimensional faces $\scrQ$ of $\nd(\scrA)$, and the proper torus orbits of each such $X_\scrQ$ are toric varieties $X_{\scrQ'}$ corresponding to proper faces $\scrQ'$ of $\scrQ$. Recall that $|\ti| \geq m - 1$ (property \eqref{newton0*:n-nonzero} of the $f_j$). It follows from property \eqref{newton0*:non-degen} of the $f_j$ and property \ref{newton1*:face} of $\scrA$ that
\begin{prooflist}[resume]
\item \label{newton0*:ti>m-1} if $|\ti| > m - 1$, then the strict transform $\tilde V'$ of $V'$ on $\tilde K^I$ does {\em not} intersect $X_{\scrQ}$ for any face $\scrQ$ of $\nd(\scrA)$,
\item \label{newton0*:ti=m-1} if $|\ti| = m - 1$, then $\tilde V'$ does {\em not} intersect $X_{\scrQ'}$ for any face $\scrQ$ of $\nd(\scrA)$ such that $\dim(\scrQ') \leq m - 2$
\end{prooflist}
(see e.g.\ \cite[Corollary VI.34]{howmanyzeroes}). Since case \ref{newton0*:ti>m-1} leads to a contradiction, it follows that $|\ti| = m - 1$. Bernstein's theorem then implies that the number of points in $\tilde V' \cap X_\scrQ$ (counted with appropriate multiplicity) is $\mv'_\nu(\In_\nu(\nd(f_{j_1})), \ldots, \In_\nu(\nd(f_{j_{m-1}})))$, where $\nu$ is the primitive inner normal to $\scrQ$, and $\ti = \{j_1, \ldots, j_{m-1}\}$. In particular, it follows that $\tilde V' \cap E$ is finite, and therefore, $\tilde V'$ is a curve, as required.
\end{proof}

\subsubsection{} \label{assumptions:newton1*}
Let $z$ be a nonsingular point on a variety $X$ of dimension $n$ and $C$ be a curve, i.e.\ purely one dimensional scheme, containing $z$. Assume we are interested in estimating $\ord_z(f|_C)$ for some $f \in \local{X}{z}$. \Cref{prop:newton1*} below provides an estimate under the following assumptions: there are $x_1, \ldots, x_n, \allowbreak f_1, \ldots, f_{n-1} \in \local{X}{z}$ and a collection $\scrI$ of subsets of $[n]$ such that

\begin{defnlist}
\item each $x_i$ vanishes at $z$, and $(x_1, \ldots, x_n)$ form a system of regular coordinates near $z$,
\item no irreducible component of $C$ containing $z$ is contained in $K^I := V(x_{i'}: i' \not\in I)$ for any $I \in \scrI$,
\item $f_1, \ldots, f_{n-1}$ generate the ideal of $C$ on $U \setminus \bigcup_{I \in \scrI}K^I$ for some neighborhood $U$ of $z$ in $X$ (or more generally, the irreducible components of $V(f_1, \ldots, f_{n-1})$ which contain $z$ and are not contained in $\bigcup_{I \in \scrI} K^I$ are precisely the irreducible components of $C$ containing $z$, and in addition, the multiplicity of each such component in $V(f_1, \ldots, f_{n-1})$ is also equal to its multiplicity in $C$),
\item \label{assumption:newton1*:non-degen0} when expressed as power series in $(x_1, \ldots, x_n)$, $f_1, \ldots, f_{n-1}$ satisfy conditions \eqref{newton0*:n-nonzero} and \eqref{newton0*:non-degen} from \Cref{prop:newton0*} for each $I \not\in \bar \scrI$, where $\bar \scrI$ is the closure of $\scrI$ under the operation of taking subsets (see \Cref{par:newton0*}).
\end{defnlist}

\begin{proprop} \label{prop:newton1*}
Let $\mscrT :=  \{I \subseteq [n]: I \not\in \bar \scrI,\ I \neq \emptyset,\ |\ti| = |I| - 1\}$, where $\ti$ is as in \Cref{prop:newton0*}.  Under the assumptions in \Cref{assumptions:newton1*}, 
\begin{align}
\begin{split}
\ord_z(f|_C)
    &\geq \sum_{I \in \mscrT}
            \multzerostar{\nd(f) \cap \rr^I, \nd(f_{j_1}) \cap \rr^I}{\nd(f_{j_{|I|-1}}) \cap \rr^I} \\
    &\qquad \qquad \quad
			\times
			\multzero{\pi_{[n]\setminus I}(\nd(f_{j'_1}))}{\pi_{[n]\setminus I}(\nd(f_{j'_{n-|I|}}))}
\end{split}
\label{eq:newton1*}
\end{align}
where for each $I \in \mscrT$, $j_1, \ldots, j_{|I|-1}$ are elements of $\ti$, and $j'_1, \ldots, j'_{n-|I|}$ are elements of $[n-1]\setminus \ti$, and $\multzerostar{\cdot}{\cdot}$ is defined as in \Cref{notation:mult*}. Moreover, the above relation is satisfied with an equality if and only if $f_1|_{K^I}, \ldots, f_{n-1}|_{K^I}, f|_{K^I}$ are $*$-non-degenerate at the origin for every $I \in \mscrT$.

\end{proprop}

\begin{proof}
This follows from the arguments of the proof of \Cref{prop:newton0*} and standard toric geometry arguments (see e.g.\ \cite[Corollary VII.27]{howmanyzeroes}).
\end{proof}

\subsubsection{Remarks on the applicability of \Cref{prop:newton1*}} \label{sec:newton1*:nonapp}
\begin{enumerate}
\item The assumptions in \Cref{assumptions:newton1*} in particular implies that on the set-theoretic level there is a decomposition near $z$ of the form $V(f_1, \ldots, f_{n-1}) = C \cup W$ with $W \subseteq \bigcup_i V(x_i)$. We do not know if this is possible to achieve for an arbitrary $C$, since this puts non-trivial constraints on the singularities of $W$ (e.g.\ the local embedding dimension of each irreducible component of $W$ is $\leq n -1$).

\item The following example shows that assumption \ref{assumption:newton1*:non-degen0} in \Cref{assumptions:newton1*} is necessary: consider
    \begin{align*}
    f_1 &= (x-y+z)(x+y-z) \\
    f_2 &= (x-y-2z) \\
    f_3 &= x^2 + 3y^2 + 5z^2
    \end{align*}
    over a field of characteristic zero or greater than $5$. Then $V(f_1, f_2)$ is a curve with two components: $C_1$ with parametrization $t \mapsto (t,t,0)$, and $C_2$ with parametrization $t \mapsto (-3t, t, -2t)$. Let $z$ be the origin, $C = C_2$, and $f = f_3$. Then $\ord_z(f_3|_C) = 2$, but the bound on the right hand side of \eqref{eq:newton1*} (applied with $\scrI = \{\{1,2\}\}$) is
    \begin{align*}
    &\multzerostaronly{\nd(f_3), \nd(f_1), \nd(f_2)} \\
        &\qquad = \min_{\nd(f_3)}(\nu) \times \mv'_\nu(\In_\nu(\nd(f_2)), \In_\nu(\nd(f_3))) \quad (\text{where}\ \nu := (1, 1, 1)) \\
        &\qquad
        = 2 \times 2
        = 4 > 2
        = \ord_z(f_3|_C)
    \end{align*}
\end{enumerate}

\begin{proexample} \label{ex:3-1-detailed:p2}
We continue with \Cref{ex:3-1-detailed:p0,ex:3-1-detailed:p1}. Recall from the end of \Cref{ex:3-1-detailed:p1} that it remains to compute $\ord_P(F_3|_C)$ in four cases, where $C$ is a (possibly non-reduced) curve such that $F_1 \cap F_2 = mZ + C$, and $Z = V(x,y)$ is the $z$-axis. We will use \Cref{prop:newton1*} for this computation.

\def\shiftone{7.5}
\def\colorzero{blue}
\def\colorone{red}
\def\colortwo{yellow}
\def\colorfour{green}
\def\opazero{0.5}
\def\viewx{75}
\def\viewy{30}
\def\tx{1}
\def\ty{-1}

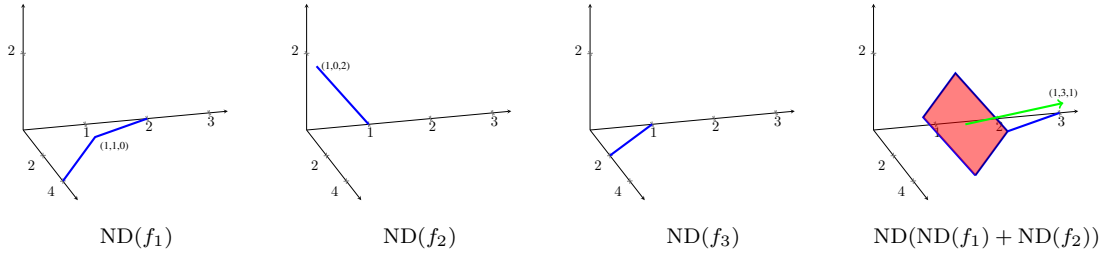
\begin{figure}[h]
\begin{center}
\begin{tikzpicture}[scale=0.5]
\pgfplotsset{every axis title/.append style={at={(0,-0.2)}}, view={\viewx}{\viewy}, axis lines=middle, enlargelimits={upper}}

\begin{scope}
\begin{axis}[
        xmin=0,
        xmax=5,
        ymin=0,
        ymax=3,
        zmin=0,
        zmax=3,
]
    \addplot3 [\colorzero, ultra thick] coordinates{(4,0,0) (1,1,0) (0,2,0)};
	\draw (axis cs:1,1,0) node [below right] {\picfontsize (1,1,0)};
\end{axis}
\node (A) at (\tx, \ty){};
\node (B) [right of=A] {\picfontsize $\nd(f_1)$};
\end{scope}

\begin{scope}[shift={(\shiftone,0)}]
\begin{axis}[
        xmin=0,
        xmax=5,
        ymin=0,
        ymax=3,
        zmin=0,
        zmax=3,
]
    \addplot3 [\colorzero, ultra thick] coordinates{(1,0,2) (0,1,0)};
	\draw (axis cs:1,0,2) node [right] {\picfontsize (1,0,2)};
\end{axis}
\node (A) at (\tx, \ty){};
\node (B) [right of=A] {\picfontsize $\nd(f_2)$};
\end{scope}

\begin{scope}[shift={(2*\shiftone,0)}]
\begin{axis}[
        xmin=0,
        xmax=5,
        ymin=0,
        ymax=3,
        zmin=0,
        zmax=3,
]
    \addplot3 [\colorzero, ultra thick] coordinates{(2,0,0) (0,1,0)};
\end{axis}
\node (A) at (\tx, \ty){};
\node (B) [right of=A] {\picfontsize $\nd(f_3)$};
\end{scope}

\begin{scope}[shift={(3*\shiftone,0)}]
\begin{axis}[
        xmin=0,
        xmax=5,
        ymin=0,
        ymax=3,
        zmin=0,
        zmax=3,
]
    \addplot3 [\colorzero, ultra thick] coordinates{(4,1,0) (1,2,0) (2,1,2) (5,0,2) (4,1,0)};
	\addplot3[fill=\colorone,opacity=\opazero] coordinates{(4,1,0) (1,2,0) (2,1,2) (5,0,2)};
    \addplot3 [\colorzero, ultra thick] coordinates{(1,2,0) (0,3,0)};
    \addplot3 [ultra thick, \colorfour, ->] coordinates{(3,1,1) (3.5, 2.5, 1.5)};
	\draw (axis cs:3.5, 2.5, 1.5) node [above] {\picfontsize (1,3,1)};
\end{axis}
\node (A) at (\tx, \ty){};
\node (B) [right of=A] {\picfontsize $\nd(\nd(f_1) + \nd(f_2))$};
\end{scope}
\end{tikzpicture}
\end{center}
\caption{
Newton diagrams of the $f_i$
} \label{fig:3-1-detailed:p2:v0:0}
\end{figure}

\paragraph{Case \ref{3-1-detailed:p2:v0:0}} 
$z = 0$, i.e.\ $P = P_0 := (0,0,0)$. Consider the usual coordinate chart $U \cong \kk^3$ with coordinates $(x,y,z)$. Since $V(f_1, f_2)$ is set-theoretically $Z \cup C$ and $Z \cap U = K^I$ with $I = \{3\}$ (in the notation of in the notation of \Cref{par:newton0*}), we can try to use \Cref{prop:newton1*} to compute $\ord_{P_0}(f_3|_C)$. 
\Cref{fig:3-1-detailed:p2:v0:0} shows the Newton diagrams of the $f_i$ with respect to $(x,y,z)$-coordinates. It is easy to see that $f_1, f_2$ satisfy condition \ref{assumption:newton1*:non-degen0} from \Cref{assumptions:newton1*}, so that \Cref{prop:newton1*} applies. Since in this case the set $\mscrT$ from \Cref{prop:newton1*} consists of only one set, namely $I := \{1, 2, 3\}$, it follows from inequality \eqref{eq:newton1*} (and the last panel of \Cref{fig:3-1-detailed:p2:v0:0}) that
\begin{align*}
\ord_{P_0}(f_3|_C)
    & \geq \multzerostaronly{\nd(f_3), \nd(f_1), \nd(f_2)} \\
    &= \min_{\nd(f_3)}(\nu) \times
		\mv'_{\nu}(\In_{\nu}(\nd(f_1)), \In_{\nu}(\nd(f_2)))\quad (\text{where}\ \nu := (1, 3, 1)) \\
    &= 2 \times 2 = 4
\end{align*}
It is straightforward to check that the $f_i$ satisfy the non-degeneracy condition of \Cref{prop:newton1*}, so that
\begin{align}
\ord_{P_0}(f_3|_C) &= 4 \label{eq:3-1-detailed:p2:v0:0}
\end{align}

\def\shiftone{7.5}
\def\colorzero{blue}
\def\colorone{red}
\def\colortwo{yellow}
\def\colorfour{green}
\def\opazero{0.5}
\def\viewx{75}
\def\viewy{30}
\def\tx{1}
\def\ty{-1}

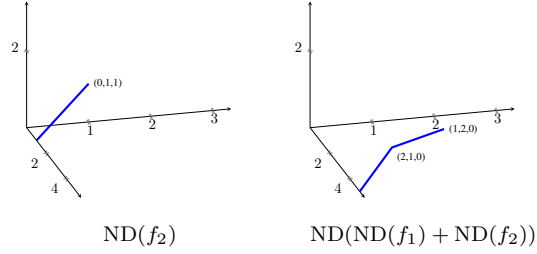
\begin{figure}[h]
\begin{center}
\begin{tikzpicture}[scale=0.5]
\pgfplotsset{every axis title/.append style={at={(0,-0.2)}}, view={\viewx}{\viewy}, axis lines=middle, enlargelimits={upper}}

\begin{scope}
\begin{axis}[
        xmin=0,
        xmax=5,
        ymin=0,
        ymax=3,
        zmin=0,
        zmax=3,
]
    \addplot3 [\colorzero, ultra thick] coordinates{(1,0,0) (0,1,1)};
	\draw (axis cs:0,1,1) node [right] {\picfontsize (0,1,1)};
\end{axis}
\node (A) at (\tx, \ty){};
\node (B) [right of=A] {\picfontsize $\nd(f_2)$};
\end{scope}

\begin{scope}[shift={(\shiftone,0)}]
\begin{axis}[
        xmin=0,
        xmax=5,
        ymin=0,
        ymax=3,
        zmin=0,
        zmax=3,
]
    \addplot3 [\colorzero, ultra thick] coordinates{(5,0,0) (2,1,0) (1,2,0)};
	\draw (axis cs:2,1,0) node [below right] {\picfontsize (2,1,0)};
    \draw (axis cs:1,2,0) node [right] {\picfontsize (1,2,0)};
\end{axis}
\node (A) at (\tx, \ty){};
\node (B) [right of=A] {\picfontsize $\nd(\nd(f_1) + \nd(f_2))$};
\end{scope}
\end{tikzpicture}
\end{center}
\caption{
Newton diagrams in $(x,y,\tilde z)$-coordinates
} \label{fig:3-1-detailed:p2:v0:2}
\end{figure}

\paragraph{Case \ref{3-1-detailed:p2:v0:2}} $z^2 + 4 = 0$, i.e.\ $z = \pm 2\sqrt{-1}$. We start with the case that $z = 2\sqrt{-1}$,
i.e.\ $P = (0,0,2\sqrt{-1})$. Consider coordinates $(x, y, \tilde z)$, where $\tilde z:= z - 2\sqrt{-1})$, on $U := \kk^3$. 
The Newton diagram of $f_1$ and $f_3$ with respect to these coordinates are the same as the ones with respect to $(x,y,z)$-coordinates (the first and the third panels of \Cref{fig:3-1-detailed:p2:v0:0}). \Cref{fig:3-1-detailed:p2:v0:2} shows the Newton diagrams of $f_2$ and of the sum of the Newton diagrams of $f_1$ and $f_2$. As in the previous case, it is easy to see that \Cref{prop:newton1*} applies with $\scrI := \{\{3\}\}$, and its non-degeneracy condition is also satisfied. Moreover, since the Newton diagram of $\nd(f_1) + \nd(f_2)$ is one-dimensional, it follows that the right hand side of the inequality \eqref{eq:newton1*} is zero, so that $\ord_{P}(f_3|_C) = 0$. Applying the same arguments for $P = (0,0,-\sqrt{-1})$, it follows that
\begin{align}
\ord_{(0,0,\pm 2\sqrt{-1})}(f_3|_C) &= 0 \label{eq:3-1-detailed:p2:v0:2}
\end{align}
In other words, $C$ does {\em not} intersect the $z$-axis at $(0, 0, \pm 2 \sqrt{-1})$. \\ 

\def\shiftone{7.5}
\def\colorzero{blue}
\def\colorone{red}
\def\colortwo{yellow}
\def\colorfour{green}
\def\opazero{0.5}
\def\viewx{75}
\def\viewy{30}
\def\tx{1}
\def\ty{-1}

\begin{figure}[h]
\begin{center}
\begin{tikzpicture}[scale=0.5]
\pgfplotsset{every axis title/.append style={at={(0,-0.2)}}, view={\viewx}{\viewy}, axis lines=middle, enlargelimits={upper}}

\begin{scope}
\begin{axis}[
        xmin=0,
        xmax=5,
        ymin=0,
        ymax=3,
        zmin=0,
        zmax=3,
]
    \addplot3 [\colorzero, ultra thick] coordinates{(4,0,0) (1,1,0) (0,3,0)};
	\draw (axis cs:1,1,0) node [below right] {\picfontsize (1,1,0)};
\end{axis}
\node (A) at (\tx, \ty){};
\node (B) [right of=A] {\picfontsize $\nd(f_1)$};
\end{scope}

\begin{scope}[shift={(\shiftone,0)}]
\begin{axis}[
        xmin=0,
        xmax=5,
        ymin=0,
        ymax=3,
        zmin=0,
        zmax=3,
]
    \addplot3 [\colorzero, ultra thick] coordinates{(1,0,0) (0,2,0) (0,1,1) (1,0,0)};
	\addplot3[fill=\colorone,opacity=\opazero] coordinates{(1,0,0) (0,2,0) (0,1,1)};
	\addplot3 [ultra thick, \colorfour, ->] coordinates{(0.33,1,0.33) (1.33, 1.5, 0.83)};
	\draw (axis cs:1.33, 1.5, 0.83) node [above] {\picfontsize (2,1,1)};
\end{axis}
\node (A) at (\tx, \ty){};
\node (B) [right of=A] {\picfontsize $\nd(f_2)$};
\end{scope}

\begin{scope}[shift={(2*\shiftone,0)}]
\begin{axis}[
        xmin=0,
        xmax=5,
        ymin=0,
        ymax=3,
        zmin=0,
        zmax=3,
]
    \addplot3 [\colorzero, ultra thick] coordinates{(2,0,0) (0,1,0)};
\end{axis}
\node (A) at (\tx, \ty){};
\node (B) [right of=A] {\picfontsize $\nd(f_3)$};
\end{scope}

\begin{scope}[shift={(3*\shiftone,0)}]
\begin{axis}[
        xmin=0,
        xmax=5,
        ymin=0,
        ymax=6,
        zmin=0,
        zmax=3,
]
    \addplot3 [\colorzero, ultra thick] coordinates{(2,1,0) (0,5,0) (0,4,1) (1,2,1) (2,1,0)};
	\addplot3[fill=\colorone,opacity=\opazero] coordinates{(2,1,0) (0,5,0) (0,4,1) (1,2,1)};
    \addplot3 [\colorzero, ultra thick] coordinates{(2,1,0) (5,0,0)};
    \addplot3 [ultra thick, \colorfour, ->] coordinates{(0.75,3,0.5) (2.75, 4, 1.5)};
	\draw (axis cs:2.75, 4, 1.5) node [above] {\picfontsize (2,1,1)};
\end{axis}
\node (A) at (\tx, \ty){};
\node (B) [right of=A] {\picfontsize $\nd(\nd(f_1) + \nd(f_2))$};
\end{scope}
\end{tikzpicture}
\end{center}
\caption{Newton diagrams in $(x',y,z')$-coordinates} \label{fig:3-1-detailed:p2:degen}
\end{figure}
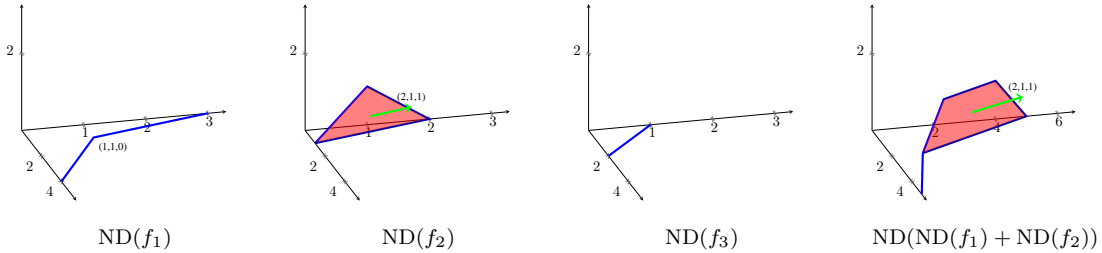

\paragraph{Case \ref{3-1-detailed:p2:degen}} $z = \pm 2$. First we treat the case $z = 2$, i.e.\ $P = (0,0,2)$. With respect to the coordinates $(x, y, z')$, where $z' := z - 2$, we have:
\begin{align*}
f_1 &
    = x^4 + xy + y^2(1 + 13y + 7yz') \\
f_2 &
    = 8(x+y) + 2xz'(4 + z'^2) - 8y^2 + yz'(1-y)(4 + z') \\
f_3 &
    = x^2(5+16x+8xz') + y(25-y+3(12z'+6z'^2 + z'^3))
\end{align*}
Note that $f_1, f_2$ violate condition \ref{assumption:newton1*:non-degen0} from \Cref{assumptions:newton1*} since with $\nu = (1,1,1)$, the initial forms of $f_1$ and $f_2$ are respectively $xy + y^2$ and $8(x+y)$, which have common solutions with nonzero coordinates. However, if we change coordinates to $(x',y,z')$ with $x' := x+y$, then
\begin{align*}
f_1 &
    = x'^4 + x'y + 13y^3 + \cdots \\
f_2 &
    = 8x'-4yz'-8y^2 + \cdots \\
f_3 &
    = 5x'^2 + 25y + \cdots
\end{align*}
Then it is easy to check that with respect to these coordinates $f_1,f_2,C$ satisfy the conditions of \Cref{assumptions:newton1*} with $\scrI := \{\{3\}\}$, and in addition, the non-degeneracy condition of \Cref{prop:newton1*} is also satisfied. It follows that
\begin{align*}
\ord_{P}(f_3|_C)
    &= \multzerostaronly{\nd(f_3), \nd(f_1), \nd(f_2)} \\
    &= \min_{\nd(f_3)}(\nu) \times
		\mv'_{\nu}(\In_{\nu}(\nd(f_1)), \In_{\nu}(\nd(f_2)))\quad (\text{where}\ \nu := (2,1,1)) \\
    &= 1 \times 1 = 1
\end{align*}
Similarly $\ord_{(0,0,-2)}(f_3|_C) = 1$ as well. It follows that
\begin{align}
\ord_{(0,0,2)}(f_3|_C) + \ord_{(0,0,-2)}(f_3|_C) &= 2 \label{eq:3-1-detailed:p2:degen}
\end{align}

\def\shiftone{7.5}
\def\colorzero{blue}
\def\colorone{red}
\def\colortwo{yellow}
\def\colorfour{green}
\def\colorthree{gray}
\def\opazero{0.5}
\def\viewx{75}
\def\viewy{30}
\def\tx{1}
\def\ty{-1}

\begin{figure}[h]
\begin{center}
\begin{tikzpicture}[scale=0.5]
\pgfplotsset{every axis title/.append style={at={(0,-0.2)}}, view={\viewx}{\viewy}, axis lines=middle, enlargelimits={upper}}

\begin{scope}
\begin{axis}[
        xmin=0,
        xmax=5,
        ymin=0,
        ymax=3,
        zmin=0,
        zmax=3,
]
	\addplot3[fill=\colorone,opacity=\opazero] coordinates{(4,0,0) (0,3,0) (1,1,2)};
    \addplot3 [\colorzero, ultra thick] coordinates{(4,0,0) (0,3,0) (1,1,2) (4,0,0)};

    \addplot3[fill=\colortwo,opacity=\opazero] coordinates{(0,3,0) (1,1,2) (0,2,2)};
    \addplot3[\colorzero, ultra thick] coordinates{(0,3,0) (1,1,2) (0,2,2) (0,3,0)};

    \addplot3 [ultra thick, \colorfour, ->] coordinates{(1.66,1.33,0.66) (2.86, 2.93, 1.66)};
	\draw (axis cs:2.86, 2.93, 1.66) node [below right] {\picfontsize (6,8,5)};

    \addplot3 [ultra thick, \colorfour, ->] coordinates{(0.33,2,1.33) (1.33, 3, 1.83)};
	\draw (axis cs:1.33, 3, 1.83) node [above right] {\picfontsize (2,2,1)};

\end{axis}
\node (A) at (\tx, \ty){};
\node (B) [right of=A] {\picfontsize $\nd(\tilde f_1)$};
\end{scope}

\begin{scope}[shift={(\shiftone,0)}]
\begin{axis}[
        xmin=0,
        xmax=5,
        ymin=0,
        ymax=3,
        zmin=0,
        zmax=3,
]
    \addplot3 [\colorzero, ultra thick] coordinates{(1,0,1) (0,2,0) (0,1,1) (1,0,1)};
	\addplot3[fill=\colorone,opacity=\opazero] coordinates{(1,0,1) (0,2,0) (0,1,1)};
	\addplot3 [ultra thick, \colorfour, ->] coordinates{(0.33,1,0.66) (0.83, 1.5, 1.16)};
	\draw (axis cs:0.83, 1.5, 1.16) node [above] {\picfontsize (1,1,1)};
\end{axis}
\node (A) at (\tx, \ty){};
\node (B) [right of=A] {\picfontsize $\nd(\tilde f_2)$};
\end{scope}

\begin{scope}[shift={(2*\shiftone,0)}]
\begin{axis}[
        xmin=0,
        xmax=5,
        ymin=0,
        ymax=3,
        zmin=0,
        zmax=3,
]
    \addplot3 [\colorzero, ultra thick] coordinates{(3,0,0) (0,1,0) (2,0,2) (3,0,0)};
	\addplot3[fill=\colorone,opacity=\opazero] coordinates{(3,0,0) (0,1,0) (2,0,2)};
	\addplot3 [ultra thick, \colorfour, ->] coordinates{(1.66,0.33,0.66) (2.33,2.33,1)};
	\draw (axis cs:2.33, 2.33, 1) node [above] {\picfontsize (2,6,1)};
\end{axis}
\node (A) at (\tx, \ty){};
\node (B) [right of=A] {\picfontsize $\nd(\tilde f_3)$};
\end{scope}

\begin{scope}[shift={(3*\shiftone,0)}]
\begin{axis}[
        xmin=0,
        xmax=5,
        ymin=0,
        ymax=6,
        zmin=0,
        zmax=3,
]
    \addplot3[fill=\colorone,opacity=\opazero] coordinates{(5,0,1) (1,3,1) (2,1,3)};
    \addplot3 [\colorzero, ultra thick] coordinates{(5,0,1) (1,3,1) (2,1,3) (5,0,1)};

    \addplot3[fill=\colortwo,opacity=\opazero] coordinates{(1,3,1) (2,1,3) (0,3,3) (0,4,1)};
    \addplot3[\colorzero, ultra thick] coordinates{(1,3,1) (2,1,3) (0,3,3) (0,4,1) (1,3,1)};

    \addplot3 [\colorzero, ultra thick] coordinates{(1,3,1) (0,5,0) (0,4,1) (1,3,1)};
	\addplot3[fill=\colorone,opacity=\opazero] coordinates{(1,3,1) (0,5,0) (0,4,1)};

    \addplot3[fill=\colorthree,opacity=\opazero] coordinates{(5,0,1) (1,3,1) (0,5,0) (4,2,0)};
    \addplot3 [\colorzero, ultra thick] coordinates{(5,0,1) (1,3,1) (0,5,0) (4,2,0) (5,0,1)};

    \addplot3 [ultra thick, \colorfour, ->] coordinates{(2.33,1.33,1.66) (3.53, 2.93, 2.66)};
	\draw (axis cs:3.53, 2.93, 2.66) node [right] {\picfontsize (6,8,5)};

    \addplot3 [ultra thick, \colorfour, ->] coordinates{(0.75,2.75,2) (1.75, 3.75, 2.5)};
	\draw (axis cs:1.75, 3.75, 2.5) node [right] {\picfontsize (2,2,1)};

    \addplot3 [ultra thick, \colorfour, ->] coordinates{(0.33,4,0.66) (0.83, 4.5, 1.16)};
	\draw (axis cs:0.83, 4.5, 1.16) node [right] {\picfontsize (1,1,1)};
    \addplot3 [ultra thick, \colorfour, ->] coordinates{(3.5, 1.5, 0.5) (4.1, 2.3, 1.5)}; 
	\draw (axis cs:4.1, 2.3, 1.5) node [right] {\picfontsize (3,4,5)};                     
\end{axis}
\node (A) at (\tx, \ty){};
\node (B) [right of=A] {\picfontsize $\nd(\nd(\tilde f_1) + \nd(\tilde f_2))$};
\end{scope}
\end{tikzpicture}
\end{center}
\caption{Newton diagrams in $(\tilde x, \tilde y, \tilde w)$-coordinates} \label{fig:3-1-detailed:p2:infty}
\end{figure}
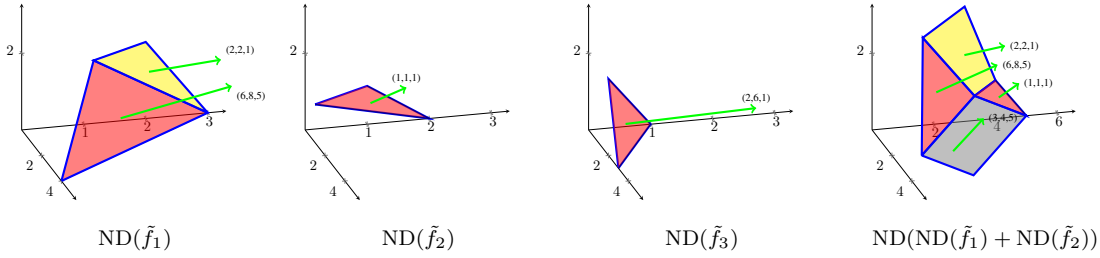

\paragraph{Case \ref{3-1-detailed:p2:infty}} It remains to consider the last case, i.e.\ $P = [0:0:1:0]$ with respect to homogeneous coordinates $[x:y:z:w]$. With respect to coordinates $(\tilde x, \tilde y, \tilde w) := (x/z, y/z, 1/z)$, $P$ is the origin, and the $f_i$ take the following form:
%

\begin{align*}
\tilde f_1 &
    := \tilde x^4 + \tilde x \tilde y \tilde w^2 + \tilde y^2(\tilde w^2 -\tilde y \tilde w + 7\tilde y) \\
\tilde f_2 &
    := 2\tilde x \tilde w + \tilde y(\tilde w - \tilde y)(4\tilde w^2 + 1) \\
\tilde f_3 &
    := \tilde x^2(5\tilde w^2 + 8 \tilde x) + \tilde y(\tilde w^3 -\tilde y\tilde w^2 +3)
\end{align*}
It is easy to check that the conditions of \Cref{assumptions:newton1*} and the non-degeneracy condition of \Cref{prop:newton1*} are satisfied by $\tilde f_1,\tilde f_2,C$ with $\scrI := \{\{3\}\}$. It follows that
\begin{align*}
\ord_{P}(\tilde f_3|_C)
    &= \multzerostaronly{\nd(\tilde f_3), \nd(\tilde f_1), \nd(\tilde f_2)} \\
    &= \sum_{i=1}^4 \min_{\nd(\tilde f_3)}(\nu_i) \times
		\mv'_{\nu_i}(\In_{\nu_i}(\nd(\tilde f_1)), \In_{\nu_i}(\nd(\tilde f_2)))
\end{align*}
where $\nu_i$ vary over the (inner) normals to the four facets of the Newton diagram of the sum of $\nd(\tilde f_1)$ and $\tilde f_2$ (see the last panel of \Cref{fig:3-1-detailed:p2:infty}). Since the initial term of $\tilde f_2$ (respectively, $\tilde f_1$) with respect to weights $(6,8,5)$ (respectively, $(1,1,1)$) is a singleton, it follows that $\mv'_{\nu_i}(\In_{\nu_i}(\nd(\tilde f_1)), \allowbreak \In_{\nu_i}(\nd(\tilde f_2))) = 0$ for $\nu_i = (6,8,5)$ or $(1,1,1)$. Consequently,

\begin{align*}
\ord_{P}(\tilde f_3|_C)
    &= \min_{\nd(\tilde f_3)}((2,2,1)) \times
		\mv'_{(2,2,1)}(\In_{(2,2,1)}(\nd(\tilde f_1)), \In_{(2,2,1)}(\nd(\tilde f_2)))
        \\
    &\qquad
       + \min_{\nd(\tilde f_3)}((3,4,5)) \times
		\mv'_{(3,4,5)}(\In_{(3,4,5)}(\nd(\tilde f_1)), \In_{(3,4,5)}(\nd(\tilde f_2))) \\
    &= 2 \times 1 + 4 \times 1 = 6
\end{align*}
It then follows from identities \eqref{eq:3-1-detailed:0}, \eqref{eq:3-1-detailed:0y}, \eqref{eq:3-1-1}, \eqref{eq:3-1-detailed:p2:v0:0}, \eqref{eq:3-1-detailed:p2:v0:2}, \eqref{eq:3-1-detailed:p2:degen} that
\begin{align}
 \multisoonly{f_1, f_2, f_3}{\kk^3}
    &= 64
        - 8
        - 4 - 0 - 2 - 6
        - 14
    = 30
        \label{eq:3-1-detailed:answer}
\end{align}
\end{proexample}

\section{An example: the number of tangent lines to four generic spheres} \label{sec:tangent}
\subsection{}
In this section we apply the results of \Cref{sec:local,sec:newton} to compute the number of tangent lines to four generic spheres in $\kk^3$ in the case that $\character(\kk) = 0$. This number was determined to be 12 by Macdonald, Pach and Theobald \cite{macdonald-pach-theobald} by expressing the governing equations in a cleverly chosen system of coordinates so that a direct application of B\'ezout's theorem solves the problem. On the other hand, a natural setting of this counting problem is the Grassmannian of lines in $\pp^3$ (or equivalently, planes in $\kk^4$). Formulating the problem in corresponding Pl\"ucker coordinates gives $5$ quadric equations in $\pp^5$ whose common zeroes include a one-dimensional component at infinity. We explain the difference of $2^5 - 12 = 20$ between the B\'ezout bound and the correct number as the sum of ordered intersection multiplicities corresponding to the excess intersection. This section was inspired by a paragraph by Sottile and Theobald (in \cite{sottile-theobald} which solved the tangent line problem in higher dimensions) which mentions a similar explanation in terms of usual excess intersection multiplicities in a manuscript by Aluffi and Fulton.

\subsection{}
Consider the open subset $U$ of the space of lines in $\kk^3$ parameterized by
\begin{align*}
(u_1, u_2, v_1, v_2) \in \kk^4 \mapsto L_{u_1, u_2, v_1, v_2} := \{(u_1 + tv_1, u_2 + tv_2, t): t \in \kk\}
\end{align*}
The points of intersection of $L_{u_1, u_2, v_1, v_2}$ and the sphere in $\kk^3$ with center $(a_1, a_2, a_3)$ and radius $r_1$ correspond to solutions (in $t$) of the following equation:
\begin{align*}
(u_1 + tv_1 - a_1)^2 + (u_2 + tv_2 - a_2)^2 + (t - a_3)^2 = r_1^2
\end{align*}
Consequently, $L_{u_1, u_2, v_1, v_2}$ is tangent to the sphere if and only if the discriminant of the above equation is zero, i.e.
\begin{align*}
(v_1(u_1 - a_1) + v_2(u_2 - a_2) - a_3)^2 - (v_1^2 + v_2^2 + 1)((u_1 - a_1)^2 + (u_2 - a_2)^2 + a_3^2-r_1^2) = 0
\end{align*}
which simplifies to
\begin{align*}
((u_2v_1 - u_1v_2) - (a_2v_1 - a_1v_2))^2 + (u_1 + a_3v_1 - a_1)^2 + (u_2 + a_3v_2 - a_2)^2 - r_1^2(v_1^2 + v_2^2 + 1)
    &= 0
\end{align*}
It turns out that (up to changing some signs) the Pl\"ucker embedding of the Grassmannian of lines in $\pp^3$ is determined by $(u_1,u_2, v_1, v_2, w)$, where $w := u_2v_1 - u_1v_2$. Consequently, in the Pl\"ucker coordinates, the lines tangent to four spheres are precisely the points in $V(f_0, \ldots, f_4) \subseteq \kk^5$, where
\begin{align*}
f_0 &:= w - u_2v_1 + u_1v_2 \\
f_1 &:= (w - (a_2v_1 - a_1v_2))^2 + (u_1 + a_3v_1 - a_1)^2 + (u_2 + a_3v_2 - a_2)^2 - r_1^2(v_1^2 + v_2^2 + 1) \\
f_2 &:= (w - (b_2v_1 - b_1v_2))^2 + (u_1 + b_3v_1 - b_1)^2 + (u_2 + b_3v_2 - b_2)^2 - r_2^2(v_1^2 + v_2^2 + 1) \\
f_3 &:= (w - (c_2v_1 - c_1v_2))^2 + (u_1 + c_3v_1 - c_1)^2 + (u_2 + c_3v_2 - c_2)^2 - r_3^2(v_1^2 + v_2^2 + 1) \\
f_4 &:= (w - (d_2v_1 - d_1v_2))^2 + (u_1 + d_3v_1 - d_1)^2 + (u_2 + d_3v_2 - d_2)^2 - r_4^2(v_1^2 + v_2^2 + 1)
\end{align*}
where the $r_1, \ldots, r_4$ are the radii and $(a_1, a_2, a_3), \ldots, (d_1, d_2, d_3)$ are the centers of the spheres. Since $f_0$ and $f_1$ does not have any irreducible components in common, the codimension of $V(f_0, f_1)$ is $2$ in $\kk^5$. Then for generic choices e.g.\ of $b_1, c_1, d_1$, it follows that $V(f_0, \ldots, f_4)$ only has isolated points in $\kk^5$. We will show that the number (counted with appropriate multiplicities) of these points is 12. Note that the B\'ezout bound for the number of solutions of the $f_i$ in $\pp^5$ is $2^5 = 32$.

\subsection{} Choosing appropriate coordinates on $\kk^3$, without loss of generality we may assume $a_1 = a_2 = a_3 = b_1 = b_2 = 0$ and $r_1 = 1$. It follows that
\begin{align*}
f_1 &= w^2 + u_1^2 + u_2^2 - (v_1^2 + v_2^2 + 1) \\
f_2 &= w^2 + (u_1 + b_3v_1)^2 + (u_2 + b_3v_2)^2 - r_2^2(v_1^2 + v_2^2 + 1)
\end{align*}

\subsection{}
Let $F_i$ be the divisor on $\pp^5$ corresponding to $f_i$. To determine the intersection ``at infinity'' on $\pp^5$ of the $F_i$, we consider their leading forms:
\begin{align*}
\ld(f_0) &= u_2v_1 - u_1v_2 \\
\ld(f_1) &= w^2 + u_1^2 + u_2^2 - v_1^2 - v_2^2 \\
\ld(f_2) &= w^2 + (u_1 + b_3v_1)^2 + (u_2 + b_3v_2)^2 - r_2^2(v_1^2 + v_2^2) \\
\ld(f_3) &= (w - (c_2v_1 - c_1v_2))^2 + (u_1 + c_3v_1)^2 + (u_2 + c_3v_2)^2 - r_3^2(v_1^2 + v_2^2) \\
\ld(f_4) &= (w - (d_2v_1 - d_1v_2))^2 + (u_1 + d_3v_1)^2 + (u_2 + d_3v_2)^2 - r_4^2(v_1^2 + v_2^2)
\end{align*}
$V(\ld(f_0), \ld(f_1))$ has codimension two in the hyperplane at infinity on $\pp^5$. For generic choices of the $b_3, c_i, d_j$ it follows that the common zero set of the $\ld(f_i)$ is empty if either $v_1$ or $v_2$ is nonzero, and consequently,
\begin{align*}
\bigcap_i \supp(F_i) \cap (\pp^5 \setminus \kk^5) = Z := V(z, v_1, v_2, w^2 + u_1^2 + u_2^2)
\end{align*}
where $z$ is the ``homogenizing variable'' which forms a system of homogeneous coordinates $[u_1:u_2:v_1:v_2:w:z]$ on $\pp^5$; note that $Z$ is a one dimensional ``imaginary circle at infinity''. It then follows from identity \eqref{eq:=multiso} on page \pageref{eq:=multiso} and B\'ezout's theorem that
\begin{align}
 \multiso{f_0}{f_4}{\kk^5}
    &= \mult{F_0}{F_4}
        - \multord{F_0}{F_4}{Z}
        - \sum_{P \in Z} \multord{F_0}{F_4}{P} \notag \\
    &= 32
        - \multord{F_0}{F_4}{Z}
        - \sum_{P \in Z} \multord{F_0}{F_4}{P}
        \label{eq:tangent:0}
\end{align}

\subsection{} \label{sec:tangent:Z}
Since $Z$ is an irreducible component of $\bigcap_{j=0}^3 \supp(F_i)$, assertion \eqref{local=mult:1:0} of \Cref{prop:local=mult:1} implies that
\begin{align}
\multord{F_0}{F_4}{Z}
    &= \multsub{F_0}{F_3}{Z} \deg_4(Z)
    = \multsub{F_0}{F_3}{Z} \deg(Z) \deg(F_4) \notag \\
    &= 4\multsub{F_0}{F_3}{Z}
    \label{eq:tangent:Z:0}
\end{align}
To compute $\multsub{F_0}{F_3}{Z}$, we work on the coordinate chart $U_w := \pp^5 \setminus V(w)$. With the coordinates $(\tilde u_1, \tilde u_2, \tilde v_1, \tilde v_2, \tilde z) := (u_1/w, u_2/w, v_1/w, v_2/w, 1/w)$ on $U_w$, the polynomials defining $F_i \cap U_w$ are:
\begin{align*}
\tilde f_0 &:= \tilde z - \tilde u_2 \tilde v_1 + \tilde u_1 \tilde v_2 \\
\tilde f_1 &:= 1 + \tilde u_1^2 + \tilde u_2^2 - (\tilde v_1^2 + \tilde v_2^2 + \tilde z^2) \\
\tilde f_2 &:= 1 + (\tilde u_1 + b_3\tilde v_1)^2 + (\tilde u_2 + b_3\tilde v_2)^2 - r_2^2(\tilde v_1^2 + \tilde v_2^2 + \tilde z^2) \\
\tilde f_3 &:= (1 - (c_2\tilde v_1 - c_1\tilde v_2))^2 + (\tilde u_1 + c_3\tilde v_1 - c_1\tilde z)^2
                    + (\tilde u_2 + c_3\tilde v_2 - c_2\tilde z)^2 - r_3^2(\tilde v_1^2 + \tilde v_2^2 + \tilde z^2) \\
\tilde f_4 &:= (1 - (d_2\tilde v_1 - d_1\tilde v_2))^2 + (\tilde u_1 + d_3\tilde v_1 - d_1\tilde z)^2
                    + (\tilde u_2 + d_3\tilde v_2 - d_2\tilde z)^2 - r_4^2(\tilde v_1^2 + \tilde v_2^2 + \tilde z^2)
\end{align*}
Note that $\tilde S := V(\tilde f_0, \tilde f_1)$ is a nonsingular threefold near each point of $Z \cap U_w$, so that $\multsub{F_0}{F_3}{Z}$ is the intersection multiplicity along $Z$ of the (restrictions of the) divisors of $\tilde f_2$ and $\tilde f_3$ on $\tilde S$; in other words,
\begin{align}
\multsub{F_0}{F_3}{Z}
    &= \multonlysub{\tilde f_2|_{\tilde S}, \tilde f_3|_{\tilde S}}{Z}
    \label{eq:tangent:Z:mult}
\end{align}

\subsubsection{} \label{sec:tangent:Z:f2}
To compute the divisor of $\tilde f_2$ on $\tilde S$ note that
\begin{align*}
\tilde f_2
    &= 1 + \tilde u_1^2 + \tilde u_2^2 + 2b_3(\tilde u_1 \tilde v_1 + \tilde u_2 \tilde v_2) - r_2^2(\tilde v_1^2 + \tilde v_2^2 + \tilde z^2)\\
    &\equiv 2b_3(\tilde u_1 \tilde v_1 + \tilde u_2 \tilde v_2) + (1- r_2^2)(\tilde v_1^2 + \tilde v_2^2 + \tilde z^2) \mod \tilde f_1
\end{align*}
Now
\begin{align*}
\tilde v_1^2 + \tilde v_2^2 + \tilde z^2
    &\equiv \tilde v_1^2 + \tilde v_2^2 + (\tilde u_2\tilde v_1 - \tilde u_1 \tilde v_2)^2  \mod \tilde f_0 \\
    &= \tilde v_1^2(1+\tilde u_2^2) + \tilde v_2^2(1 + \tilde u_1^2) -2\tilde u_1\tilde u_2 \tilde v_1 \tilde v_2 \\
    &= \tilde v_1^2(1+\tilde u_1^2 +\tilde u_2^2) + \tilde v_2^2(1 + \tilde u_1^2 + \tilde u_2^2)
            - (\tilde u_1^2\tilde v_1^2 + \tilde u_2^2 \tilde v_2^2 + 2\tilde u_1\tilde u_2 \tilde v_1 \tilde v_2) \\
    &= (1+\tilde u_1^2 +\tilde u_2^2)(\tilde v_1^2 + \tilde v_2^2) - (\tilde u_1\tilde v_1 + \tilde u_2 \tilde v_2)^2 \\
    &\equiv (v_1^2 + \tilde v_2^2 + \tilde z^2)(\tilde v_1^2 + \tilde v_2^2) - (\tilde u_1\tilde v_1 + \tilde u_2 \tilde v_2)^2
        \mod \tilde f_1
\end{align*}
so that, modulo $\tilde f_0$ and $\tilde f_1$,
\begin{align*}
\tilde v_1^2 + \tilde v_2^2 + \tilde z^2
    \equiv - \frac{(\tilde u_1\tilde v_1 + \tilde u_2 \tilde v_2)^2}{1 - \tilde v_1^2 - \tilde v_2^2}
\end{align*}
near $Z \cap U_w$. Consequently,
\begin{align}
\tilde f_2
    &\equiv \tilde v_0 ( 2b_3 - \frac{1- r_2^2}{1 - \tilde v_1^2 - \tilde v_2^2} \tilde v_0)
    \label{eq:tangent:Z:f2}
\end{align}
modulo $\tilde f_0$ and $\tilde f_1$ near $Z \cap U_w$, where $\tilde v_0 := (\tilde u_1 \tilde v_1 + \tilde u_2 \tilde v_2)$. In particular, when $b_3 \neq 0$, the divisor of $\tilde f_2$ on $\tilde S$ is defined by $\tilde v_0$ near all points of $Z \cap U_w$.

\subsubsection{} \label{sec:tangent:Z:f3}
Now we compute the divisor of $\tilde f_3$ on $\tilde S$. Since $(\tilde v_0, \tilde v_1)$ have linearly independent differentials whenever $\tilde u_2 \neq 0$, we may mod out by the relation $\tilde v_2 = (\tilde v_0 - \tilde u_1 \tilde v_1)/\tilde u_2$ as well as by $f_0$ and $f_1$. This yields, after a long computation, that
\begin{align}
\tilde f_3
    &\equiv
    \frac{\tilde h_{3,0}}{(1-\tilde v_1^2 - \tilde v_2^2)\tilde u_2^2}\tilde v_0
    + \frac{1}{\tilde u_2^2}((c_1 \tilde u_1 + c_2 \tilde u_2)^2 + (c_3 \tilde u_2 + c_1)^2 + (c_3\tilde u_1 - c_2)^2)\tilde v_1^2
    \label{eq:tangent:Z:f3}
\end{align}
where $\tilde h_{3,0}$ is a polynomial in all $\tilde u_i, \tilde v_j, c_k$ and $r_3$.

\subsubsection{} \label{sec:tangent:Z:ord}
It follows that at each point $P$ of $Z \cap U_w \cap \{\tilde u_2 \neq 0\}$, the ideal generated by $\tilde f_2, \tilde f_3$ in $\local{\tilde S}{P}$ is the same as the ideal generated by $\tilde v_0$ and $\tilde h_{3,1} \tilde v_1^2$, where
\begin{align*}
\tilde h_{3,1} &:= (c_1 \tilde u_1 + c_2 \tilde u_2)^2 + (c_3 \tilde u_2 + c_1)^2 + (c_3\tilde u_1 - c_2)^2
\end{align*}
Since $\tilde h_{3,1} \neq 0$ near generic points of $Z$, it follows from identities \eqref{eq:tangent:Z:f2} and \eqref{eq:tangent:Z:f3} that
\begin{align*}
\multonlysub{\tilde f_2|_{\tilde S}, \tilde f_3|_{\tilde S}}{Z}
    &= \multonlysub{\tilde v_0|_{\tilde S}, \tilde v_1^2|_{\tilde S}}{Z}
    = 2
\end{align*}
as long as $b_3 \neq 0$ (since $\tilde v_0, \tilde v_1$ generate the ideal of $Z$ in $\local{\tilde S}{P}$ for generic $P \in \tilde S \cap Z$, the above equality can be regarded as a (rather trivial) application of \Cref{cor:mult0}). Identities \eqref{eq:tangent:Z:0} and \eqref{eq:tangent:Z:mult} then imply that
\begin{align}
\multord{F_0}{F_4}{Z}
    &= 4 \times 2
    = 8
    \label{eq:tangent:Z:0:computation}
\end{align}

\subsection{}
It remains to compute the sum $\sum_{P \in Z} \multord{F_0}{F_4}{P}$ from the right hand side of \eqref{eq:tangent:0}. Write the scheme theoretic intersection of $F_0, \ldots, F_3$ as $mZ + C$. Then
\begin{align}
\sum_{P \in Z} \multord{F_0}{F_4}{P}
    &= \sum_{P \in Z} \ord_P(F_4|_C)
    \label{eq:tangent:dim0}
\end{align}
(see e.g.\ \Cref{ex:n-1}).

\subsubsection{} \label{sec:tangent:dim0:Uw2}
First we consider the case that $P \in U_{w,2} := U_w \cap \{\tilde u_2 \neq 0\}$. The arguments from \Cref{sec:tangent:Z} imply that
\begin{itemize}
\item if $b_3 \neq 0$, then $C \cap U_{w,2}$ is the scheme theoretic intersection of $\tilde S$ and $V(\tilde v_0, \tilde h_{3,1})$, where $\tilde h_{3,1} := (c_1 \tilde u_1 + c_2 \tilde u_2)^2 + (c_3 \tilde u_2 + c_1)^2 + (c_3\tilde u_1 - c_2)^2$, and
\item the restriction of $\tilde f_4$ to $C \cap U_{w,2}$ equals (the restriction of) $\tilde h_{4,1}\tilde v_1^2/\tilde u_2^2$, where $\tilde h_{4,1} := (d_1 \tilde u_1 + d_2 \tilde u_2)^2 + (d_3 \tilde u_2 + d_1)^2 + (d_3\tilde u_1 - d_2)^2$.
\end{itemize}
Since $U_{w,2} \cap Z \cap C = U_{w,2} \cap V(\tilde z, \tilde v_1, \tilde v_2, \tilde u_1^2 + \tilde u_2^2 + 1, \tilde h_{3,1})$, for generic $c_1, c_2, c_3$, there are precisely $2 \times 2 = 4$ points in $U_{w,2} \cap Z \cap C$ due to B\'ezout's theorem. If $d_1, d_2, d_3$ are generic, then $\tilde h_{4,1}$ is nonzero at each of these points, so that
\begin{align}
\sum_{P \in Z \cap U_{w,2}} \ord_P(F_4|_C)
    &= \sum_{P \in Z \cap C \cap U_{w,2}} \ord_P((\tilde h_{4,1}\tilde v_1^2)|_C) \notag \\
    &= \sum_{P \in Z \cap C \cap U_{w,2}} \ord_P(\tilde v_1^2|_C)
    = 4 \times 2
    = 8
    \label{eq:tangent:dim0:Uw2}
\end{align}

\subsubsection{} It remains to consider the points in $Z \setminus U_{w,2} = V(z,v_1, v_2, u_2, u_1^2 + w^2) \cup V(z,v_1, v_2, w, u_1^2 + u_2^2)$. First we consider the case that $P \in  V(z,v_1, v_2, u_2, u_1^2 + w^2)$. Running the same computations from \Cref{sec:tangent:Z:f3} with the substitution $\tilde v_1 = (\tilde v_0 - \tilde u_2\tilde v_2)/\tilde u_1$ (instead of the substitution of $\tilde v_2$) shows that near $Z$ on $\tilde S \cap \{\tilde u_1 \neq 0\}$,
\begin{align*}
\tilde f_3 &\equiv \frac{\tilde h_{3,1}}{\tilde u_1^2}\tilde v_1^2
\end{align*}
modulo (the ideal generated by) $\tilde v_0$. If $c_1, c_2, c_3$ are generic, then $\tilde h_{3,1}$ is nonzero when $\tilde u_2 = 0$. It follows that for each $P \in  V(z,v_1, v_2, u_2, u_1^2 + w^2)$, $Z$ is the {\em only} irreducible component of $\bigcap_{i=0}^3 F_i$ near $P$; in particular,
\begin{align}
\sum_{P \in V(z,v_1, v_2, u_2, u_1^2 + w^2)} \ord_P(F_4|_C)
    &= 0
    \label{eq:tangent:dim0:u2}
\end{align}

\subsubsection{} Now we treat the case that $P \in  V(z,v_1, v_2, w, u_1^2 + u_2^2)$. Consider the chart $U_2 := \pp^5 \setminus V(u_2)$ with coordinates $(u'_1, v'_1, v'_2, w', z') := (u_1/u_2, v_1/u_2, v_2/u_2, w/u_2, 1/u_2)$. The polynomials defining $F_i|_{U_2}$ are:
\begin{align*}
f'_0 &:= z'w' - v'_1 + u'_1v'_2 \\
f'_1 &:= w'^2 + u'^2_1 + 1 - (v'^2_1 + v'^2_2 + z'^2) \\
f'_2 &:= w'^2 + (u'_1 + b_3v'_1)^2 + (1 + b_3v'_2)^2 - r_2^2(v'^2_1 + v'^2_2 + z'^2) \\
f'_3 &:= (w' - (c_2v'_1 - c_1v'_2))^2 + (u'_1 + c_3v'_1 - c_1z')^2 + (1 + c_3v'_2 - c_2z')^2 - r_3^2(v'^2_1 + v'^2_2 + z'^2) \\
f'_4 &:= (w' - (d_2v'_1 - d_1v'_2))^2 + (u'_1 + d_3v'_1 - d_1z')^2 + (1 + d_3v'_2 - d_2z')^2 - r_4^2(v'^2_1 + v'^2_2 + z'^2)
\end{align*}
Let $S' := V(f'_0, f'_1) \subseteq U_2$. Computations similar to those in \Cref{sec:tangent:Z:f2} show that the following hold in $\hatlocal{S'}{P} \cong \kk[[v'_2, w', z']]$:
\begin{align*}
f'_2 &= (1-r_2^2) z'^2 + 2b_3u'_1w'z' - 2b_3v'_2w'^2 + \cdots \\
f'_3 &= -2(c_2 + c_1u'_1)z' - 2(c_2u'_1 - c_1) v'_2w' + (c_2u'_1 - c_1)^2 v'^2_2 + \cdots
\end{align*}

\def\shiftone{7.5}
\def\colorzero{blue}
\def\colorone{red}
\def\colortwo{yellow}
\def\colorfour{green}
\def\colorthree{gray}
\def\opazero{0.5}
\def\viewx{75}
\def\viewy{30}
\def\tx{1}
\def\ty{-1}

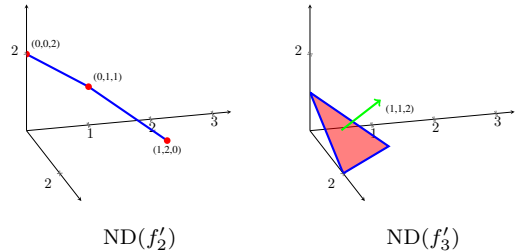
\begin{figure}[h]
\begin{center}
\begin{tikzpicture}[scale=0.5]
\pgfplotsset{every axis title/.append style={at={(0,-0.2)}}, view={\viewx}{\viewy}, axis lines=middle, enlargelimits={upper}}

\tikzstyle{dot} = [red, circle, minimum size=5pt, inner sep = 0pt, fill]

\begin{scope}
\begin{axis}[
        xmin=0,
        xmax=3,
        ymin=0,
        ymax=3,
        zmin=0,
        zmax=3,
]
    \addplot3 [\colorzero, ultra thick] coordinates{(0,0,2) (0,1,1) (1,2,0)};
    \node[dot] at (axis cs:0,0,2) {};
    \node[dot] at (axis cs:0,1,1) {};
    \node[dot] at (axis cs:1,2,0) {};
    \draw (axis cs:0,0,2) node [above right] {\picfontsize (0,0,2)};
    \draw (axis cs:0,1,1) node [above right] {\picfontsize (0,1,1)};
    \draw (axis cs:1,2,0) node [below] {\picfontsize (1,2,0)};
\end{axis}
\node (A) at (\tx, \ty){};
\node (B) [right of=A] {\picfontsize $\nd(f'_2)$};
\end{scope}

\begin{scope}[shift={(\shiftone,0)}]
\begin{axis}[
        xmin=0,
        xmax=3,
        ymin=0,
        ymax=3,
        zmin=0,
        zmax=3,
]
	\addplot3[fill=\colorone,opacity=\opazero] coordinates{(0,0,1) (1,1,0) (2,0,0)};
    \addplot3 [\colorzero, ultra thick] coordinates{(0,0,1) (1,1,0) (2,0,0) (0,0,1)};

    \addplot3 [ultra thick, \colorfour, ->] coordinates{(0.66,0.33,0.33) (1.16,0.83,1.33)};
	\draw (axis cs:1.16,0.83,1.33) node [below right] {\picfontsize (1,1,2)};

\end{axis}
\node (A) at (\tx, \ty){};
\node (B) [right of=A] {\picfontsize $\nd(f'_3)$};
\end{scope}

\end{tikzpicture}
\end{center}
\caption{Newton diagrams of $f'_2$ and $f'_3$ in $(v'_2, w', z')$-coordinates} \label{fig:tangent:dim0:w':z'}
\end{figure}

Since near $P$ on a set-theoretic level $V(f'_2, f'_3) \cap S'$ is the union of $C$ and $Z \cap S' \cong V(v'_2, z')$, we may try to use \Cref{prop:newton1*} to estimate $\ord_P(f'_4|_C)$. However, it is straightforward to check that for weighted orders with weights $\nu := (1,\omega,1+\omega)$ with respect to $(v'_2, w', z')$-coordinates, where $0 < \omega < 1$, the initial forms of $f'_2$ and $f'_3$ at $P$ are:
\begin{align*}
\In_\nu(f'_2) &= 2b_3w'(u'_1z' - v'_2w') = 2b_3u'_1w'(z' +u'_1v'_2w')\\
\In_\nu(f'_3) &= -2((c_2 + c_1u'_1)z' + (c_2u'_1 - c_1) v'_2w') = -2(c_2 +c_1u'_1)(z' + u'_1v'_2w')
\end{align*}
(since $u'^2_1 = -1$ at $P$) which are {\em proportional}, and consequently, have common solutions in $(\kk\setminus\{0\})^3$, e.g.\ $(v'_2, w', z') = (1,1,-u'_1)$. Consequently, $(f'_2, f'_3)$ violates condition \ref{assumption:newton1*:non-degen0} of \Cref{assumptions:newton1*}, and \Cref{prop:newton1*} does {\em not} apply. However, if we let $z'' := z' + u'_1v'_2w'$ and expand $f'_2, f'_3$ in $(v'_2, w', z'')$-coordinates, then

\begin{align*}
f'_2 &= (1-r_2^2) z''^2 + 2b_3u'_1w'z'' + \cdots \\
f'_3 &= -2(c_2 + c_1u'_1)z'' + (c_2u'_1 - c_1)^2 v'^2_2 + \cdots
\end{align*}

\def\shiftone{7.5}
\def\colorzero{blue}
\def\colorone{red}
\def\colortwo{yellow}
\def\colorfour{green}
\def\colorthree{gray}
\def\opazero{0.5}
\def\viewx{75}
\def\viewy{30}
\def\tx{1}
\def\ty{-1}

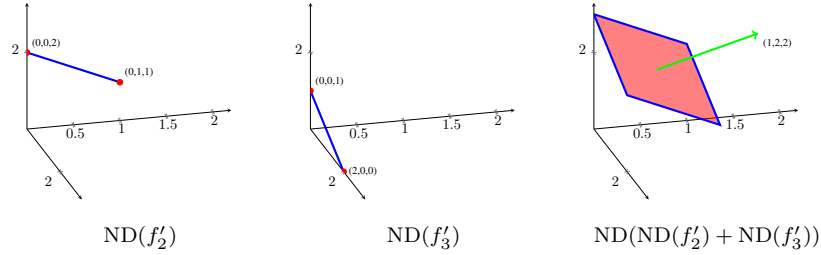
\begin{figure}[h]
\begin{center}
\begin{tikzpicture}[scale=0.5]
\pgfplotsset{every axis title/.append style={at={(0,-0.2)}}, view={\viewx}{\viewy}, axis lines=middle, enlargelimits={upper}}

\tikzstyle{dot} = [red, circle, minimum size=5pt, inner sep = 0pt, fill]

\begin{scope}
\begin{axis}[
        xmin=0,
        xmax=3,
        ymin=0,
        ymax=2,
        zmin=0,
        zmax=3,
]
    \addplot3 [\colorzero, ultra thick] coordinates{(0,0,2) (0,1,1)};
    \node[dot] at (axis cs:0,0,2) {};
    \node[dot] at (axis cs:0,1,1) {};
    \draw (axis cs:0,0,2) node [above right] {\picfontsize (0,0,2)};
    \draw (axis cs:0,1,1) node [above right] {\picfontsize (0,1,1)};
\end{axis}
\node (A) at (\tx, \ty){};
\node (B) [right of=A] {\picfontsize $\nd(f'_2)$};
\end{scope}

\begin{scope}[shift={(\shiftone,0)}]
\begin{axis}[
        xmin=0,
        xmax=3,
        ymin=0,
        ymax=2,
        zmin=0,
        zmax=3,
]
    \addplot3 [\colorzero, ultra thick] coordinates{(0,0,1) (2,0,0)};
    \node[dot] at (axis cs:0,0,1) {};
    \node[dot] at (axis cs:2,0,0) {};
    \draw (axis cs:0,0,1) node [above right] {\picfontsize (0,0,1)};
    \draw (axis cs:2,0,0) node [right] {\picfontsize (2,0,0)};
\end{axis}
\node (A) at (\tx, \ty){};
\node (B) [right of=A] {\picfontsize $\nd(f'_3)$};
\end{scope}

\begin{scope}[shift={(2*\shiftone,0)}]
\begin{axis}[
        xmin=0,
        xmax=3,
        ymin=0,
        ymax=2,
        zmin=0,
        zmax=3,
]
	\addplot3[fill=\colorone,opacity=\opazero] coordinates{(0,0,3) (2,0,2) (2,1,1) (0,1,2)};
    \addplot3 [\colorzero, ultra thick] coordinates{(0,0,3) (2,0,2) (2,1,1) (0,1,2) (0,0,3)};

    \addplot3 [ultra thick, \colorfour, ->] coordinates{(1,0.5,2) (1.5,1.5,3)};
	\draw (axis cs:1.5,1.5,3) node [below right] {\picfontsize (1,2,2)};

\end{axis}
\node (A) at (\tx, \ty){};
\node (B) [right of=A] {\picfontsize $\nd(\nd(f'_2) + \nd(f'_3))$};
\end{scope}

\end{tikzpicture}
\end{center}
\caption{Newton diagrams in $(v'_2, w', z'')$-coordinates} \label{fig:tangent:dim0:w':z''}
\end{figure}

It is straightforward to check that if $b_3, c_1, c_2$ are generic, then $(f'_2, f'_3, f'_4)$ satisfy all assumptions and non-degeneracy conditions of \Cref{prop:newton1*}, so that
\begin{align*}
\ord_P(f'_4|_C)
    &= \multzerostaronly{\nd(f'_2), \nd(f'_3), \nd(f'_4)} \\
    &= \min_{\nd(f'_4)}(\nu) \times
		\mv'_{\nu}(\In_{\nu}(\nd(f'_2)), \In_{\nu}(\nd(f'_4)))\quad (\text{where}\ \nu := (1, 2, 2)) \\
    &= 2 \times 1 = 2
\end{align*}
Since there are two points in $V(z,v_1, v_2, w, u_1^2 + u_2^2)$, it follows that
\begin{align}
\sum_{P \in V(z,v_1, v_2, w, u_1^2 + u_2^2)} \ord_P(F_4|_C)
    &= 2 \times 2
    = 4
    \label{eq:tangent:dim0:w=0}
\end{align}

\subsection{}
Combining equations \eqref{eq:tangent:0}, \eqref{eq:tangent:Z:0:computation}, \eqref{eq:tangent:dim0}, \eqref{eq:tangent:dim0:Uw2}, \eqref{eq:tangent:dim0:u2}, \eqref{eq:tangent:dim0:w=0} we obtain
\begin{align*}
 \multiso{f_0}{f_4}{\kk^5}
    &= 32 - 8 - 8 - 0 - 4
    = 12
\end{align*}
when $b_3, c_i, d_j$ are generic. This completes the computation of the number of tangent lines to four generic spheres in $\kk^3$ (when $\character(\kk) = 0$). Note that the above arguments did {\em not} require any $r_i$ to be generic, i.e.\ we showed that the bound of 12 tangent lines can be achieved even if the radii of the spheres are fixed (and nonzero).

\section{Inductive ``solutions'' to the affine B\'ezout problem} \label{sec:ind}
In \Crefrange{sec:local}{sec:tangent} we described a few scenarios in which 
ordered intersection multiplicity 
can be interpreted in terms of intersection multiplicities of regular sequences. However, as pointed out in \Cref{sec:newton1*:nonapp}, those techniques are {\em not} always applicable. In this section we present some other interpretations in terms of intersection multiplicities of regular sequences which, in theory, are more generally available, but seem rather non-constructive (those from \Cref{sec:ind:asymp,sec:ind:pos}) or complicated to obtain (the interpretation from \Cref{sec:ind:res}). Each of these interpretations, say interpretation $\mscrI$, gives rise to a (similarly non-constructive or complicated) algorithm for counting isolated solutions of (square) systems of polynomials as follows:
\begin{enumerate}
\item start with an {\em upper} bound which is exact in a generic scenario (e.g.\ B\'ezout's bound or some other Bernstein-Kushnirenko type estimate \cite[Theorems X.4 and X.7]{howmanyzeroes} when the ambient space is $\kk^n$ or $\pp^n$); if the system is {\em non-degenerate}, i.e.\ there are no ``extra intersections'' (e.g.\ ``intersections at infinity'' or non-isolated points in the intersection), then this bound gives the correct result;
\item \label{sol:step:interprete} otherwise express the corresponding ordered multiplicities in terms of usual intersection multiplicities of regular sequences using interpretation $\mscrI$;
\item \label{sol:step:newton} compute Bernstein-Kushnirenko type {\em lower} bounds for the latter intersection multiplicity in terms of associated Newton diagrams (using \Cref{thm:mult0} and \Cref{cor:mult0}); if the system is non-degenerate, then these estimates are exact;
\item \label{sol:step:repeat} otherwise there are excess intersections in corresponding toric blow-ups at the origin, and repeat the process starting from step \eqref{sol:step:interprete} 
    to estimate the corresponding ordered intersection multiplicities.
\end{enumerate}

\subsection{An ``asymptotic'' interpretation} \label{sec:ind:asymp}
Consider the set up from \Cref{sec:formula:set-up:0}. Given $\vec g = (g_1, \ldots, g_n) \allowbreak \in \prod_i L_i$ and $\vec r = (r_1, \ldots, r_n) \in \rr^n$, write $\tilde X(\vec g; \vec r)$ for the ``weighted blow up'' of $\bar X \times \kk$ considered in \Cref{sec:formula:blow-up}, i.e.\ the closure in $\bar X \times \kk \times \pp^n(r_1, \ldots, r_n, 1)$ of the graph of the map
\begin{align*}
\bar X \times \kk \ni (x, t) &\mapsto [\frac{f_1}{g_1}(x): \cdots : \frac{f_n}{g_n}(x): t] \in \pp^n(r_1, \ldots, r_n, 1)
\end{align*}
As in \Cref{sec:formula:blow-up} denote the (weighted) homogeneous coordinates on $\pp^n(r_1, \ldots, r_n, 1)$ by $[w_1: \cdots :w_{n+1}]$, and let $u_i := w_i/(w_{n+1})^{r_i}$, $i = 1, \ldots, n$, be the standard affine coordinates on $\tilde U_{n+1} := \tilde X(\vec g; \vec r) \setminus V(w_{n+1})$. The following result shows that under \eqref{condition:>>} 
the ordered intersection multiplicity can be interpreted as a sum of usual intersection multiplicities at nonsingular points on the normalization of $\tilde U_{n+1}$.

\begin{proprop}
Let $X$ be an $n$-dimensional affine variety, $f_1, \ldots, f_n \in \kk[X]$, and for each $i = 1, \ldots, n$, let $L_i$ be a base-point free (finite dimensional) vector subspace of $\kk[X]$ containing $f_i$. 
Let $Z$ be a distinguished component for (the ordered intersection of the hypersurfaces defined by) $V(f_1), \ldots, V(f_n)$ (\Cref{defn:distinguished}), and $\check{Z}$ be the intersection of $Z$ with the union of all other distinguished components $Z'$ with $\codim(Z') \geq \codim(Z)$. %
Then for each nonempty open subset $Z^*$ of $Z$, there are $\nu \in \rr$, and a nonempty open subset $L^*$ of $\prod_i L_i$ such that under \eqref{condition:>nu}, for each $(g_1, \ldots, g_n) \in L^*$, the normalization $\tilde U'_{n+1}$ of $\tilde U_{n+1}$ is nonsingular at each point of $V(t, u_1 - 1, \ldots, u_n - 1)$, and
\begin{align*}
\multord{F_1}{F_n}{Z}
    &= \sum_{u \in \tilde U'_{n+1} \cap \pi^{-1}(Z^*\setminus \check{Z})} \multsub{t,u_1-1}{u_n-1}{u}
\end{align*}
where $\pi: \tilde U'_{n+1} \to \bar X$ is the natural projection.
\end{proprop}

\begin{proof}
This follows from arguments as in the proof of \Cref{prop:formula:blow-up} and \Cref{prop:formula:constant} and the observation that a normal variety is nonsingular in codimension one.
\end{proof}

\subsection{An expression in positive characteristic} \label{sec:ind:pos}
Cowsik and Nori \cite{cowsik-nori} showed that affine curves are set theoretic complete intersections in positive characteristic. It follows that in positive characteristic each ``summand'' of $\multord{F_1}{F_n}{Z}$ can be represented as an intersection multiplicity up to a multiple; more precisely, for each $\vec i = (i_1, \ldots, i_\rho)$, consider the curves $C_{\vec i, \vec g} = V(I^{(\rho-1)}_{i_1, \ldots, i_\rho}, g_{i'_1}, \ldots, g_{i'_{n-\rho}})$ considered in \Cref{sec:local-curves}, where $i'_j$ are as usual the elements of $\{1, \ldots, n\}\setminus \{i_1, \ldots, i_\rho\}$, and $g_{i'_j}$ are generic elements from $\scrL_{i'_j}$. By Cowsik and Nori's theorem, for each irreducible component $C_{\vec i, \vec g; k}$ of $C_{\vec i, \vec g}$, one can find $h_{i_1, \ldots, i_\rho;k,l}$, $l = 1, \ldots, n-1$, such that set theoretically
\begin{align*}
C_{\vec i, \vec g; k}
    &= \bigcap_{l = 1, \ldots, n-1} V(h_{i_1, \ldots, i_\rho;k,l})
\end{align*}

\begin{proprop}
With the above set up,
\begin{align*}
\len(\local{X}{Z}/I^{(\rho)}_{i_1, \ldots, i_\rho} \local{X}{Z}) \deg_{i'_1, \ldots,i'_{n-\rho}}(Z)
    &= \sum_k
        \frac{1}{n_{i_1, \ldots, i_\rho; k}}\sum_{z \in Z}
        \multsub{h_{i_1, \ldots, i_\rho;k,1}}{h_{i_1, \ldots, i_\rho;k,n-1}, f_{i_\rho}}{z} \\
\multord{F_1}{F_n}{Z}
    &= \sum_{\substack{i_1, \ldots, i_\rho, k \\ z \in Z}}
        \frac{1}{n_{i_1, \ldots, i_\rho; k}}
        \multsub{h_{i_1, \ldots, i_\rho;k,1}}{h_{i_1, \ldots, i_\rho;k,n-1}, f_{i_\rho}}{z}
    \end{align*}
\end{proprop}
\begin{proof}
This follows immediately from \Cref{prop:local-curves}.
\end{proof}

\subsection{Expression for irreducible crossing divisors} \label{sec:ind:res}
Given a Cartier divisor $D$ on $X$, we say that $D$ is an {\em irreducible crossing divisor} if for each $z \in \supp(D)$, there is an open neighborhood $U$ of $z$ in $X$, and regular functions $h_1, \ldots, h_k$ on $U$ such that
\begin{enumerate}
\item $D$ is represented on $U$ by a monomial in $h_1, \ldots, h_k$,
\item \label{property:irreducible-crossing:irr} $V(h_{k_1}, \ldots, h_{k_m}) \cap U$ is irreducible for each choice of $k_1, \ldots, k_m$.
\end{enumerate}
We say that $D$ is a {\em simplicial crossing divisor} if the following stronger version of property \eqref{property:irreducible-crossing:irr} is satisfied:
\begin{enumerate}[label=(\arabic*${'}$)]
\setcounter{enumi}{1}
\item \label{property:irreducible-crossing:simpl} if $k_1, \ldots, k_m$ are pairwise distinct, then $V(h_{k_1}, \ldots, h_{k_m}) \cap U$ is irreducible of {\em codimension} $m$ if $m \leq n$, and empty if $m > n$.
\end{enumerate}
We now show that the ordered intersection multiplicity of $F_1, \ldots, F_n$ along a variety $Z$ can be expressed in terms of the usual intersection multiplicities if $\bigcup \supp(F_i)$ is the support of an irreducible crossing divisor near generic points of $Z$. Since a strict normal crossing divisor is a simplicial crossing divisor, and in particular, an irreducible crossing divisor, this gives another inductive route to the affine B\'ezout problem when resolution of singularities is available (e.g.\ in characteristic zero).

\begin{proprop} \label{prop:irreducible-crossing}
Let $Z$ be an irreducible component of $Z^{(\rho)}_{i_1, \ldots, i_\rho}$ (as defined in \Cref{thm:Z}) such that $Z \not\subseteq \sing(X)$. Assume there is an open subset $\tilde U$ of $X$ such that $\tilde U \cap Z \neq \emptyset$, and $\tilde U \cap \sum_i F_i$ is an irreducible crossing divisor. Then there is an open subset $U$ of $\tilde U$ and $f_1, \ldots, f_n \in \kk[U]$ such that
\begin{enumerate}
\item $U \cap Z \neq \emptyset$,
\item each $F_i$ is represented by $f_i$ on $U$, and
\item there are divisors $f_{i_1, \ldots, i_\rho;j,s_j}$ of $f_{i_j}$ in $\local{X}{Z}$, $j = 1, \ldots, \rho$, such that
\begin{align*}
\len(\local{X}{Z}/I^{(\rho)}_{i_1, \ldots, i_\rho} \local{X}{Z})
    &= \sum_{s_1, \ldots, s_\rho} \multsub{f_{i_1, \ldots, i_\rho;1,s_1}}{f_{i_1, \ldots, i_\rho;\rho,s_\rho}}{Z}
\end{align*}
\end{enumerate}
If $\sum_i F_i$ is a simplicial crossing divisor, then there is only one possible choice of $s_j$ for each $j$, and more precisely, there are irreducible elements $f'_{i_1, \ldots, i_\rho; j} \in \local{X}{Z}$ and positive integers $\alpha_{i_1, \ldots, i_\rho;j}$, $j = 1, \ldots, \rho$, such that for each $j$,
\begin{enumerate}[resume]
\item up to multiplication by a unit in $\local{X}{Z}$, $f_{i_j}$ equals $(f'_{i_1, \ldots, i_\rho; j})^{\alpha_{i_1, \ldots, i_\rho;j}}$ times a monomial in $f'_{i_1, \ldots, i_\rho; j+1}, \ldots, f'_{i_1, \ldots, i_\rho; \rho}$; in particular,
\item $f_{i_\rho}$ equals $(f'_{i_1, \ldots, i_\rho; \rho})^{\alpha_{i_1, \ldots, i_\rho;\rho}}$ times a unit in $\local{X}{Z}$, and
\begin{align*}
\len(\local{X}{Z}/I^{(\rho)}_{i_1, \ldots, i_\rho} \local{X}{Z})
    &= \multsub{(f'_{i_1, \ldots, i_\rho; 1})^{\alpha_{i_1, \ldots, i_\rho;1}}}
        {(f'_{i_1, \ldots, i_\rho; \rho})^{\alpha_{i_1, \ldots, i_\rho;\rho}}}{Z}
\end{align*}
\end{enumerate}
\end{proprop}

\begin{proof}
Fix an open subset $U$ of $\tilde U$ such that $U \cap Z \neq \emptyset$, and each $F_i$ is represented by some $f_i \in \kk[U]$, and there are regular functions $h_1, \ldots, h_k$ on $U$ such that
\begin{prooflist}
\item $\sum_i F_i$ is represented on $U$ by a monomial in $h_1, \ldots, h_k$, and
\item \label{irreducible-crossing:irr} $V(h_{k_1}, \ldots, h_{k_m}) \cap U$ is irreducible for each choice of $k_1, \ldots, k_m$.
\end{prooflist}
Choosing $U$ sufficiently small if necessary, we may also assume that
\begin{prooflist}[resume]
\item each $h_i$ is an irreducible element in $\local{X}{Z}$, and
\item \label{irreducible-crossing:subset} $Z \cap U \subseteq V(h_1, \ldots, h_k)$.
\end{prooflist}
Now consider the construction of the ideals $I^{(j)}_{i_1, \ldots, i_\rho}$ from \Cref{thm:length}. Since $Z \not\subseteq \sing(X)$, the Auslander–Buchsbaum theorem implies that $\local{X}{Z}$ is a UFD, so that
\begin{prooflist}[resume]
\item \label{irreducible-crossing:mon} each $f_i$ is a monomial in $h_1, \ldots, h_k$.
\end{prooflist}
Since all sets of the form $V(h_{s_1}, \ldots, h_{s_m}) \cap U$ are irreducible, this implies that for each $j = 1, \ldots, \rho$,
\begin{prooflist}[resume]
\item \label{irreducible-crossing:j} each primary component of $I^{(j)}_{i_1, \ldots, i_\rho} \cap U$ is of the form $V(h_{s_1}^{m_1}, \ldots, h_{s_j}^{m_j})$ where $s_1, \ldots, s_j$ are pairwise distinct, and $h_{s_l}^{m_l}$ are divisors of $f_{i_l}$, $l = 1, \ldots, j$.
\end{prooflist}
The preceding observation, applied to $j = \rho$, implies that
\begin{align*}
\len(\local{X}{Z}/I^{(\rho)}_{i_1, \ldots, i_\rho} \local{X}{Z})
    &= \sum_{s_1, \ldots, s_\rho} \multsub{h_{s_1}^{m_1}}{h_{s_\rho}^{m_\rho}}{Z}
\end{align*}
where the sum is over all $s_1, \ldots, s_\rho$ such that $h_{s_1}^{m_1}, \ldots, h_{s_\rho}^{m_\rho}$ generate a primary component of $I^{(\rho)}_{i_1, \ldots, i_\rho}$. This proves the first assertion of \Cref{prop:irreducible-crossing}. Now since $Z$ is an irreducible component of $V(I^{(\rho)}_{i_1, \ldots, i_\rho})$, observations \ref{irreducible-crossing:subset} and \ref{irreducible-crossing:j} imply that
\begin{prooflist}[resume]
\item $Z = V(h_1, \ldots, h_k)$.
\end{prooflist}
Replacing $U$ by a smaller open subset if necessary, we may assume that $Z \cap U$ is the only irreducible component of $Z^{(\rho)}_{i_1, \ldots, i_\rho} \cap U = Z^{(\rho-1)}_{i_1, \ldots, i_\rho} \cap V(f_{i_\rho}) \cap U$. Now assume $\sum_i F_i$ is a simplicial crossing divisor. Then
\begin{prooflist}[resume]
\item $\rho = \codim(Z) = \codim(h_1, \ldots, h_k) = k$.
\end{prooflist}
On the other hand, it follows from the definition of the $I^{(j)}_{i_1, \ldots, i_\rho}$ that
\begin{prooflist}[resume]
\item if $h_{s_1}^{m_1}, \ldots, h_{s_{j}}^{m_{j}}$ generate a primary component of $I^{(j)}_{i_1 ,\ldots, i_\rho}$, then none of the $h_{s_l}$ divides $\prod_{j'>j}f_{i_{j'}}$.
\end{prooflist}
Since $\rho = k$, applying the preceding observation to $j = \rho-1$, then to $j = \rho -2$, and so on up to $j = 1$, it follows after a renumbering of the $h_s$ if necessary that for each $j = 1, \ldots, \rho$,
\begin{prooflist}[resume]
\item up to multiplication by a unit in $\local{X}{Z}$, $f_{i_j}$ equals $h_j^{m_j}$ times a monomial in $h_{j+1}, \ldots, h_\rho$;
\item $I^{(j)}_{i_1, \ldots, i_\rho}$ is generated by $h_1^{m_1}, \ldots, h_j^{m_j}$.
\end{prooflist}
This completes the proof of \Cref{prop:irreducible-crossing}.
\end{proof}

\part{Appendix}

\appendix
\section{Miscellaneous algebraic results used in the article}

\begin{lemma} \label{lemma:nilorder}
Let $C$ be a (possibly non-reduced) curve, $z \in C$, and $t$ be a non zero-divisor and non-unit in $\local{C}{z}$, so that $m := \len(\hatlocal{C}{z}/t\hatlocal{C}{z})$ is {\em positive} and finite. If $h \in \local{C}{z}$ is nilpotent, then $h^m = 0$.
\end{lemma}

\begin{proof}
Since $t$ is a non zero-divisor in $\local{C}{z}$, it follows from standard commutative algebra that there is a $\kk[[t]]$-module isomorphism
\begin{align*}
\hatlocal{C}{z} \cong \kk[[t]]^m
\end{align*}
(see e.g.\ \cite[Proof of Proposition C.6]{howmanyzeroes}). Since $t$ is a non zero-divisor, it follows that
\begin{align*}
(\hatlocal{C}{z})_t \cong \kk((t))^m
\end{align*}
as a vector space over the field $\kk((t))$ of Laurent series in $t$. Now pick the smallest $s \geq 1$ such that $h^s = 0$. We claim that $1, h, \ldots, h^{s-1}$ are linearly independent over $\kk((t))$. Indeed, otherwise there is a relation of the form
\begin{align*}
h^{s'}(\phi_0(t) + \sum_{j=1}^{s''} h^j\phi_j(t)) = 0
\end{align*}
with $0 \leq s' < s$ and each $\phi_j \in \kk[[t]]$. However, since $t$ is a non zero-divisor and $h$ is nilpotent, it follows that $\phi_0(t) + \sum_{j=1}^{s''} h^j\phi_j(t)$ is also a non zero-divisor. Then $h^{s'} = 0$, contradicting the choice of $s$. This proves the claim. It follows that $s \leq m$, as required.
\end{proof}

\begin{prop} \label{prop:length_f}
Let $I$ be an ideal of pure codimension $k$ in a Noetherian ring $R$. Let $f \in R$ and $I' := IR_f \cap R$.
\begin{enumerate}
\item If $f \in \sqrt{I}$, then $I' = R$.
\item If $f \not\in \sqrt{I}$, then $I'$ has pure codimension $k$ and $I' + fR$ has pure codimension $k+1$ in $R$. Moreover, for each codimension $k+1$ prime $\ppp'$ of $R$,
\begin{align}
\begin{split}
\len(R_{\ppp'}/(I'R_{\ppp'} + fR_{\ppp'}))
    &= \sum_{\substack{\ppp \subseteq \ppp' \\ \codim(\ppp) = k}}
        \len(R_\ppp/I'R_\ppp)\len(R_{\ppp'}/(\ppp R_{\ppp'} + fR_{\ppp'})) \\
    &= \sum_{\substack{f \not\in \ppp \subseteq \ppp' \\ \codim(\ppp) = k}}
        \len(R_\ppp/IR_\ppp)\len(R_{\ppp'}/(\ppp R_{\ppp'} + fR_{\ppp'}))
\end{split}
\label{eq:length_f}
\end{align}
\end{enumerate}
\end{prop}

\begin{proof}
The first assertion is immediate. So assume $f \not\in \sqrt{I}$. 
Then $I' + fR$ has pure codimension $k+1$ in $R$ due to Krull's principal ideal theorem. %
So for each prime ideal $\ppp$ associated to $I$, we have
\begin{align*}
\codim(I' + fR) = k+1 > k = \codim(\ppp)
\end{align*}
so that
\begin{align*}
I' + fR \not \subseteq \ppp
\end{align*}
In particular, if $f \in \ppp$, then $I' \not\subseteq \ppp$, so that $I'R_\ppp = R_\ppp$. The second equality of \eqref{eq:length_f} then follows from the following observation that
\begin{align*}
I'R_\ppp
    &=
    \begin{cases}
    R_\ppp & \text{if}\ f \in \ppp, \\
    IR_\ppp & \text{if}\ f \not\in \ppp.
    \end{cases}
\end{align*}
Fix a codimension $k+1$ prime $\ppp'$ of $R$ containing $I' + fR$. An application of \cite[Lemma A.2.7]{fultersection} with $M = A = R_{\ppp'}/I'R_{\ppp'}$ shows that
\begin{align*}
\sum_{\substack{I' \subseteq \ppp \subseteq \ppp' \\ \codim(\ppp) = k}}
        \len(R_\ppp/I'R_\ppp)\len(R_{\ppp'}/(\ppp R_{\ppp'} + fR_{\ppp'}))
    &= e_{R_{\ppp'}/I'R_{\ppp'}}(f)
\end{align*}
where $e_{\cdot}(\cdot)$ is the ``multiplicity'' defined by
\begin{align*}
e_A(a) &:= \len(A/aA) - \len(0:a)
\end{align*}
However, it follows from the definition of $I'$ that $f$ is a non zero-divisor in $R_{\ppp'}/I'R_{\ppp'}$, and consequently,
\begin{align*}
e_{R_{\ppp'}/I'R_{\ppp'}}(f) &= \len(R_{\ppp'}/(I'R_{\ppp'} + fR_{\ppp'}))
\end{align*}
as required.
\end{proof}

\section{A result about completions of local rings at a subvariety}

The goal of this section is to prove \Cref{claim:integral-claim-1} which is used in \Cref{proof:length:confirmation:I-g-subset} to prove \Cref{length:confirmation:I-g-subset}.

\begin{thm} \label{prop:generically-integral}
Let $X$ be an irreducible affine variety, $Z$ be an irreducible subvariety of $X$. Let $\phi_1, \ldots, \phi_q, f_1, \ldots, f_r \in \kk[X]$ such that
\begin{enumerate}
\item $\kk[Z]$ is algebraic over $\kk[\phi_1|_Z, \ldots, \phi_q|_Z]$,
\item $Z$ is set-theoretically the closed subset $V(f_1, \ldots, f_r)$ of $X$.
\end{enumerate}
Then
\begin{prooflist}
\item \label{generically-integral:0} $\kk(\phi_1, \ldots, \phi_q)[[f_1, \ldots, f_r]]$ is a subring of $\hatlocal{X}{Z}$, and there is a nonzero element $\phi \in \kk[\phi_1, \ldots, \phi_q]$ such that $\kk[X]$ is integral over $\kk[\phi_1, \ldots, \phi_q]_\phi[[f_1, \ldots, f_r]]$ in $\hatlocal{X}{Z}$.
\item \label{generically-integral:1} If $\kk[Z]$ is integral over $\kk[\phi_1|_Z, \ldots, \phi_q|_Z]$, then
\begin{prooflist}
\item \label{generically-integral:1:1} we can take $\phi = 1$ in the preceding sentence, i.e.\ $\kk[X]$ is integral over $\kk[\phi_1, \ldots, \phi_q][[f_1, \ldots, f_r]]$ in $\hatlocal{X}{Z}$,
\item \label{generically-integral:1:k[X]} in addition, the completion $\widehat{\kk[X]}$ of $\kk[X]$ with respect to the prime ideal of $Z$ is integral over the subring $\kk[\phi_1, \ldots, \phi_q][[f_1, \ldots, f_r]]$.
\end{prooflist}
\end{prooflist}
\end{thm}


\begin{proof}
Assume $X$ is a subvariety of $\kk^N$ with coordinates $(x_1, \ldots, x_N)$. We abuse the notation to use $\phi_i$ (respectively, $f_j$) also for a representative of $\phi_i$ (respectively, $f_j$) in $\kk[x_1, \ldots, x_N]$. Let $J$ be the ideal of $\kk[x_1, \ldots, x_N]$ generated by $f_1, \ldots, f_r$, so that
\begin{align*}
I(Z) = \sqrt{I(X)+J}
\end{align*}
For each $i$, $x_i|_Z$ is algebraic over $\kk(\phi_1|_Z, \ldots, \phi_q|_Z)$. Taking a power of the minimal equation of $x_i|_Z$ over $\kk(\phi_1|_Z, \ldots, \phi_q|_Z)$ if necessary, we have that
\begin{align*}
a_{i,0}(\phi_1, \ldots, \phi_q) x_i^{d_i} + \sum_{j=1}^{d_i} a_{i,j}(\phi_1, \ldots, \phi_q)\in (I(X)+J)
\end{align*}
for some $a_{i,0}, \ldots, a_{i,d_i} \in \kk[\phi_1, \ldots, \phi_q]$ with $a_{i,d_i} \neq 0$. It follows that
\begin{align*}
x_i^{d_i} + \sum_{j=1}^{d_i} \frac{a_{i,j}}{a_{i,0}} x_i^{d_i -j} = \frac{1}{a_{i,0}}\left( \sum_{j=1}^r g_{i,j}f_j + g'_i \right)
\end{align*}
where $g_{i,1}, \ldots, g_{i,r}, g'_i$ are polynomials in $(x_1, \ldots, x_N)$ and $g'_i \in I(X)$. Let
\begin{align*}
\phi := \prod_{i=1}^N a_{i,0} \in \kk[\phi_1, \ldots, \phi_q]
\end{align*}
and $\scrA$ be the (finite) set of all $\alpha = (\alpha_1, \ldots, \alpha_N) \in \znonnegg{N}$ such that $\alpha_i < d_i$ for each $i$.

\begin{proclaim} \label{claim:generically-finite-module}
Each $h \in \kk[X]$ can be represented in $\hatlocal{X}{Z}$ as a $\kk[\phi_1, \ldots, \phi_q]_\phi[[f_1, \ldots, f_r]]$-linear combination of $x^\alpha$, $\alpha \in \scrA$.
\end{proclaim}

\begin{proof}
Given $h \in \kk[x_1, \ldots, x_N]$, replacing each $(x_i)^{d_i}$ using the above equation and repeating as many times as necessary, one has that
\begin{align*}
h &= \sum_{\alpha \in \scrA} b_\alpha x^\alpha + \frac{1}{\phi^{m_1}}\left(\sum_{j=1}^{r} f_j h_j + h'_1\right)
\end{align*}
where the $b_\alpha \in \kk[\phi_1, \ldots, \phi_q]_\phi$, $h_1, \ldots, h_r, h'_1$ are polynomials in $(x_1, \ldots, x_N)$ and $h'_1 \in I(X)$.  Then continue the same process with each $h_j$ to yield
\begin{align*}
h &= \sum_{\alpha \in \scrA} b_\alpha x^\alpha
        + \sum_{j=1}^{r} \frac{1}{\phi^{m_1}} f_j
            \left(
                \sum_{\alpha \in \scrA} b_{\alpha,j} x^\alpha +
                \frac{1}{\phi^{m_2}}\left(\sum_{k=1}^{r} f_k h_{j,k} + h'_{1,j}\right)
            \right)
        + \frac{h'_1}{\phi^{m_1}}
\end{align*}
which can be rewritten as
\begin{align*}
h &= \sum_{\alpha \in \scrA} \left(b_\alpha + \sum_{j=1}^{r} \frac{b_{\alpha,j}}{\phi^{m_1}}f_j\right)x^\alpha
        + \frac{1}{\phi^{m_1 + m_2}}\left(\sum_{j,k} f_j f_k h_{j,k} + h'_2\right)
\end{align*}
This can be continued with
\begin{align*}
h &= \sum_{\alpha \in \scrA}
        \left(
            b_\alpha
            + \sum_j \frac{b_{\alpha,j}}{\phi^{m_1}}f_j
            + \sum_{j,k} \frac{b_{\alpha,j,k}}{\phi^{m_1+m_2}}f_jf_k
        \right) x^\alpha
        + \frac{1}{\phi^{m_1 + m_2+m_3}}\left(\sum_{j,k,l} f_j f_k f_l h_{j,k,l} + h'_3\right)
\end{align*}
and so on to obtain a sum $\sum_{\alpha \in \scrA} c_\alpha x^\alpha$ with $c_\alpha \in \kk[\phi_1, \ldots, \phi_q]_\phi [[f_1, \ldots, f_r]]$ such that for each $M$, $h|_X \equiv \sum_{\alpha \in \scrA} (c_\alpha x^\alpha)|_X$ modulo the ideal in $\hatlocal{X}{Z}$ generated by $f^{\alpha}$ with $|\alpha| > M$. Since each $f_j$ is in the maximal ideal of $\hatlocal{X}{Z}$, it follows (e.g.\ since the intersection of all powers of a proper ideal in a Noetherian ring is zero) that
\begin{align*}
h &= \sum_{\alpha \in \scrA} c_\alpha x^\alpha \in \hatlocal{X}{Z}
\end{align*}
This completes the proof.
\end{proof}


Assertions \ref{generically-integral:0} and \ref{generically-integral:1:1} of \Cref{prop:generically-integral} immediately follow from \Cref{claim:generically-finite-module} via the ``determinantal trick''. For the remaining assertion \ref{generically-integral:1:k[X]} it suffices to show that \Cref{claim:generically-finite-module} continues to hold if we replace $\kk[X]$ by $\widehat{\kk[X]}$, since then we can use the determinantal trick to prove integrality. So assume $\phi = 1$ as in assertion \ref{generically-integral:1}, and pick $h \in \widehat{\kk[X]}$. Assume $h$ is the inverse limit of $(h_n)_{n \geq 0}$, with
\begin{align*}
h_n \equiv h_m \mod (I(Z))^n
\end{align*}
for all $m \geq n$. We are going to construct $c_{\alpha, \beta} \in \kk[\phi_1, \ldots, \phi_q]$ such that $h = \sum_{\alpha, \beta} c_{\alpha,\beta}x^\alpha f^\beta$, where the $\alpha$ vary over $\scrA$ and the $\beta$ over $\znonnegg{r}$, and $x^\alpha, f^\beta$ are short-hands for respectively $\prod_{i=1}^N x_i^{\alpha_i}$ and $\prod_{i=1}^r f_i^{\beta_i}$. Fix a monomial ordering
\begin{align*}
0 = \beta^{0} \prec \beta^1 \prec \cdots
\end{align*}
of $\znonnegg{r}$ which has ``finite depth'', i.e.\ for each $\beta \in \znonnegg{r}$, there are only finitely many $\beta' \in \znonnegg{r}$ such that $\beta' \preceq \beta$. For each $k \geq 0$, let $J_k$ be the ideal generated by $\{f^{\beta^{j}}: j \geq k\}$. Fix a set of generators $g_1, \ldots, g_s$ of $I(Z)$. Since $\prec$ has finite depth, we can pick $0 < n_0 < n_1 < \cdots$ such that for each $\gamma \in \znonnegg{s}$ with $|\gamma| := \sum_i \gamma_i = n_k$, the ``$g$ monomial'' $g^\gamma := \prod_i (g_i)^{\gamma_i}$ is in $J_k$. Due to \Cref{claim:generically-finite-module}, $h_{n_1}$ has an expression of the form
\begin{align*}
h_{n_1} &= \sum_{\alpha, \beta} c_{n_1, \alpha,\beta}(\phi_1, \ldots, \phi_q)  x^\alpha f^\beta
\end{align*}
For each $\alpha \in \scrA$, define
\begin{align*}
c_{\alpha, \beta^0} &:= c_{n_1, \alpha, \beta^0}
\end{align*}
For each $n \geq n_1$, since $h_n - h_{n_1} \in (I(z))^{n_1}$, by definition of $n_1$ and using \Cref{claim:generically-finite-module} we can write
\begin{align*}
h_{n_2} - h_{n_1} &= \sum_{\alpha \in \scrA} \sum_{\beta \succeq \beta^1} c_{n_2, \alpha,\beta}(\phi_1, \ldots, \phi_q) x^\alpha f^\beta
\end{align*}
For each $\alpha \in \scrA$, define
\begin{align*}
c_{\alpha, \beta^1} &:= c_{n_1, \alpha, \beta^1} + c_{n_2, \alpha, \beta^1}
\end{align*}
In the same way, for each $k$, we write
\begin{align*}
h_{n_{k+1}} - h_{n_k} &= \sum_{\alpha \in \scrA} \sum_{\beta \succeq \beta^k} c_{n_{k+1}, \alpha,\beta}(\phi_1, \ldots, \phi_q) x^\alpha f^\beta, \\
c_{\alpha, \beta^k} &:= c_{n_1, \alpha, \beta^k} + \cdots + c_{n_{k+1}, \alpha, \beta^k}
\end{align*}
Finally, define
\begin{align*}
h' :=  \sum_{\alpha \in \scrA} \sum_{k \geq 0} c_{\alpha,\beta^k} x^\alpha f^{\beta^k},
\end{align*}
It is straightforward to check that
\begin{align*}
h' &\equiv h_{n_k} \mod J_k
\end{align*}
which implies that $h' = h$, as required to complete the proof.
\end{proof}

\begin{lemma}\label{claim:integral-claim-1}
In the set up of \Cref{prop:generically-integral}, assume $\kk[Z]$ is integral over $\kk[\phi_1|_Z, \ldots, \phi_q|_Z]$. Let $s \leq r$ and $f \in \kk[X]$ be such that $f$ does not vanish identically on any irreducible component of $V_s := V(f_1, \ldots, f_s)$ or equivalently, $f$ is not in any minimal prime ideal in $\kk[X]$ of the radical of the ideal $\qqq_s$ generated by $f_1, \ldots, f_s$. Then $f$ satisfies an identity of the form
\begin{align*}
f^d + \sum_{j=1}^d \xi_j f^{d-j} + \sum_{k=1}^s f_k \chi_k= 0
\end{align*}
in the completion $\widehat{\kk[X]}$ of $\kk[X]$ with respect to $I(Z)$, where
\begin{enumerate}
\item each $\xi_j \in \kk[\phi_1, \ldots, \phi_q][[f_1, \ldots, f_r]]$,
\item \label{integral-claim:xi-d} $\xi_d$ is nonzero and has a monomial in $(f_1, \ldots, f_r)$ in which $f_k$ does not appear for $k \leq s$,
\item each $\chi_k \in \widehat{\kk[X]}$.
\end{enumerate}
\end{lemma}

\begin{proof}
By \Cref{prop:generically-integral}, $f$ satisfies an integral equation in $\widehat{\kk[X]}$ of the form
\begin{align*}
f^d + \sum_{j=1}^d \xi_j f^{d-j} = 0
\end{align*}
with each $\xi_j \in \kk[\phi_1, \ldots, \phi_q][[f_1, \ldots, f_r]]$. If $\xi_d$ has a monomial in $(f_1, \ldots, f_r)$ in which $f_k$ does not appear for $k \leq s$, then \Cref{claim:integral-claim-1} clearly holds. So assume either $\xi_d = 0$ or each monomial in $\xi_d$ contains some $f_i$ with $i \leq s$. Now, by assumption, $f$ is a non zero-divisor in $\kk[X]/\sqrt{\qqq_s}$ (indeed, if $fg \in \sqrt{\qqq_s}$, then $fg$ vanishes identically on every irreducible component of $V(f_1, \ldots, f_s$), and consequently, $g$ vanishes identically on every irreducible component of $V(f_1, \ldots, f_s$)). It then follows from the flatness of completion that $f$ is also a non zero-divisor in $\widehat{\kk[X]}/\widehat {\sqrt{\qqq_s}} \cong \widehat{\kk[X]}/\sqrt{\qqq_s}\widehat{\kk[X]}$. However, by the assumption on $\xi_d$, the above integral equation of $f$ can be written as
\begin{align*}
f^{e}(f^{d-e} + \sum_{j=1}^{d-e} \xi_j f^{d-e-j}) \in \qqq_s \widehat{\kk[X]}
\end{align*}
where $d-e$ is the largest index such that $\xi_{d-e} \not \in \qqq_s \widehat{\kk[X]}$. Since $f$ is non zero-divisor in $\widehat{\kk[X]}/\sqrt{\qqq_s}\widehat{\kk[X]}$, it follows that
\begin{align*}
f^{d-e} + \sum_{j=1}^{d-e} \xi_j f^{d-e-j} \in \sqrt{\qqq_s} \widehat{\kk[X]}
\end{align*}
and consequently,
\begin{align*}
(f^{d-e} + \sum_{j=1}^{d-e} \xi_j f^{d-e-j})^m \in \qqq_s \widehat{\kk[X]}
\end{align*}
for some $m \geq 1$. This completes the proof.
\end{proof}

\section{Proof of \Cref{length:confirmation:I-g-subset}} \label{proof:length:confirmation:I-g-subset}
In this section we prove \Cref{length:confirmation:I-g-subset}. Given an $n$-dimensional affine variety $X$ and {\em fixed} $f_1, \ldots, f_n, \allowbreak g_1, \ldots, g_n \in \kk[X]$, we will consider $C(g_1, \allowbreak \ldots, \allowbreak g_n; \allowbreak r_1, \allowbreak \ldots, \allowbreak r_n)$ (defined as in \Cref{sec:prelim:C}) for various $r_1, \ldots, r_n$. As in the proof of \Cref{thm:length:confirmation}, we assume $Z$ is an irreducible component of $Z^{(\rho)}_{i_1, \ldots, i_\rho}$, and let $I'^{(j)}_{i_1, \ldots, i_\rho}, I''^{(j)}_{i_1, \ldots, i_\rho}$ be as in \Cref{sec:length:confirmation:I'I''}.

\begin{lemma} \label{prop:cutting-out}
There are $\nu^*_0 \in \rr$, a finite set $Q \subseteq \qq$, and an open subset $U^*$ of $X$, and $h''_j \in I''^{(j)}_{i_1, \ldots, i_\rho}\kk[U^*]$, $j = 1, \ldots, \rho - 1$, such that
\begin{enumerate}
\item $U^* \cap Z \neq \emptyset$,
\item \label{cutting-out:ord} if $r_1, \ldots, r_n \in \rr$ and a branch $B$ of $C = C(g_1, \ldots, g_n;r_1, \ldots, r_n)$ centered at $(U^* \cap Z) \times \{0\}$ satisfy condition \eqref{condition:nonzero>nu} for $\nu \geq \nu^*_0$, then for each $j = 1, \ldots, \rho - 1$,
\begin{align*}
\ord_B(h''_j)
    = \sum_{k=1}^{\rho-j} q_k r_{i_{j+k}}\ord_B(t)
    = \sum_{k=1}^{\rho-j} q_k \ord_B(f_{i_{j+k}})
\end{align*}
for $q_1, \ldots, q_{\rho - j} \in Q$.
\end{enumerate}
In addition, $U^*$ does not change (but $\nu^*_0$ possibly changes) if the $f_i$ are replaced by $(f_i)^{m_i}$ for $m_i \geq 1$.
\end{lemma}



\begin{proof}
The first ingredient of our proof of \Cref{prop:cutting-out} is \Cref{prop:nonzero:++j} below. Here we consider a stronger version of the condition \eqref{nonzero:+} from \Cref{sec:nonzero:+}: pick $1 \leq i_1 < i_2 < \cdots < i_\rho \leq n$, a nonempty open subset $Z^*_0$ of an irreducible subvariety $Z$ of $X$, and a collection $\scrH^*$ of pairs $(h,i)$ such that $1 \leq i \leq n$, and $h$ is a regular function on an open subset of $X$ containing $Z^*_0$. We say that $\scrH^*, Z^*_0$ satisfy \eqref{nonzero:++} if the following holds:
\begin{align}
\parbox{0.9\textwidth}{
for {\em each} $\mu^*, \nu^* \in \rr$ and {\em each} nonempty subset $Z^*$ of $Z^*_0$, condition \eqref{nonzero:+} holds for $\scrH^*, Z^*$, i.e.\ there are $\nu \geq \nu^*$, $r_1, \ldots, r_n \in \rr$ and a branch $B$ of $C = C(g_1, \ldots, g_n;r_1, \ldots, r_n)$ centered at $Z^* \times \{0\}$ such that \eqref{condition:nonzero>nu} holds, and in addition, $\ord_B(h) > \mu^* \ord_B(f_{i}) \geq \mu^* r_i\ord_B(t)$ for each $(h,i) \in \scrH^*$.
}
\tag{$\cancel{0}_+$}
\label{nonzero:++}
\end{align}
We also consider a stronger version of \eqref{nonzero:++}: we say that $\scrH^*, Z^*_0$ satisfy \eqref{nonzero:+++} if the following holds:
\begin{align}
\parbox{0.9\textwidth}{
for {\em each} $\mu^*, \nu^* \in \rr$ and {\em each} nonempty subset $Z^*$ of $Z^*_0$, there is $\nu^*_+ \geq \nu^*$ such that for each $\nu \geq \nu^*_+$,
\begin{defnlist}
\item there are $r_1, \ldots, r_n \in \rr$ and a branch $B$ of $C = C(g_1, \ldots, g_n;r_1, \ldots, r_n)$ centered at $Z^* \times \{0\}$ such that \eqref{condition:nonzero>nu} holds, and in addition,
\item for {\em each} such $B$, $\ord_B(h) > \mu^* \ord_B(f_{i}) \geq \mu^* r_i\ord_B(t)$ for each $(h,i) \in \scrH^*$.
\end{defnlist}
}
\tag{$\cancel{0}_{++}$}
\label{nonzero:+++}
\end{align}

\begin{prolemma} \label{prop:nonzero:++j}
Fix $i^* = i_1 < i_2 < \cdots < i_\rho \leq n$ (where $i^*$, as in \Cref{observation:ord}, is the smallest index such that $f_{i^*}$ is not identically zero), and an irreducible subvariety $Z$ of $X$. Let $U^*$ be an open subset of $X$ such that $U^* \cap Z \neq \emptyset$, and $\scrH$ be a collection of pairs $(h, i)$, where $h \in \kk[U^*]$ and $1 \leq i \leq n$, such that
\begin{defnlist}
\item \label{nonzero:++j:assumption:H} condition \eqref{nonzero:++} (respectively, \eqref{nonzero:+++}) is satisfied with $\scrH^* = \scrH$ and $Z^*_0 = U^* \cap Z$.
\end{defnlist}
Assume in addition that
\begin{defnlist}[resume]
\item there is $j$, $1 \leq j \leq \rho - 1$, such that
\begin{defnlist}
\item $\scrH$ consists of $j$ pairs $(h_{k}, \tilde i_{k})$, $k = 1, \ldots, j$;
\item $\tilde i_{k} \leq i_{j+1}$ for each $k = 1, \ldots, j$;
\item \label{nonzero:++j:assumption:codimension} $V(h_1, \ldots, h_j) \cap \bigcup_{k = j+1}^\rho V(f_{i_{k}})$ is (set-theoretically) a codimension $j+1$ subvariety of $U^*$ containing $U^* \cap Z$.
\end{defnlist}
\end{defnlist}
Then
\begin{enumerate}
\item \label{nonzero:++j:h} There are $h_{j+1}, \ldots, h_\rho \in \kk[U^*]$ such that
\begin{enumerate}
\item $V(h_1, \ldots, h_k)$ has codimension $k$ for each $k = j+1, \ldots, \rho$,
\item $V(h_1, \ldots, h_k) \cap \bigcup_{k'=k+1}^\rho V(f_{i_{k'}})$ has codimension $k+1$ for each $k = j, \ldots, \rho-1$,
\item \eqref{nonzero:++} (respectively, \eqref{nonzero:+++}) continues to hold with $Z^*_0 = U^* \cap Z$ and $\scrH^* = \scrH \allowbreak \cup \{(h_{j+1}, i_{j+2}), \allowbreak \ldots, \allowbreak (h_{\rho-1}, i_{\rho})\}$.
\end{enumerate}
In addition, one can take
\begin{enumerate}[resume]
\item $h_\rho := f_{i_\rho}$.
\end{enumerate}
\end{enumerate}
Pick $h \in \kk[U^*]$ which does not vanish identically on any irreducible component of $V(h_1, \ldots, h_j)$ containing $U^* \cap Z$.
\begin{enumerate}[resume]
\item \label{nonzero:++j:h:++} If the \eqref{nonzero:++}-version of assumption \ref{nonzero:++j:assumption:H} holds, then there is a {\em finite} set $Q \subseteq \qq$, and an open subset $U'^* \subseteq U^*$ intersecting $Z$ such that for each $\nu^* \in \rr$ and each nonempty open subset $Z^*$ of $U'^* \cap Z$, there are $\nu \geq \nu^*, r_1, \ldots, r_n \in \zz$ and a branch $B$ of $C = C(g_1, \ldots, g_n;r_1, \ldots, r_n)$ centered at $Z^* \times \{0\}$ such that \eqref{condition:nonzero>nu} holds, and
\begin{align*}
\ord_B(h)
    = \sum_{k=1}^{\rho-j} q_k r_{i_{j+k}}\ord_B(t)
    = \sum_{k=1}^{\rho-j} q_k \ord_B(f_{i_{j+k}})
\end{align*}
for some $q_1, \ldots, q_{\rho-j} \in Q$.

\item \label{nonzero:++j:h:+++} If the \eqref{nonzero:+++}-version of assumption \ref{nonzero:++j:assumption:H} holds, then there are $\nu^* \in \rr$, a {\em finite} set $Q \subseteq \qq$, and an open subset $U'^* \subseteq U^*$ intersecting $Z$ such that for each $\nu \geq \nu^*$ and each nonempty open subset $Z^*$ of $U'^* \cap Z$,
\begin{enumerate}
\item there are $r_1, \ldots, r_n \in \rr$ and a branch $B$ of $C = C(g_1, \ldots, g_n;r_1, \ldots, r_n)$ centered at $Z^* \times \{0\}$ such that \eqref{condition:nonzero>nu} holds, and in addition,
\item for {\em each} such $B$,
\begin{align*}
\ord_B(h)
    = \sum_{k=1}^{\rho-j} q_k r_{i_{j+k}}\ord_B(t)
    = \sum_{k=1}^{\rho-j} q_k \ord_B(f_{i_{j+k}})
\end{align*}
for some $q_1, \ldots, q_{\rho-j} \in Q$.
\end{enumerate}
\item In addition, $U'^*$ from assertions \eqref{nonzero:++j:h:++} and \eqref{nonzero:++j:h:+++} does {\em not} change if the $f_i$ and $h_j$ are replaced respectively by $(f_i)^{m_i}$ and $(h_j)^{n_j}$ for $m_i, n_j \geq 1$.
\end{enumerate}
\end{prolemma}

\begin{proof}
Note that $Z  \cap U^* \subseteq V(f_1, \ldots, f_n) \cap V(h_1, \ldots, h_j)$ (due to assumptions \ref{nonzero:++j:assumption:H} and \ref{nonzero:++j:assumption:codimension}, and the definition of $C(g_1, \ldots, g_n; r_1, \ldots, r_n)$). \Cref{prop:nonzero:+j} implies that $V(h_1, \ldots, h_j) \not\subseteq V(f_{i_{j+1}})$. Pick $h_{j+1}$ such that
\begin{prooflist}
\item $h_{j+1}$ vanishes identically on each irreducible component of $V(h_1, \ldots, h_j) \cap V(f_{i_{j+1}})$ not contained in $\bigcup_{k=j+1}^\rho V(f_{i_k})$,
\item $V(h_1, \ldots, h_{j+1}) \cap \bigcup_{k = j+2}^\rho V(f_{i_{k}})$ has codimension $j+2$ in $U^*$.
\end{prooflist}
It follows from the arguments of the proof of \Cref{prop:nonzero:+j} that the property \ref{nonzero:++j:assumption:H} from the statement of \Cref{prop:nonzero:++j} continues to hold if $\scrH$ is replaced by $\scrH \cup \{(h_{j+1}, i_{j+2})\}$. Repeating this process we obtain $h_{j+1}, \ldots, h_\rho \in \kk[U^*]$ such that
\begin{prooflist}[resume]
\item \label{nonzero:++j:codimension:h} $V(h_1, \ldots, h_k)$ has codimension $k$ for each $k = j+1, \ldots, \rho$,
\item \label{nonzero:++j:codimension:f} $V(h_1, \ldots, h_k) \cap \bigcup_{k'=k+1}^\rho V(f_{i_{k'}})$ has codimension $k+1$ for each $k = j, \ldots, \rho-1$,
\item \label{nonzero:++j:nonzero:+} property \ref{nonzero:++j:assumption:H} continues to hold with $Z^*_0 = U^* \cap Z$ and $\scrH^* = \scrH \cup \{(h_{j+1}, i_{j+2}), \ldots, (h_{\rho-1}, i_{\rho})\}$.
\end{prooflist}
It is straightforward to check that we can in fact take
\begin{prooflist}[resume]
\item $h_\rho := f_{i_\rho}$.
\end{prooflist}
In particular, this implies assertion \eqref{nonzero:++j:h} of \Cref{prop:nonzero:++j}. Now we prove assertions \eqref{nonzero:++j:h:++} and \eqref{nonzero:++j:h:+++}. Due to property \ref{nonzero:++j:codimension:h}, we can pick an open subset $U'^*_0$ of $U^*$ such that $U'^*_0 \cap Z \neq \emptyset$ and $U'^*_0 \cap V(h_1, \ldots, h_\rho)$ has only one irreducible component, namely $Z \cap U'^*_0$.

\begin{subproclaim} \label{claim:nonzero:++j:h_k}
There are $\mu'^*, \nu'^* \in \rr$, a {\em finite} set $Q' \subseteq \qq$ and an open subset $U''^*_0$ of $U'^*_0$ such that if \eqref{nonzero:+} is satisfied for some $\mu^* \geq \mu'^*$ and $\nu^* \geq \nu'^*$ with $\scrH^* = \scrH \cup \{(h_{j+1}, i_{j+2}), \ldots, (h_{\rho-1}, i_{\rho})\}$ and $Z^* = U''^*_0 \cap Z$, then for each $k = j+1, \ldots, \rho$,
\begin{align*}
\ord_B(h_k)
    = \sum_{l=k}^\rho q'_l r_{i_l}\ord_B(t)
    = \sum_{l=k}^\rho q'_l \ord_B(f_{i_l})
\end{align*}
for some $q'_k, \ldots, q'_\rho \in Q'$.
\end{subproclaim}

\begin{proof}
Since $h_\rho = f_{i_\rho}$, the claim is true for $k = \rho$. We proceed by backward induction on $k$ and assume that
\begin{prooflist}[resume]
\item \Cref{claim:nonzero:++j:h_k} holds for $k = \rho, \rho - 1, \ldots, k'+1$ with a finite subset $Q'_{k'+1}$ in place of $Q'$, and a nonempty subset $U''^*_{k'+1}$ of $U'^*_0$ in place of $U''^*_0$.
\end{prooflist}
By Noether normalization there are $u_1, \ldots, u_{n-\rho} \in \kk[U'^*]$ such that
\begin{prooflist}[resume]
\item \label{nonzero:++j:independent} $u_1|_Z, \ldots, u_{n-\rho}|_Z$ are algebraically independent over $\kk$, and
\item \label{nonzero:++j:integral} $\kk[Z \cap U'^*]$ is integral over $\kk[u_1|_Z, \ldots, u_{n-\rho}|_Z]$.
\end{prooflist}
Due to properties \ref{nonzero:++j:codimension:f} and \ref{nonzero:++j:integral}, \Cref{claim:integral-claim-1} applies (with $h_i$ in place of the $f_i$, the $u_i$ in place of the $\phi_i$, and $f = f_{i_{k'}}$ and $s = k'-1$) and implies that $f_{i_{k'}}$ satisfies an identity in the completion $\widehat{\kk[U'^*]}$ of $\kk[U'^*]$ with respect to the ideal $I(Z)$ of $Z$ of the form
\begin{align*}
f_{i_{k'}}^d + \sum_{e=1}^d \xi_e f_{i_{k'}}^{d-e} + \sum_{k=1}^{k'-1} h_k \chi_k
    &= 0 \in \widehat{\kk[U'^*]}
\end{align*}
where
\begin{prooflist}[resume]
\item each $\xi_e \in \kk[u_1, \ldots, u_{n-\rho}][[h_1, \ldots, h_\rho]]$,
\item \label{integral-claim:xi-d} $\xi_d$ is nonzero and has a monomial in $(h_1, \ldots, h_\rho)$ in which $h_k$ does not appear for $k \leq k'-1$,
\item each $\chi_k \in \widehat{\kk[X]}$.
\end{prooflist}
For each $e$, let $\alpha_e := (\alpha_{e,1}, \ldots, \alpha_{e,\rho})$ denote the {\em lexicographically minimal} element of $\supp(\xi_d)$ (which is the set of all $(\beta_1, \ldots, \beta_\rho)$ such that the coefficient of $\prod_{k=1}^\rho h_k^{\beta_k}$ in $\xi_e$ is nonzero). Then the above integral equation of $f_{i_{k'}}$ can be expressed as
\begin{align*}
f_{i_{k'}}^d
    + \sum_l (c_{e_l}\prod_{k = k'}^\rho h_k^{\alpha_{e_l,k}})f_{i_{k'}}^{d-e_l}
    + \sum_l \xi'_{e_l}  f_{i_{k'}}^{d-e_l}
    + \sum_{k=1}^{k'-1} h_k \chi'_k
    &= 0
\end{align*}
where the second (and the third) sum is over all $l$ such that $\alpha_{e_l, k} = 0$ for $k = 1, \ldots, k' - 1$, and for each such $l$,
\begin{prooflist}[resume]
\item $c_{e_l} \in \kk[u_1, \ldots, u_{n-\rho}]$ is the {\em coefficient} $\prod_{k=1}^{\rho} h_k^{\alpha_{e_l, k}}$ in $\xi_{e_l}$ is nonzero, and
\item $\xi'_{e_l} := \xi_{e_l} - c_{e_l}\prod_{k=1}^{\rho} h_k^{\alpha_{e_l, k}}$.
\end{prooflist}
Since $u_1|_Z, \ldots, u_n|_Z$ are algebraically independent over $\kk$, none of the $c_{e_l}$ vanishes identically on $Z$. Consequently, $U''^*_{k'} := U''^*_{k'+1} \setminus V(\prod_l c_{e_l})$ has a nonempty intersection with $Z$ (where $U''^*_{k'+1}$ is from the inductive hypothesis). It then follows from the inductive hypothesis and property \ref{nonzero:++j:nonzero:+} that if $r_1, \ldots, r_n \in \zz$ and $C(g_1, \ldots, g_n;r_1, \ldots, r_n)$ centered at $(U''^*_{k'} \cap Z) \times \{0\}$ satisfy condition \eqref{condition:nonzero>nu} for $\nu \gg 1$, then there must be two distinct $l \neq m$, such that
\begin{align*}
\ord_B(f_{i_{k'}}^{d-e_l}\prod_{k=k'}^\rho h_k^{\alpha_{e_l, k}})
    &= \ord_B(f_{i_{k'}}^{d-e_m}\prod_{k=k'}^\rho h_k^{\alpha_{e_m, k}})
\end{align*}
and in addition, $\alpha_{e_{l,k'}} \neq \alpha_{e_{m,k'}}$. It is then straightforward to check that \Cref{claim:nonzero:++j:h_k} holds for $k = k'$ with $U''^*_{k'}$ in place of $U''^*_0$, and $\scrQ'$ equal to the union of $\scrQ'_{k'+1}$ (from the inductive hypothesis) and at most finitely many rational numbers.
\end{proof}
Assertions \eqref{nonzero:++j:h:++} and \eqref{nonzero:++j:h:+++} of \Cref{prop:nonzero:++j} now follows from applying \Cref{claim:integral-claim-1} (with $h_i$ in place of the $f_i$, the $u_i$ in place of the $\phi_i$, and $f = h$ and $s = j$), and following the arguments from the proof of \Cref{claim:nonzero:++j:h_k}.
\end{proof}

\subsubsection{} \label{cutting-out:nonempty}
We go back to the proof of \Cref{prop:cutting-out}. Due to \Cref{thm:Z-existence} we may assume \woutlog\ that $Z^{(\rho)}_{i_1, \ldots, i_\rho} \neq \emptyset$.

\subsubsection{} \label{cutting-out:j=1}
For the case of $j = 1$, pick an irreducible component $W$ of $V(I''^{(1)}_{i_1, \ldots, i_\rho})$. We will find $h_W \in \kk[X]$ as follows:
\begin{enumerate}
\item If $W \subseteq V(f_{i_2})$, then set $h_W := f_{i_2}$.
\item Otherwise there is $i_W$, $i_1 < i_W < i_2$, such that $f_{i_W}|_{W} \not\equiv 0$. Pick $h_W$ such that $h_W$ vanishes on $W$ but does not vanish identically on any irreducible component of $V(f_{i_W})$.
\end{enumerate}
We claim that assertion \eqref{cutting-out:ord} of \Cref{prop:cutting-out} holds for $j = 1$ if we take $h''_1 = h_W$. Indeed, this is clear in the first case when $h_W = f_{i_2}$, and in the second case this follows from applying the \eqref{nonzero:+++} version of \Cref{prop:nonzero:++j} with $\scrH = \{(f_{i_W}, i_2)\}$ and $h = h_W$. Take $h''_1 := \prod_W h_W$. Then it is easy to see that
\begin{enumerate}[resume]
\item \Cref{prop:cutting-out} holds for $j = 1$ with $h''_1$.
\end{enumerate}
We now choose $h'_1, \tilde h'_1 \in I'^{(1)}_{i_1, \ldots, i_\rho}$ as follows (these elements will help inductive construction of $h''_j$ for $j > 1$):
\begin{enumerate}[resume]
\item
\begin{enumerate}
\item $h'_1 \in I'^{(1)}_{i_1, \ldots, i_\rho}$,
\item $h'_1$ does not vanish identically on any irreducible component of $V(\prod_{j\geq 2}f_{i_j})$,
\item $h'_1$ does not vanish identically on any irreducible component of $V(I''^{(1)}_{i_1, \ldots, i_\rho})$.
\end{enumerate}
\item
\begin{enumerate}
\item $\tilde h'_1 \in I'^{(1)}_{i_1, \ldots, i_\rho}$,
\item $\tilde h'_1$ does not vanish identically on any irreducible component of $V(\prod_{j\geq 2}f_{i_j})$,
\item $\tilde h'_1$ does not vanish identically on any irreducible component of $V(h'_1)$ which is not an irreducible component of $V(I'^{(1)}_{i_1, \ldots, i_\rho})$,
\item $\tilde h'_1$ does not vanish identically on any irreducible component of $V(h'_1, f_{i_2})$ which is not contained in $V(I'^{(1)}_{i_1, \ldots, i_\rho})$.
\end{enumerate}
\end{enumerate}
Since $h'_1h''_1$ and $\tilde h'_1h''_1$ are elements of $\sqrt{I^{(1)}_{i_1, \ldots, i_\rho}} = \sqrt{\langle f_{i_1} \rangle}$, the arguments from the proof of \Cref{prop:nonzero:+j} imply that
\begin{enumerate}[resume]
\item Condition \eqref{nonzero:+++} holds with $\scrH^* = \{(h'_1, i_2)\}$ and also with $\scrH^* = \{(\tilde h'_1, i_2)\}$.
\end{enumerate}

\subsubsection{} \label{sec:cutting-out:ind}
By induction, assume there are $h''_1, \ldots, h''_k$, $1 \leq k \leq \rho - 2$, such that for each $j = 1, \ldots, k$,
\begin{enumerate}
\item $h''_j \in I''^{(j)}_{i_1, \ldots, i_\rho}$,
\item \Cref{prop:cutting-out} holds for $j = 1, \ldots, k$ (with $h''_1, \ldots, h''_j$)
\end{enumerate}
Assume in addition that there are $h'_1, \ldots, h'_k$, $1 \leq k \leq \rho - 2$, such that
\begin{enumerate}[resume]
\item \label{cutting-out:ind:I} $h'_j \in I'^{(j)}_{i_1, \ldots, i_\rho}$ for each $j$,
\item \label{cutting-out:ind:I''} $V(h'_1, \ldots, h'_k)$ does not contain any irreducible component of $V(I''^{(k)}_{i_1, \ldots, i_\rho})$.
\item \label{cutting-out:ind:codim:0} $V(h'_1, \ldots, h'_k)$ has pure codimension $k$,
\item \label{cutting-out:ind:codim:1} $V(h'_1, \ldots, h'_k, \prod_{j=k+1}^\rho f_{i_j})$ has pure codimension $k+1$,
\item \label{cutting-out:ind:+++} $\scrH^* = \{(h'_1, i_2), \ldots, (h'_k, i_{k+1})\}$ satisfies \eqref{nonzero:+++} with an appropriate $Z^*_0$.
\end{enumerate}
For technical reasons we also assume that there is $\tilde h'_k \in I'^{(k)}_{i_1, \ldots, i_\rho}$ such that
\begin{enumerate}[resume]
\item \label{cutting-out:ind:extra} each of the properties \eqref{cutting-out:ind:I} to \eqref{cutting-out:ind:+++} continues to be satisfied if $h'_k$ is replaced by $\tilde h'_k$,
\item \label{cutting-out:ind:extra:codim} $\tilde h'_k$ does not vanish identically on any irreducible component of $V(h'_1, \ldots, h'_k)$ which is not an irreducible component of $V(I'^{(j)}_{i_1, \ldots, i_\rho})$.
\item \label{cutting-out:ind:extra:codim:f} $\tilde h'_k$ does not vanish identically on any irreducible component of $V(h'_1, \ldots, h'_k, f_{i_{k+1}})$ which is not contained in $V(I'^{(k)}_{i_1, \ldots, i_\rho})$.
\end{enumerate}
We will construct $h'_{k+1}, \tilde h'_{k+1}, h''_{k+1}$ extending the above properties.

\subsubsection{} \label{sec:cutting-out:ind:k+1}
Let $h'_{k+1} \in I'^{(k+1)}_{i_1, \ldots, i_\rho}$ be such that
\begin{enumerate}
\item $h'_{k+1}$ does not vanish identically on any irreducible component of $V(h'_1, \ldots, h'_k)$,
\item $h'_{k+1}$ does not vanish identically on any irreducible component of $V(h'_1, \ldots, h'_k, \prod_{j=k+2}^\rho f_{i_j})$,
\item \label{cutting-out:ind:k+1:I''} $h'_{k+1}$ does not vanish identically on any irreducible component of $V(I''^{(k+1)}_{i_1, \ldots, i_\rho})$.
\end{enumerate}
Pick $\tilde h'_{k+1} \in I'^{(k+1)}_{i_1, \ldots, i_\rho}$ be such that
\begin{enumerate}[resume]
\item $\tilde h'_{k+1}$ satisfies each of the above properties of $h'_{k+1}$, and in addition,
\item $\tilde h'_{k+1}$ does not vanish identically on any irreducible component of $V(h'_1, \ldots, h'_{k+1})$ which is not an irreducible component of $V(I'^{(k+1)}_{i_1, \ldots, i_\rho})$.
\item $\tilde h'_{k+1}$ does not vanish identically on any irreducible component of $V(h'_1, \ldots, h'_{k+1}, f_{i_{k+2}})$ which is not contained in $V(I'^{(k+1)}_{i_1, \ldots, i_\rho})$.
\end{enumerate}

\begin{proclaim}\label{claim:cutting-out:ind}
Condition \eqref{nonzero:+++} holds with $\scrH^* = \{(h'_{k+1}, i_{k+2})\}$ and also with $\scrH^* = \{(\tilde h'_{k+1}, i_{k+2})\}$.
\end{proclaim}

\begin{proof}
We only show that $\scrH^* = \{(h'_{k+1}, i_{k+2})\}$ satisfies \eqref{nonzero:+++}, since the proof of the case of $\{(\tilde h'_{k+1}, i_{k+2})\}$ is completely analogous, and follows from running the arguments below after swapping $h'_{k+1}$ with $\tilde h'_{k+1}$ (and vice versa). \\

\paragraph{} \label{cutting-out:ind:condition}
Pick an irreducible component $W$ of $V(h'_1, \ldots, h'_k, f_{i_{k+1}})$ such that $h'_{k+1}|_W \not\equiv 0$ (which means in particular that $W \not \subseteq V(I'^{(k+1)}_{i_1, \ldots, i_\rho})$). We will construct regular functions $h_{W,j}$, $j = 1, \ldots, n_W$, which vanish identically on $W$, and satisfy the following condition:
\begin{align*}
\parbox{0.9\textwidth}{
there are real numbers $\mu^*_W, \nu^*_W$ and a nonempty open subset $Z^*_W$ of $Z$ such that if \eqref{condition:nonzero>nu} holds with $\nu \geq \nu^*_W$ and $\Center(B) \in Z^*_W$, then $\ord_B(h_{W,j}|_C) \leq \mu^*_W r_{i_{k+2}}\ord_B(t)$ for some $j = 1, \ldots, n_W$.
}
\end{align*}
There are three possible options for $W$:\\

\paragraph{}
First consider the case that $W \subseteq V(\prod_{l=k+2}^\rho f_{i_l})$. In this case take $n_W := 1$ and $h_{W,1} := \prod_{l=k+2}^\rho f_{i_l}$. \\

\paragraph{}\label{cutting-out:ind:iW}
The second possibility is that $W \not \subseteq V(\prod_{l=k+2}^\rho f_{i_l})$, and there is $i_W$, $i_{k+1} < i_W < i_{k+2}$, such that $f_{i_W}|_W \not\equiv 0$. In this case take $n_W := k + 1$, and inductively choose $h_{W,1}, \ldots, h_{W,k+1}$ such that
\begin{enumerate}
\item \label{cutting-out:ind:iW:1} $h_{W,1}$ vanishes identically on $W$ but does not vanish identically on any irreducible component of $V(f_{i_W} \prod_{l=k+2}^\rho f_{i_l})$,
\item \label{cutting-out:ind:iW:ind} for each $j = 1, \ldots, k$,
\begin{enumerate}
\item $h_{W,j+1}$ vanishes identically on $W$,
\item $h_{W, j+1}$ does not vanish identically on any irreducible component of $V(h_{W,1}, \ldots, h_{W,j})$,
\item $h_{W,j+1}$ does not vanish identically on any irreducible component of $V(h_{W,1}, \ldots, h_{W,j}, \allowbreak f_{i_W} \prod_{l=k+2}^\rho f_{i_l})$.
\end{enumerate}
\end{enumerate}
We claim that $h_{W,j}$, $j = 1, \ldots, k+1$, satisfy the condition from \Cref{cutting-out:ind:condition} above. Indeed, otherwise due to the assumption in \Cref{cutting-out:nonempty}, condition \eqref{nonzero:++} is satisfied with $\scrH^* = \{(h_{W,1}, i_{k+2}), \allowbreak \ldots, \allowbreak (h_{W,k+1}, i_{k+2})\}$. Consequently, the \eqref{nonzero:++} version of \Cref{prop:nonzero:++j} applies with $h = f_{i_W}$ and yields a contradiction. \\

\paragraph{}
In the remaining scenario, $W \not\subseteq V(\prod_{l=k+2}^\rho f_{i_l} I''^{(k+1)}_{i_1, \ldots, i_\rho})$. Then in particular, $W \not\subseteq V(I'^{(k+1)}_{i_1, \ldots, i_\rho}) \allowbreak \cup \allowbreak V(I''^{(k+1)}_{i_1, \ldots, i_\rho}) = V(I^{(k+1)}_{i_1, \ldots, i_\rho})$. Since $W \subseteq V(f_{i_{k+1}})$, it follows that $W \not\subseteq V(I'^{(k)}_{i_1, \ldots, i_\rho})$ (see the definition of these ideals in \Cref{sec:length:confirmation:I'I''}). Due to inductive property \eqref{cutting-out:ind:extra:codim:f} of \Cref{sec:cutting-out:ind} we may inductively pick $h_{W,1}, \ldots, h_{W,k+1}$ which vanish identically on $W$, and satisfy does both properties \eqref{cutting-out:ind:iW:1} and \eqref{cutting-out:ind:iW:ind} from \Cref{cutting-out:ind:iW} with $f_{i_W}$ replaced by $\tilde h'_k$. An application of the \eqref{nonzero:++} version of \Cref{prop:nonzero:++j} with $\scrH^* = \{(h_{W,1}, i_{k+2}), \allowbreak \ldots, \allowbreak (h_{W,k+1}, i_{k+2})\}$ then shows that $h_{W,1}, \ldots, h_{W, k+1}$ satisfy the condition from \Cref{cutting-out:ind:condition} above with $n_W = k + 1$. \\

\paragraph{}
For each tuple $\sigma = (\sigma_W)_W$ with $\sigma_W \in \{1, \ldots, n_W\}$, define
\begin{align*}
h_\sigma &:= \prod_W h_{W,\sigma_W}
\end{align*}
If \eqref{condition:nonzero>nu} holds with $\nu \geq \max_W \nu^*_W$ and $\Center(B) \in \bigcap_W Z^*_W$, then there is $\sigma$ such that $\ord_B(h_\sigma|_C) \leq \sum_W\mu^*_W r_{i_{k+2}}\ord_B(t)$. Since $h'_{k+1}h_\sigma \in \sqrt{\langle h'_1, \ldots, h'_k, f_{i_{k+1}} \rangle}$ for each $\sigma$, it follows from the inductive property \eqref{cutting-out:ind:+++} of \Cref{sec:cutting-out:ind} and the arguments from the proof of \Cref{prop:nonzero:+j} that \eqref{nonzero:+++} holds with $\scrH^* = \{(h'_{k+1}, i_{k+2})\}$, as required.
\end{proof}

\subsubsection{}
Now pick $h''_{k+1} \in I''^{(k+1)}_{i_1, \ldots, i_\rho}$ which does not vanish identically on any irreducible component of $V(h'_1, \ldots, h'_{k+1})$ (this is possible due to property \eqref{cutting-out:ind:k+1:I''} from \Cref{sec:cutting-out:ind:k+1}). Then due to the inductive hypothesis and \Cref{claim:cutting-out:ind}, the \eqref{nonzero:+++} version of \Cref{prop:nonzero:++j} can be applied with $\scrH^* = \{(h'_1, i_2), \ldots, (h'_{k+1}, i_{k+2})\}$ and $h = h''_{k+1}$, and implies that \Cref{prop:cutting-out} holds for $j = 1, \ldots, k+1$ with $h''_1, \ldots, h''_{k+1}$. This completes the inductive step and concludes the proof of \Cref{prop:cutting-out}.
\end{proof}

\begin{claim}[\Cref{length:confirmation:I-g-subset}] \label{proof:length:confirmation:I-g-subset:copy}
There is a nonempty open subset $Z^*_0$ of $Z$ and $\nu^*_0 \in \rr$ such that if $z \in Z^*_0$ and $g_1, \ldots, g_n, r_1, \ldots, r_n$ satisfy \eqref{condition:nonzero>nu} with $\nu \geq \nu^*_0$, then
\begin{enumerate}
\item \label{length:confirmation:I-g-subset:copy:Ijt} for each $j$, $1 \leq j \leq \rho$,
\begin{align*}
I^{(j)}_{i_1, \ldots, i_\rho} \hatlocal{C}{(z,0)} \subseteq t^{r_{i_j}}\hatlocal{C}{(z,0)}
\end{align*}
In addition, $Z^*_0$ does {\em not} change (but $\nu^*_0$ possibly changes) if the $f_i$ are replaced by $(f_i)^{m_i}$ for $m_i \geq 1$.
\end{enumerate}
Pick $i' \in \{1, \ldots, n\}\setminus \{i_1, \ldots, i_\rho\}$. Recall that $i_1 = i^*$, where $i^*$ is the smallest index such that $F_{i^*}$ is not identically zero.
\begin{enumerate}[resume]
\item \label{length:confirmation:I-g-subset:copy:g:zero} If $i' < i^*$, then $g_{i'}|_C \equiv 0 \in t\hatlocal{C}{(z,0)}$.
\item \label{length:confirmation:I-g-subset:copy:g:t} Assume $i' > i^*$. Pick the largest $j$ such that $i_j < i'$. Then $f_{i'} \in \sqrt{I'^{(j)}_{i_1, \ldots, i_\rho}}$. If $f_{i'} \in I'^{(j)}_{i_1, \ldots, i_\rho}$, then $g_{i'} \in t\hatlocal{C}{(z,0)}$.
\end{enumerate}
\end{claim}

\begin{proof}
We start with the proof of assertion \eqref{length:confirmation:I-g-subset:copy:Ijt}. Fix $\phi_j \in I^{(j)}_{i_1, \ldots, i_\rho}\local{C}{(z,0)}$. Then $\phi_j = \phi'_{j-1} + f_{i_j}\psi_j$ for some $\phi'_{j-1} \in I'^{(j-1)}_{i_1, \ldots, i_\rho}\local{C}{(z,0)}$ and $\psi_j \in \local{C}{(z,0)}$. However, then
\begin{align*}
h''_{j-1}\phi_j &= \phi_{j-1} + h''_{j-1}f_{i_j}\psi_j
\end{align*}
where $h''_{j-1}$ is from \Cref{prop:cutting-out} and $\phi_{j-1} := h''_{j-1}\phi'_{j-1} \in I'^{(j-1)}_{i_1, \ldots, i_\rho}I''^{(j-1)}_{i_1, \ldots, i_\rho} \subseteq I^{(j-1)}_{i_1, \ldots, i_\rho}$. \Cref{prop:cutting-out} implies that for each branch $B$ of $C$ at $(z,0)$, there are integers $m_{j-1,B}, n_{j-1,B}$ such that
\begin{align*}
\ord_B(h''_{j-1}) = \frac{m_{j-1,B}}{n_{j-1,B}}r_{i_j}\ord_B(t)
\end{align*}

\begin{proclaim}
There are invertible elements $\xi_{j-1,B, k} \in \hatlocal{C}{(z,0)}$ and an integer $m'_{j-1}$ such that
\begin{align*}
(t^{r_{i_j}m_{j-1,B}N_{j-1,B}} + \sum_{k=1}^{N_{j-1,B}} \xi_{j-1, B, k} t^{r_{i_j}m_{j-1,B}(N_{j-1,B}-k)} (h''_{j-1})^{kn_{j-1, B}})^{m'_{j-1}}
    &= 0 \in \hatlocal{C}{(z,0)}
\end{align*}
\end{proclaim}

\begin{proof}
Let $C'$ be the (unique) reduced curve with support equal to $\supp(C)$. Let $\pi: \tilde C \to C'$ be a desingularization of $C'$, and $\tilde z_B$ be the center of $B$ on $\tilde C$. Since $\hatlocal{\tilde C}{\tilde z_B}$ is integral over $\pi^*(\hatlocal{C'}{(z,0)})$, we may take an integral equation of $t^{m_{j-1,B}r_{i_j}}/(h''_{j-1})^{n_{j-1, B}}$ over $\pi^*(\hatlocal{C'}{(z,0)})$ and clear out the denominators to obtain an identity of the following form in $\hatlocal{\tilde C}{\tilde z_B}$:
\begin{align*}
t^{r_{i_j}m_{j-1,B}N_{j-1,B}} + \sum_{k=1}^{N_{j-1,B}} \pi^*(\xi_{j-1, B, k}) t^{r_{i_j}m_{j-1,B}(N_{j-1,B}-k)} (h''_{j-1})^{kn_{j-1, B}}
    &= 0
\end{align*}
Due to exactness of completions (i.e.\ the injectivity of the map $\hatlocal{C'}{(z,0)} \to \prod_B \hatlocal{\tilde C}{\tilde z_B}$), it then follows that the product over all $B$ of the left hand side of the above equation is zero in $\hatlocal{C'}{(z,0)}$. Since the kernel of the natural map from $\hatlocal{C}{(z,0)} \to \hatlocal{C'}{(z,0)}$ is nilpotent, it becomes zero in $\hatlocal{C}{(z,0)}$ after raising to a power, as required to prove the claim.
\end{proof}

\begin{proclaim}
The total number of branches $B$, and all $m_{j-1,B}, n_{j-1,B}$, and $m'_{j-1}$ can be bounded above by an integer which is independent of $\nu^*_0$ and $B$.
\end{proclaim}

\begin{proof}
The total number of branches $B$ is clearly  $\leq \mult{F_1}{F_n}$. \Cref{prop:cutting-out} implies that each $m_{j-1,B}, n_{j-1,B}$ can be chosen to be bounded above by a quantity independent of $\nu^*_0$ and $B$. On the other hand, \Cref{lemma:nilorder} implies that $m'_{j-1}$ is bounded by above $\ord_B(t)$. Since $\ord_B(t) \leq \mult{F_1}{F_n}$, this completes the proof.
\end{proof}
%
%


\subsubsection{} \label{length:confirmation:Isubset:ind1}
It follows that for each $j = 2, \ldots, \rho$, there is an identity of the form
\begin{align*}
t^{Nr_{i_j}} - \sum_{k=1}^{N} \xi_{j-1,k} t^{(N-k)r_{i_j}}(h''_{j-1})^{\alpha_{j-1,k}}
    \equiv 0 \in \hatlocal{C}{(z,0)}
\end{align*}
where $N$ is independent of $\nu^*_0$ and $B$, and each $\alpha_{j-1,k}$ is {\em positive}. Then in $\hatlocal{C}{(z,0)}$ one has
\begin{align*}
t^{Nr_{i_j}} \phi_j
    &= t^{Nr_{i_j}} (\phi'_{j-1} + f_{i_j}\psi_j) \\
    &= \sum_{k=1}^{N} \xi_{j-1,k} t^{(N-k)r_{i_j}}(h''_{j-1})^{\alpha_{j-1,k}}
        \phi'_{j-1}
        + t^{Nr_{i_j}} f_{i_j}\psi_j \\
    &= \chi_{j-1} \phi_{j-1}
        + t^{Nr_{i_j}+r_{i_j}} g_{i_j}\psi_j
\end{align*}
where
\begin{align*}
\chi_{j-1}
    &:= \sum_{k=1}^{N} \xi_{j-1,k} t^{(N-k)r_{i_j}}(h''_{j-1})^{\alpha_{j-1,k}-1}
    \in \hatlocal{C}{(z,0)},\ \text{and} \\
\phi_{j-1}
    &:= h''_{j-1}\phi'_{j-1} \in I^{(j-1)}_{i_1, \ldots, i_\rho}
\end{align*}

\subsubsection{} \label{length:confirmation:Isubset:indall}
Repeating the arguments of the preceding paragraph for $j-1, j-2, \ldots, 2$ yields, with some $N' \geq N$, that
\begin{align*}
t^{N' \sum_{k=2}^j r_{i_k}} \phi_j
    &= t^{N'(\sum_{k=2}^j r_{i_k})+r_{i_j}} g_{i_j}\psi_j
        + \chi_{j-1} t^{N'(\sum_{k=2}^{j-1} r_{i_k})+r_{i_{j-1}}} g_{i_{j-1}}\psi_{j-1} \\
    &\quad \qquad
       + \cdots
        + (\prod_{k=2}^{j-1} \chi_k)t^{N'r_{i_2}+r_{i_2}}g_{i_2}\psi_2
        + (\prod_{k=1}^{j-1} \chi_k) \phi_1 \\
    &= t^{N'(\sum_{k=2}^j r_{i_k})+r_{i_j}} g_{i_j}\psi_j
        + \chi_{j-1} t^{N'(\sum_{k=2}^{j-1} r_{i_k})+r_{i_{j-1}}} g_{i_{j-1}}\psi_{j-1} \\
    &\quad \qquad
       + \cdots
        + (\prod_{k=2}^{j-1} \chi_k)t^{N'r_{i_2}+r_{i_2}}g_{i_2}\psi_2
        + (\prod_{k=1}^{j-1} \chi_k)t^{r_{i_1}}g_{i_1}\psi_1
\end{align*}
where $\psi_1, \ldots, \psi_\rho$ are elements of $\local{C}{(z,0)}$ which depend on $\phi_\rho$. Now if $r_1, \ldots, r_n$ satisfy \eqref{condition:>nu} for $\nu > N'+1$, then
\begin{align*}
r_{i_1} > N'r_{i_2} + r_{i_2} > N'(r_{i_2} + r_{i_3}) + r_{i_3} > \cdots > N'(\sum_{j=2}^\rho r_{i_j})+r_{i_\rho}
\end{align*}
so that
\begin{align*}
t^{N'\sum_{k=2}^j r_{i_k}}\phi_j = t^{N'(\sum_{k=2}^j r_{i_k})+r_{i_j}}\phi'_j
\end{align*}
for some $\phi'_j \in \hatlocal{C}{(z,0)}$. Then $t^{N'\sum_{k=2}^j r_{i_k}}(\phi_j - t^{r_{i_j}}\phi'_j) = 0$. Since $t$ is a non zero-divisor in $\hatlocal{C}{(z,0)}$, it follows that $\phi_j = t^{r_{j}}\phi'_j$. Consequently, $I^{(j)}_{i_1, \ldots, i_\rho}\hatlocal{C}{(z,0)} \subseteq t^{r_{i_j}}\hatlocal{C}{(z,0)}$, which completes the proof of assertion \eqref{length:confirmation:I-g-subset:copy:Ijt} of \Cref{proof:length:confirmation:I-g-subset:copy}. 

\subsubsection{}
Assertion \eqref{length:confirmation:I-g-subset:copy:g:zero} of \Cref{proof:length:confirmation:I-g-subset:copy} follows directly from the definition of $i^*$. For assertion  \eqref{length:confirmation:I-g-subset:copy:g:t}, assume $i'_k > i^*$. Pick the largest $j$ such that $i_j < i'_k$. Then $g_{i'_k}t^{r_{i'_k}} = f_{i'_k} \in \sqrt{I'^{(j)}_{i_1, \ldots, i_\rho}}$. Assume now that $f_{i'_k} \in I'^{(j)}_{i_1, \ldots, i_\rho}$. Then we claim that $g_{i'_k} \in t\hatlocal{\tilde C}{\tilde z_B}$. Indeed, it follows from the arguments in \Cref{length:confirmation:Isubset:ind1,length:confirmation:Isubset:indall} that
\begin{align*}
t^{r_{i'_k} + N'r_{i_{j+1}}} g_{i'_k}
    &= t^{N'r_{i_{j+1}}} f_{i'_k}
    = \chi_j\phi_j
\end{align*}
with $\phi_j \in I$, and in turn,
\begin{align*}
t^{N'\sum_{j'=2}^j r_{i_{j'}} + r_{i'_k} + N'r_{i_{j+1}}} g_{i'_k}
    &= t^{N'\sum_{j'=2}^j r_{i_{j'} + r_{i_j}}} \phi'_j
\end{align*}
for some $\phi'_j \in \hatlocal{C}{(z,0)}$. Then under \eqref{condition:>>}, $r_{i_j} > r_{i'_k} + N'r_{i_{j+1}}$ and consequently, $g_{i'_k} \in \hatlocal{\tilde C}{\tilde z_B}$, as required.

\end{proof}

\bibliographystyle{alpha}
\bibliography{../../../../utilities/bibi}

\end{document}